\documentclass[11pt]{amsart}
\usepackage{macrosAMS}
\usepackage[margin=1.25in]{geometry}

\begin{document}

\title{Stable Harmonic Analysis and Stable Transfer}
\author{Matthew Sunohara}
\address{Department of Mathematics\\Johns Hopkins University\\3400 N Charles St\\Baltimore\\MD 21218}



\begin{abstract}
    Langlands posed the question of whether a local functorial transfer map of stable tempered characters can be interpolated by the transpose of a linear operator between spaces of stable orbital integrals of test functions. These so-called stable transfer operators are intended to serve as the main local ingredient in Beyond Endoscopy, his proposed strategy for proving the Principle of Functoriality. Working over a local field of characteristic zero and assuming a hypothesis on the local Langlands correspondence for $p$-adic groups, we prove the existence of continuous stable transfer operators between spaces of stable orbital integrals of Harish-Chandra Schwartz functions, test functions, and $K$-finite test functions. This is achieved via stable Paley--Wiener theorems for each of the three types of function spaces. The stable Paley--Wiener theorem for Harish-Chandra Schwartz functions is new and includes the result that stable tempered characters span a weak-$*$ dense subspace of the space of stable tempered distributions, a result previously unknown for $p$-adic groups.
\end{abstract}

\maketitle

\tableofcontents

\section{Introduction} \label{sec:intro}
Beyond Endoscopy is the strategy proposed by Langlands for proving his Principle of Functoriality in general \cite{LanEB} (see also \cite{LanBE, LanNouveau}). The strategy involves developing new refinements of the stable trace formula and comparisons between them using stable transfer operators. These operators, which form the main local ingredient of Beyond Endoscopy and are our focus here, were introduced by Langlands in \cite{LanBE, LanST}. We refer to the articles \cite{ArthurProblemsBE, SakIHES} by Arthur and Sakellaridis for an overview of the strategy of Beyond Endoscopy. In \cite{ArthurProblemsBE}, Arthur explains the formal analogy between the role of stable transfer operators in Beyond Endoscopy and the role of endoscopic transfer operators in the theory of Endoscopy.

Throughout this paper, we fix a local field $F$ of characteristic zero and fix an algebraic closure $\ol{F}$ of $F$. Let $W_F$ be the Weil group of $F$ and let $L_F$ be the complex version of the local Langlands group of $F$, that is, $L_F=W_F$ if $F$ is archimedean and $L_F=\SL_2(\CC)\times W_F$ if $F$ is non-archimedean.

Let $H$ and $G$ be connected reductive groups over $F$ and let $\xi:\Lgp{H}\to\Lgp{G}$ be an injective $L$-homomorphism (taken up to $\dual{G}$-conjugacy). Throughout this paper, we use the Weil form of the $L$-group. We assume that $G$ is quasisplit, so that we have a well-defined pushforward map 
\[
\xi_*:\Phi(H)\longrightarrow\Phi(G)
\]
of $L$-parameters, called functorial transfer. Furthermore, we assume that $\xi$ is tempered in the sense that $\xi|_{W_F}$ is a tempered $L$-parameter of $G$, so that $\xi_*$ restricts to a map $\xi_*:\Phi_\temp(H)\to\Phi_\temp(G)$ of tempered $L$-parameters.

Let $\Pi(H)$ denote the set of equivalence classes of irreducible admissible representations of $H(F)$. A local Langlands correspondence for $H$ is a surjective ``reciprocity'' map
\[
\rec_H:\Pi(H)\longrightarrow\Phi(H)
\]
with finite fibres satisfying various desiderata \cite{Borel, KalTai}. The fibre $\Pi_\phi$ over $\phi\in\Phi(H)$ is called the $L$-packet of $\phi$. A local Langlands correspondence is known and uniquely characterised for real groups \cite{LanReal, AdamsVogan, AdamsKaletha}. A local Langlands correspondence is  known in various forms for $p$-adic groups \cite{HarTay, Henniart, Scholze, ArthurEndo, MokEndo, MoegEndo}.

Beyond Endoscopy is intended to be built upon the theory of Endoscopy and the local Langlands correspondence. Thus, we will assume the existence of a local Langlands correspondence for $p$-adic groups satisfying a subset of the usual desiderata. See \Cref{hypothesis} for the precise form of our hypothesis on the local Langlands correspondence for $p$-adic groups. Part of the local Langlands correspondence is the assignment to each $\phi\in\Phi_\temp(H)$ of a stable tempered character $\Theta_\phi$. The distribution $\Theta_\phi$ is a positive integral linear combination of the tempered characters $\Theta_\pi$ for $\pi\in\Pi_\phi$. The functorial transfer $\xi_*:\Phi_\temp(H)\to\Phi_\temp(G)$ of tempered $L$-parameters thus determines a functorial transfer map $\xi_*:\Theta_\phi\mapsto\Theta_{\xi_*\phi}$ from the set of stable tempered characters of $H$ to the set of stable tempered characters of $G$. Langlands posed the following question \cite[Questions A \& B]{LanST}.

\begin{question}\label{question}
    For each $f\in C_c^\infty(G(F))$, does there exist a function $\Tc_\xi f\in C_c^\infty(H(F))$, called a \emph{stable transfer of $f$}, such that 
    \[
    \langle \Theta_\phi,\Tc_\xi f\rangle=\langle \xi_*\Theta_\phi,f\rangle
    \]
    for all $\phi\in\Phi_\temp(H)$?
\end{question}

In \cite{LanST}, Langlands gave an affirmative answer to \Cref{question} in the case when $G=\SL_2$, $H$ is a maximal torus of $G$, and $\xi$ is a natural $L$-embedding \cite[\S2.1--\S2.4 (p. 182--210)]{LanST}. This case had previously been studied by Gelfand--Graev in \cite{GelfandGraev}, where an explicit formula for the operator $\Tc_\xi$ is given. The Gelfand--Graev formula was also established in \cite[\S10.5]{SakTransferHankel2}. Johnstone gave an affirmative answer and obtained a formula for $\Tc_\xi$ in the case when $G=\SL_\ell$ or $G=\GL_\ell$ with $\ell$ an odd prime, $H$ is an elliptic maximal torus of $G$, and $\xi$ is a natural $L$-embedding \cite{JohnstoneThesis,JohnstoneArticle}. The arguments in these works involve detailed computations with explicit stable character formulas, which are not available in general. The existence of stable transfers of test functions has also been explored in \cite{JohnstoneLuo,thomas2020towards}. Sakellaridis has studied stable transfer for relative functoriality, and in certain low rank cases has proved that they exist and obtained explicit formulas for them \cite{SakBEI,SakBEII,SakFuncViaTrace,SakTransferHankel1,SakTransferHankel2}.

The existence of stable transfers is useful in global applications of the stable trace formula that lie outside the theory of Endoscopy but do not belong to the program of Beyond Endoscopy. In \cite{DalalGG} stable transfers are used to prove statistics of automorphic representations for unramified unitary groups, with applications to the Sato--Tate Conjecture, the Sarnak--Xue Conjecture, and the cohomology of locally symmetric spaces. Extensions of this work to other unitary groups and other classical groups require a better understanding of the properties of stable transfer. We expect that stable transfer may have other applications outside of Beyond Endoscopy, both in global and local harmonic analysis.

In order to give a more natural formulation of \Cref{question}, we recall the basic notions of stable harmonic analysis. For brevity, we write $C_c^\infty(H)=C_c^\infty(H(F))$ and we write $\Cc(H)=\Cc(H(F))$ for the space of Harish-Chandra Schwartz functions, both equipped with their usual locally convex topologies, which makes the inclusion $C_c^\infty(H)\to \Cc(H)$ continuous with dense image. We recall that a distribution on $H(F)$, i.e. an element of the strong dual $C_c^\infty(H)'$, is tempered if it extends continuously to an element of $\Cc(H)'$. Since stable tempered characters are tempered distributions, it is thus natural to consider the analogue of \Cref{question} where $C_c^\infty$ is replaced by $\Cc$.

Let $f\in \Cc(H)$. For a regular semisimple stable conjugacy class $\delta\in\Delta_\rs(H)$ in $H(F)$, we denote the normalised stable orbital integral of $f$ at $\delta$ by $f^H(\delta)=|D^H(\delta)|^{1/2}O_\delta(f)$, where $D^H(\delta)$ is the Weyl discriminant. We denote the resulting quotient space of stable orbital integrals by 
\[
\Sc(H)=\{f^H:\Delta_\rs(H)\to\CC : f\in\Cc(H)\}.
\]
A tempered distribution $u\in\Cc(H)'$ is said to be stable if it descends to an element of $\Sc(H)'$.

Stable tempered characters are stable tempered distributions. For $\phi\in\Phi_\temp(H)$, we write $\wh{f^H}(\phi)=\Theta_\phi(f)$. The resulting function $\wh{f^H}$ is the stable Fourier transform of $f$. In the literature, $\wh{f^H}(\phi)$ is often denoted by $f^H(\phi)$, and thus $f^H$ is used to denote both the stable orbital integrals and the stable Fourier transform of $f$. We will also follow this practice below. We denote the resulting quotient space of stable Fourier transforms by 
\[
\wh{\Sc}(H)=\{\wh{f^H}:\Phi_\temp(H)\to\CC : f\in\Cc(H)\}.
\]
The stable Fourier transform of $f$ only depends on its stable orbital integrals $f^H(\delta)$ for $\delta\in\Delta_\rs(H)$ (in fact, even for $\delta\in\Delta_\sr(H)$), so the map $\Cc(H)\to\wh{\Sc}(H), f\mapsto\wh{f^H}$ descends to a surjective continuous linear operator
\[
\Fc^\st:\Sc(H)\longrightarrow\wh{\Sc}(H).
\]
By replacing the Harish-Chandra Schwartz space $\Cc(H)$ with the space of test functions $C_c^\infty(H)$, we obtain spaces $\Sc_c(H)$ and $\wh{\Sc_c}(H)$, which are the test function analogues of $\Sc(H)$ and $\wh{\Sc}(H)$. The stable Fourier transform restricts to a continuous surjective linear operator $\Fc^\st:\Sc_c(H)\to\wh{\Sc_c}(H)$. For a maximal compact subgroup $K$ of $H(F)$, we have analogous spaces $\Sc_f(H)$ and $\wh{\Sc_f}(H)$ defined in terms of the space $C_c^\infty(H,K)$ of $K$-finite test functions, and the Fourier transform restricts further to a continuous surjective linear operator $\Fc^\st:\Sc_f(H)\to\wh{\Sc_f}(H)$.

Stable spectral density for test functions tells us that the stable tempered characters of $H$ span a weak-$*$ dense subspace of the space of stable distributions on $H(F)$. Equivalently, for $f\in C_c^\infty(H)$, the function $f^H$ is determined by $\wh{f^H}$, that is, we have a continuous linear isomorphism
\[
\Fc^\st:\Sc_c(H)\longrightarrow\wh{\Sc_c}(H).
\]
It follows that an equivalent formulation of \Cref{question} is whether there exists a linear operator 
\[
\Tc_\xi:\Sc_c(G)\longrightarrow\Sc_c(H)
\]
such that
\[
\Theta_\phi\circ\Tc_\xi=\Theta_{\xi_*\phi}
\]
for all $\phi\in\Phi_\temp(H)$.

Our main result is as follows.

\begin{theorem} \label{mainthm}
    If $F$ is non-archimedean, we assume \Cref{hypothesis} for $H$ and $G$. There exists a continuous linear operator
    \[
    \Tc_\xi:\Sc(G)\longrightarrow\Sc(H)
    \]
    whose transpose
    \[
    \Tc_\xi':\Sc(H)'\longrightarrow\Sc(G)'
    \]
    satisfies
    \[
    \Tc_\xi'\Theta_\phi=\Theta_{\xi_*\phi}
    \]
    for all $\phi\in\Phi_\temp(H)$. Moreover, $\Tc_\xi$ restricts to continuous linear operators $\Sc_c(G)\to\Sc_c(H)$ and $\Sc_f(G)\to\Sc_f(H)$.
\end{theorem}

This provides an affirmative answer to \Cref{question} in general (conditional on \Cref{hypothesis} in the $p$-adic case). We call the operator $\Tc_\xi$ the stable transfer operator attached to $\xi$. 

Stable transfer should not be confused with endoscopic transfer, which maps from the space $\Ic(G)$ of invariant orbital integrals of Harish-Chandra Schwartz functions to the space $\Sc(H)$, where $H$ is an endoscopic group of $G$. Stable transfer has also been called functorial transfer \cite{ArthurProblemsReal} and stable-stable transfer \cite{thomas2020towards} in order to avoid confusion with endoscopic transfer.

With a view towards applying stable transfer to the stable trace formula, an important problem remains of understanding how the stable orbital integrals $(\Tc_\xi f^G)(\delta_H)$ for $\delta_H\in\Delta_\rs(H)$ relate to the stable orbital integrals $f^G(\delta_G)$ for $\delta_G\in\Delta_\rs(G)$. Spectrally, the relation between $\Tc_\xi f^G$ and $f^G$ is very simple. The stable Fourier transform of $\Tc_\xi f^G$ is the pullback of the stable Fourier transform of $f^G$ along the functorial transfer map $\xi_*$:
\[
\Fc^\st(\Tc_\xi f^G)=\Fc^\st(f^G)\circ\xi_*.
\]
Comparing this with the Fourier-restriction formula for the classical Radon transform suggests that stable transfer operators are the analogues in stable harmonic analysis of Radon transforms in classical harmonic analysis.

There is a natural Schwartz space $\Ss^\st(H)$ on $\Phi_\temp(H)$, and similarly there are natural Paley--Wiener spaces $PW^\st(H)$ and $PW_f^\st(H)$ on $\Phi_\temp(H)$. These will be defined in \Cref{sec:stablePW}. We prove the following result, which contains stable Paley--Wiener theorems for Harish-Chandra Schwartz functions, test functions, and $K$-finite test functions.

\begin{theorem} \label{stablePW}
    If $F$ is non-archimedean, we assume \Cref{hypothesis} for $H$. The stable Fourier transform is an isomorphism of topological vector spaces
    \[
    \Fc^\st:\Sc(H)\longrightarrow\Ss^\st(H)
    \]
    and restricts to isomorphisms of topological vector spaces $\Sc_c(H)\to PW^\st(H)$ and $\Sc_f(H)\longrightarrow PW_f^\st(H)$.
\end{theorem}

The stable Paley--Wiener theorem for Harish-Chandra Schwartz functions is new. In the case where $H$ is quasisplit, Arthur proved the stable Paley--Wiener theorem for test functions on $p$-adic groups in \cite{ArthurRelations}, while Moeglin--Waldspurger proved the stable Paley--Wiener theorems for test functions and $K$-finite test functions on real groups in \cite[Ch. IV]{MWI}. (See the paragraph below \Cref{thm:stablePW} for further discussion on these results and their relation to this work.) The injectivity of the stable Fourier transform on $\Sc(H)$ is equivalent to stable spectral density for Harish-Chandra Schwartz functions. This result is new for $p$-adic groups; for real groups, it was established by Shelstad \cite[Lemma 5.3]{She79}.

Granting \Cref{stablePW}, to prove \Cref{mainthm} it then suffices to prove that pullback along $\xi_*$ gives a well-defined continuous linear operator
\[
\xi^*:\Ss^\st(G)\longrightarrow\Sc^\st(H)
\]
which restricts to continuous linear operators $PW^\st(G)\to PW^\st(H)$ and $PW_f^\st(G)\to PW_f^\st(H)$. Indeed, we may then simply define $\Tc_\xi$ to make the following diagram commute
\[
\begin{tikzcd}
    \Sc(G) \arrow[r,"\Tc_\xi"] \arrow[d,"\Fc^\st"] & \Sc(H) \arrow[d,"\Fc^\st"] \\
    \Ss^\st(G) \arrow[r,"\xi^*"] & \Ss^\st(H)
\end{tikzcd}
\]
and then $\Tc_\xi$ has the properties claimed. The proof that pullback along $\xi_*$ gives the well defined maps $\xi^*$ makes use of the fact the $\xi_*$ is $\pi_0$-finite, that is, the fibres of the map
\[
\pi_0(\xi_*):\pi_0(\Phi(H))\longrightarrow\pi_0(\Phi(G))
\]
are finite. We prove this as \Cref{thm:finiteness}.

In \Cref{sec:notation}, we set up the basic notation and conventions for the remainder of the paper. In \Cref{sec:invariant}, we recall the results from invariant harmonic analysis that we will use: the theory of $R$-groups, Arthur's virtual tempered representations, and invariant Paley--Wiener theorems. In \Cref{sec:spaces}, we define abstract Schwartz and Paley--Wiener spaces in a form suitable for invariant and stable Paley--Wiener theorems. In \Cref{sec:stable}, we recall the fundamentals of stable harmonic analysis, including our hypothesis on the local Langlands correspondence for $p$-adic groups (\Cref{hypothesis}), and prove the stable Paley--Wiener theorems contained in \Cref{stablePW}. In \Cref{sec:transfer}, we prove \Cref{mainthm} and give some simple examples of stable transfer operators.

\bigskip
\noindent\textbf{Acknowledgments.}
The author is grateful to James Arthur, William Casselman, Hannah Constantin, Clifton Cunningham, Rahul Dalal, Melissa Emory, Malors Espinosa, Mathilde Gerbelli-Gauthier, Jayce Getz, Florian Herzig, Daniel Johnstone, Stephen Kudla, Paul Mezo, Patrice Moisan-Roy, Chung Pang Mok, Yiannis Sakellaridis, Jacob Tsimerman, and Bin Xu for valuable discussions and comments on this work.

We acknowledge the support of the Natural Sciences and Engineering Research Council of Canada (NSERC). (Nous remercions le Conseil de recherches en sciences naturelles et en g\'enie du Canada (CRSNG) de son soutien.)

\section{Notation and conventions} \label{sec:notation}

Recall that we have fixed a local field $F$ of characteristic zero and an algebraic closure $\ol{F}$ of $F$. If $F$ is non-archimedean, we let $q_F$ denote the cardinality of the residue field of $F$. Let $|\cdot|_F$ denote the canonical absolute value of $F$, which is defined for all $a\in F$ by $\dd{(ax)}=|a|_F\dd{x}$, for any Haar measure $\dd{x}$ on $F$. We also denote its canonical extension to an absolute value on $\ol{F}$ by $|\cdot|_F$.

Let $G$ be a connected reductive group over $F$. We will say that $G$ is a real group if $F$ is archimedean and a $p$-adic group if $F$ is a $p$-adic field. We will use $X^*(G)$ (resp. $X_*(G)$) to denote the group of algebraic characters $G\to\GG_m$ (resp. cocharacters $\GG_m\to G$) over $F$. Thus, the group of algebraic characters (resp. cocharacters) of the base change $G_{\ol{F}}$ will be denoted by $X^*(G_{\ol{F}})$ (resp. $X_*(G_{\ol{F}})$).

We have the real vector spaces
\[
\ak_G=\Hom_\ZZ(X^*(G),\RR)
\]
and
\[
\ak_G^*= X^*(G)\otimes_\ZZ\RR.
\]
Note that we have a natural perfect pairing between $\ak_G$ and $\ak_G^*$, realising $\ak_G^*$ as the dual space of $\ak_G$. Let $A_G$ be the split component of the centre of $G$. The restriction gives an injective homomorphism $X^*(G)\to X^*(A_G)$ with finite cokernel. By functoriality, we obtain isomorphisms $\ak_{A_G}\xrightarrow{\sim}\ak_G$ and $\ak_G^*\xrightarrow{\sim}\ak_{A_G}^*$. In this way we identify $\ak_{A_G}=\ak_G$ and $\ak_{A_G}^*=\ak_G^*$.

We have the Harish-Chandra logarithm homomorphism $H_G:G(F)\to\ak_G$ defined by 
\[
\langle H_G(x),\chi\rangle=\log|\chi(x)|_F
\]
for all $x\in G(F)$ and $\chi\in X^*(G)$. Let
\[
G(F)^1=\ker H_G=\bigcap_{\chi\in X^*(G)}\ker|\chi|_F.
\]
Note that $H_{A_G}$ is the restriction of $H_G$, and therefore $A_G(F)^1=A_G(F)\cap G(F)^1$. Let 
\[
\ak_{G,F}=H_G(G(F))=G(F)/G(F)^1
\]
and 
\[
\wt{\ak}_{G,F}=\ak_{A_G,F}=H_G(A_G(F))=A_G(F)/A_G(F)^1.
\]
We define
\[
\dual{\ak_{G,F}}=\Hom_\ZZ(\ak_{G,F},2\pi i\ZZ)
\]
and similarly
\[
\dual{\wt{\ak}_{G,F}}=\Hom_\ZZ(\wt{\ak}_{G,F},2\pi i\ZZ).
\]
We have the inclusions
\[
\wt{\ak}_{G,F}\subseteq\ak_{G,F}\subseteq\ak_G.
\]
If $F$ is archimedean, then $\wt{\ak}_{G,F}=\ak_{G,F}=\ak_G$ and $\dual{\wt{\ak}_{G,F}}=\dual{\ak_{G,F}}=0$. If $F$ is non-archimedean, then $\ak_{G,F}$ and $\wt{\ak}_{G,F}$ are (full) lattices in $\ak_G$, and $\dual{\ak_{G,F}}\subseteq\dual{\wt{\ak}_{G,F}}$ are lattices in $i\ak_{G}^*$.

We denote the group of unramified characters of $G(F)$ by 
\[
X^\nr(G)=\Hom(G(F)/G(F)^1,\CC^\times)
\]
and the subgroup of unitary unramified characters by 
\[
X^\nr(G)^1=\Hom(G(F)/G(F)^1,\CC^1).
\]
We have a surjective homomorphism 
\[
\ak_{G,\CC}^*\longrightarrow X^\nr(G)
\]
defined by $\lambda\mapsto |\cdot|_G^\lambda$, where $|g|_G^\lambda=e^{\langle \lambda, H_G(g) \rangle}$. If $\lambda=\theta\otimes s\in \ak_{G,\CC}^*$, then $|g|_G^\lambda=|\theta(g)|^s$. The surjection $\ak_{G,\CC}^*\to X^\nr(G)$ descends to an isomorphism $\ak_{G,\CC}^*/\ak_{G,F}^\vee\to X^\nr(G)$. Thus, if $F$ is archimedean, then the map $\ak_{G,\CC}^*\to X^\nr(G)$ is an isomorphism and $X^\nr(G)$ is a complex vector space. If $F$ is non-archimedean, then $X^\nr(G)$ is a complex torus and the homomorphism $\ak_{G,\CC}^*\to X^\nr(G)$ factors through the complex torus $\ak_{G,\CC}^*/\frac{2\pi i}{\log q_F}X^*(G)$ with finite kernel $\ak_{G,F}^\vee/\frac{2\pi i}{\log q_F}X^*(G)$.

\subsection{Parabolic and Levi subgroups}
We recall the following standard notation.
\begin{itemize}
    \item $\Pc(M)=\Pc^G(M)$ is the set of parabolic subgroups of $G$ with $M$ as a Levi factor; and
    \item $\Lc(M)=\Lc^G(M)$ is the set of Levi subgroups of $G$ containing $M$.
\end{itemize}
If $P$ is a parabolic subgroup of $G$, we denote the unipotent radical of $P$ by $N_P$. We refer to a pair $(P,M)$ consisting of a Levi subgroup $M$ of $G$ and a parabolic subgroup $P\in\Pc(M)$ as a parabolic pair of $G$.

Let $M\subseteq G$ be a Levi subgroup of $G$. We have $A_G\subseteq A_M$. The restriction homomorphism $X^*(G)\to X^*(M)$ is injective, so gives rise to a linear injection 
\[
\ak_G^*=X^*(G)\otimes_\ZZ\RR\longrightarrow\ak_M^*=X^*(M)\otimes_\ZZ\RR
\]
and a dual linear surjection
\[
\ak_M=\Hom_\ZZ(X^*(M),\RR)\longrightarrow \ak_G=\Hom_\ZZ(X^*(G),\RR).
\]
The restriction homomorphism $X^*(A_M)\to X^*(A_G)$ is surjective, so gives rise to a linear surjection
\[
\ak_M^*=X^*(A_M)\otimes_\ZZ\RR\longrightarrow\ak_G^*=X^*(A_G)\otimes_\ZZ\RR
\]
and a dual linear injection
\[
\ak_G=\Hom_\ZZ(X^*(A_G),\RR)\longrightarrow\ak_M=\Hom_\ZZ(X^*(A_M),\RR).
\]
Let $\ak_M^G=\ker(\ak_M\to\ak_G)$. The homomorphism $\ak_G\to\ak_M$ is a section of $\ak_M\to\ak_G$. Thus, we have a split short exact sequence
\[
\begin{tikzcd}
    0 \arrow[r] & \ak_M^G \arrow[r] & \ak_M \arrow[r] & \ak_G \arrow[r] \arrow[l, bend right=33] & 0
\end{tikzcd}
\]
and $\ak_M=\ak_M^G\oplus\ak_G$. We also have the dual exact sequence
\[
\begin{tikzcd}
    0 \arrow[r] & \ak_G^* \arrow[r] & \ak_M^* \arrow[r] \arrow[l, bend right=33] & (\ak_M^G)^* \arrow[r] & 0
\end{tikzcd}
\]
and $\ak_M^*=(\ak_M^G)^*\oplus\ak_G^*$.

We fix a minimal parabolic pair $(P_0,M_0)$ of $G$. We will often use a subscript or superscript ``0'' instead of ``$M_0$'' to indicate dependence on $M_0$. For example, we write $N_0=N_{P_0}$ and $A_0=A_{M_0}$. With respect to $M_0$, a parabolic subgroup $P$ of $G$ is called standard (resp. semistandard) if it contains $P_0$ (resp. $M_0$). A Levi subgroup of $G$ is said to be semistandard if it contains $M_0$. If $P$ is a semistandard parabolic, then it has a unique semistandard Levi factor $M_P$.

Let $(P,M)$ be a parabolic pair. We denote the set of simple roots of $(P,A_M)$ (or equivalently $(N_P,A_M)$) by $\Delta(P,A_M)$, which is a basis of $(\ak_M^G)^*$. There is an associated set of simple coroots $\Delta^\vee(P,A_M)$, which is a basis of $\ak_M^G$, and there is a bijection $\Delta(P,A_M)\to\Delta^\vee(P,A_M),\alpha\mapsto\alpha^\vee$. For $(P,M)$ semistandard, we write $\Delta_P=\Delta(P,A_M)$ and $\Delta_P^\vee=\Delta^\vee(P,A_M)$. (Cf. \cite[\S1.3]{MWLocal}.) We define the open cone
\[
(\ak_M^*)^{P,>0}=\{\lambda\in\ak_M^* : \langle\lambda,\alpha^\vee\rangle>0,\,\forall\alpha\in\Delta(P,A_M)\},
\]
which plays a role in the Langlands classification. We denote its closure by $(\ak_M^*)^{P,\geq0}$. Similarly, we have the open cone
\[
(\ak_M)^{P,>0}=\{H\in\ak_M : \langle\alpha,H\rangle>0,\,\forall\alpha\in\Delta(P,A_M)\}
\]
and its closure $(\ak_M)^{P,\geq0}$.

\subsection{Maximal compact subgroups}
A maximal compact subgroup $K$ of $G(F)$ is said to be in good position (or admissible) relative to a Levi subgroup $M$ of $G$ if the following holds.
\begin{itemize}
    \item When $F$ is archimedean, the Lie algebra of $K$ and $A_M(F)$ are orthogonal with respect to the Killing form of $G$.
    \item When $F$ is non-archimedean, $K$ is the stabiliser of a special vertex in the apartment attached to a maximal split torus of $M$.
\end{itemize}
We recall from \cite[\S1]{ArtTFInvForm} that if $K$ is in good position relative to $M$, then 
\begin{enumerate}
    \item $G(F)=P(F)K$ for any $P\in\Pc(M)$;
    \item any coset in $G(F)/M(F)$ which normalises $M$ has a representative in $K$;
    \item $P(F)\cap K=(M_P(F)\cap K)(N_P(F)\cap K)$ for any parabolic subgroup $P$ of $G$ containing $M$; and
    \item if $L\in\Lc(M)$, then $K_L\defeq K\cap L(F)$ is a maximal compact subgroup of $L(F)$ that is in good position relative to the Levi subgroup $M$ of $L$.
\end{enumerate}
If $K$ is in good position relative to $M$, then $K$ is in good position relative to every Levi subgroup containing $M$.

We fix a maximal compact subgroup $K$ of $G(F)$ that is in good position relative to our fixed minimal Levi subgroup $M_0$ of $G$. In particular, the Iwasawa decomposition $G(F)=P_0(F)K$ holds.

Let $P$ be a semistandard parabolic subgroup of $G$. We can extend $H_{M_P}:M_P(F)\to\ak_{M_P}$ to a homomorphism $H_P:P(F)\to\ak_{M_P}$ by composing with $P(F)\to M_P(F)$. We have the decompositions $G(F)=P(F)K=M_P(F)N_P(F)K$. For each $x\in G(F)$ we choose elements $p_P(x)\in P(F)$ and $k_P(x)\in K$ such that $x=p_P(x)k_P(x)$, and let $m_P(x)\in M_P(F)$ and $n_P(x)\in N_P(F)$ be the unique elements such that $p_P(x)=m_P(x)n_P(x)$. We extend $H_{P}$ to a function $H_P:G(F)\to\ak_{M_P}$ by $H_P(x)=H_P(p_P(x))=H_{M_P}(m_P(x))$.

Let $\Delta_0=\Delta_{P_0}$, $\ak_0^{\geq0}=(\ak_0)^{P_0,\geq0}$, and define $M_0(F)^{\geq0}$ to be the set of $m\in M_0(F)$ such that $H_0(m)\in\ak_0^{\geq0}$. Then $G(F)=KM_0(F)^{\geq0}K$. (See \cite[\S1.1]{MWLocal}.)

\subsection{Weyl groups}
For $T$ a torus of $G$, we make use of the following Weyl groups:
\begin{itemize}
    \item the absolute Weyl group $W(G,T)\defeq N_G(T)/C_G(T)$;
    
    \item the relative Weyl group $W_F(G,T)\defeq N_{G(F)}(T)/C_{G(F)}(T)$; and
    
    \item the stable Weyl group $W(G,T)(F)\defeq(N_G(T)/C_G(T))(F)$.
\end{itemize}
We have 
\[
W_F(G,T)\subseteq W(G,T)(F)\subseteq W(G,T).
\]
By Galois descent, we can express the stable Weyl group as
\[
W(G,T)(F)=\{g\in N_G(T) \mid g^{-1}\sigma(g)\in C_G(T),\,\forall\sigma\in\Gamma_F\}/C_G(T),
\]
and therefore $W(G,T)(F)$ consists of all $w\in W(G,T)$ such that 
\[
\Int(w):C_G(T)\to C_G(T)
\]
is defined over $F$.

If $T$ is a maximal torus of $G$, then $C_G(T)=T$ and $W_F(G,T)$ is equal to $W(G(F),T(F))\defeq N_{G(F)}(T(F))/C_{G(F)}(T(F))$. For a Levi subgroup $M$ of $G$, we write $W^G(M)$ for the relative Weyl group $W_F(G,A_M)$. That is,
\[
W^G(M)=N_{G(F)}(A_M)/C_{G(F)}(A_M)=N_{G(F)}(M)/M(F).
\]
The Weyl group $W^G(M_0)$ is the relative Weyl group of $G$, and we abbreviate it by $W_0^G$. We have
\[
W^G(M)=\{\wt{w}\in N_{G(F)}(A_0) : \wt{w}\cdot M=M\}/M(F)
\]
and we have a canonical isomorphism
\[
W^G(M)=\{w\in W_0^G : w\cdot M=M\}/W_0^M.
\]

Fix a $W_0^G$-invariant inner product $\langle\cdot,\cdot\rangle$ on $\ak_0$. This restricts to a $W^G(M)$-invariant inner product on $\ak_M$ for each semistandard Levi subgroup $M$ of $G$. Transporting these inner products by conjugation, we obtain a unique $W^G(M)$-invariant inner product $\langle\cdot,\cdot\rangle$ on $\ak_M$ for each Levi subgroup $M$ of $G$. We denote the associated norm by $\|\cdot\|$. The decomposition $\ak_M=\ak_M^G\oplus\ak_G$ is orthogonal with respect to the inner product on $\ak_M$.

\subsection{Group norms}
We recall the notion of norms on affine varieties over local fields from \cite[\S18]{KottwitzHarmonic}, which are used to capture polynomial growth and decay. Let $X$ be a set. An abstract norm on $X$ is a function $\|\cdot\|:X\to\RR_{\geq1}$. Let $\|\cdot\|_1,\|\cdot\|_2$ be abstract norms on $X$. We write $\|\cdot\|_1\preceq\|\cdot\|_2$ if there exists $M>0$ such that $\|\cdot\|_1\ll\|\cdot\|_2^M$. The abstract norms $\|\cdot\|_1,\|\cdot\|_2$ are said to be equivalent and we write $\|\cdot\|_1\approx\|\cdot\|_2$ if $\|\cdot\|_1\preceq\|\cdot\|_2$ and $\|\cdot\|_2\preceq\|\cdot\|_1$.

Let $X$ be an affine scheme of finite type over $F$. There is a canonical equivalence class of abstract norms on $X(F)$, defined as follows. For any set of generators $f_1,\dots,f_m$ of the $F$-algebra $\Oc_X(X)$, of regular functions on $X$, we have an abstract norm $\|\cdot\|$ on $X(F)$ defined by
\[
\|x\|\defeq\sup\{1,|f_1(x)|_F,\dots,|f_m(x)|_F\}.
\]
The equivalence class of $\|\cdot\|$ does not depend on the choice of $f_1,\dots,f_m$. We call any abstract norm in the equivalence class of $\|\cdot\|$ a norm on $X(F)$. For the definition of norms on $X(F)$ when $X$ is a non-affine scheme of finite type over $F$, see \cite[\S18.5]{KottwitzHarmonic}.

Let $G$ be a linear algebraic group over $F$. By \cite[Proposition 18.1 (7)]{KottwitzHarmonic}, if $\Omega$ is a bounded (relatively compact) subspace of $G(F)$, for every norm $\|\cdot\|$ on $G(F)$ there exist $C,M>0$ such that $\|\omega_1g\omega_2\|\leq C\|g\|^M$ for all $\omega_1,\omega_2\in\Omega$ and $g\in G(F)$.

We may construct a norm on $G(F)$ as follows. Let $\iota:G\to\GL(V)$ be a faithful algebraic representation of $G$ on a finite-dimensional vector space $V$ over $F$. Choose a vector space norm $\|\cdot\|$ on $V$ that is compatible with the canonical absolute value $|\cdot|_F$ on $F$. For example, if $e_1,\dots,e_n$ is a basis of $V$ and $v=\sum_{i=1}^na_ie_i$, we can take $\|v\|=\max_{i=1,\dots,n}\{|a_i|_F\}$. We equip $\End(V)$ with the corresponding operator norm, which we also denote by $\|\cdot\|$. Define $\|g\|=\max(\|\iota(g)\|,\|\iota(g^{-1})\|)$ for all $g\in G(F)$. Then $\|\cdot\|:G(F)\to\RR_{\geq1}$ is a norm on $G(F)$ that satisfies the following properties:
\begin{enumerate}
    \item $\|\cdot\|$ is continuous;
    \item $\|g_1g_2\|\leq\|g_1\|\|g_2\|$ for all $g_1,g_2\in G(F)$;
    \item $\|g^{-1}\|=\|g\|\geq1$ for all $g\in G(F)$;
    \item for all $R\geq 0$, the subspace $B_R=\{g\in G(F) : \|g\|\leq R\}$ is compact.
\end{enumerate}
Since $\|\cdot\|$ is a norm on $G(F)$, changing the choice of $\iota$ or of the vector space norm on $V$ results in an equivalent norm on $G(F)$. We call the norms $\|\cdot\|$ on $G(F)$ that are obtained from the above construction of group norms.

If $\Omega_1,\Omega_2$ are bounded subsets of $G(F)$, then $\|\omega_1x\omega_2\|\asymp\|x\|$ for $x\in G(F)$ and $\omega_i\in\Omega_i$. This follows from 
\[
\|x\|=\|y_1^{-1}y_1xy_2y_2^{-1}\|\leq\|y_1^{-1}\|\|y_1 xy_2\|\|y_2^{-1}\|
\]
and $\|y_1xy_2\|\leq\|y_1\|\|x\|\|y_2\|$ for all $x,y_1,y_2\in G(F)$.

Now suppose that $G$ is our connected reductive group. We fix a group norm $\|\cdot\|=\|\cdot\|_G$ on $G(F)$ and define the non-negative function $\sigma=\sigma_G\defeq\log\|\cdot\|$. For a closed central subgroup $\Zc$ of $G(F)$, we define $\sigma^{\Zc}=\sigma_G^\Zc:G(F)\to\RR_{\geq0}$ by $\sigma^{\Zc}(g)=\inf_{z\in \Zc}\sigma(zg)$. We will use this to define growth properties modulo $\Zc$. If $\Zc$ is of the form $\Zc=Z(F)$ for a subgroup $Z\subseteq Z_G$, we will write $\sigma^Z=\sigma_G^Z$ instead of $\sigma^{\Zc}=\sigma_G^\Zc$ to simplify the notation.

Recall that we have fixed a $W_0^G$-invariant inner product on $\ak_0$. There exist constants $C_1,C_2>0$ such that
\[
C_1(1+\|H_0(m_0)\|)\leq 1+\sigma(m_0)\leq C_2(1+\|H_0(m_0)\|)
\]
for all $m_0\in M_0(F)$. That is, $1+\sigma(m_0)\asymp 1+\|H_0(m_0)\|$ for $m_0\in M_0(F)$.

\subsection{Universal enveloping algebras}
Let $\gk$ be a complex reductive Lie algebra. We denote the centre of the universal enveloping algebra $\Uk(\gk)$ of $\gk$ by $\Zk(\gk)$. We recall that 
\[
\Zk(\gk)=\Uk(\gk)^{\ad(\gk)}=\Uk(\gk)^{e^{\ad(\gk)}}=\Uk(\gk)^{\Int(\gk)},
\]
where $\Int(\gk)$ is the group of inner automorphisms of $\gk$, that is, the connected subgroup of $\Aut(\gk)$ with Lie algebra $\ad(\gk)$. 

For each Levi subalgebra $\mk$ of $\gk$ there is a canonical Harish-Chandra homomorphism $\xi_{\mk}^\gk:\Zk(\gk)\to\Zk(\mk)$. (See Definition 6.3 in \cite{VoganUnitary}.) The construction of $\xi_\mk^\gk$ uses a choice of a parabolic subalgebra $\pk$ with Levi factor $\mk$, but $\xi_\mk^\gk$ does not depend on this choice. The Harish-Chandra homomorphism is injective and has image $\Zk(\mk)^{W(\gk,\mk)}$. The isomorphism $\xi_\mk^\gk:\Zk(\gk)\to\Zk(\mk)^{W(\gk,\mk)}$ is often called a Harish-Chandra isomorphism. The Harish-Chandra homomorphisms are functorial in the sense that if $\lk\subseteq\mk$ are two Levi subalgebras of $\gk$, then $\xi_\lk^\gk=\xi_\lk^\mk\circ\xi_\mk^\gk$. When $\mk=\tk$ is a maximal toral subalgebra, then $\Zk(\tk)=\Uk(\tk)=\Sym(\tk)$ and the Harish-Chandra isomorphism $\xi_\tk^\gk:\Zk(\gk)\to\Sym(\tk)^{W(\gk,\tk)}$ is the classical Harish-Chandra isomorphism most commonly found in the literature. If $\tk$ is contained in $\mk$, then it follows from functoriality of the Harish-Chandra homomorphisms that the following diagram commutes
\[
\begin{tikzcd}
    \Zk(\gk) \arrow[r, "\xi_\mk^\gk"] \arrow[d,swap,"\xi_{\tk}^\gk"] & \Zk(\mk) \arrow[d, "\xi_\tk^\mk"] \\
    \Sym(\tk)^{W(\gk,\tk)} \arrow[r,hook] & \Sym(\tk)^{W(\mk,\tk)}
\end{tikzcd}
\]
That is, if we identify $\Zk(\gk)=\Sym(\tk)^{W(\gk,\tk)}$ and $\Zk(\mk)=\Sym(\tk)^{W(\mk,\tk)}$ using the Harish-Chandra isomorphisms, then the Harish-Chandra homomorphism $\xi_\mk^\gk:\Zk(\gk)\to\Zk(\mk)$ is simply the inclusion $\Sym(\tk)^{W(\gk,\tk)}\hookrightarrow\Sym(\tk)^{W(\mk,\tk)}$.

If $G$ is a real Lie group, then $\Uk(\gk_\CC)^{\Ad(G)}\subseteq\Zk(\gk_\CC)$. If $G$ belongs to the Harish-Chandra class, in particular, if $G$ is the group of $\RR$-points of a connected reductive group over $\RR$, then $\Ad(G)\subseteq\Int(\gk_\CC)$ and it follows that $\Zk(\gk_\CC)=\Uk(\gk_\CC)^{\Ad(G)}$. 

We will need a ``norm'' on the set $\Hom_{\CC\alg}(\Zk(\gk_\CC),\CC)$ of infinitesimal characters  of $G$. Choose a maximal torus $T$ of $G_\CC$. By the Harish-Chandra isomorphism, we may identify 
\[
\Hom_{\CC\alg}(\Zk(\gk_\CC),\CC)=\tk^*/W(G_\CC,T).
\]
We fix a $W(G_\CC,T)$-invariant inner product on $\tk^*$. Then, for $\mu\in \tk^*/W(G_\CC,T)$ the resulting norm $\|\mu\|=\|\mu\|_G$ is well-defined. Different choices of inner products or maximal tori result in a function $\|\cdot\|'$ on the set of infinitesimal characters of $G$ with $\|\cdot\|'\asymp\|\cdot\|$. Since we will only use $\|\cdot\|$ in estimates, definitions made using $\|\cdot\|$ will not depend on the choices made to define it. Note that if $M$ is a Levi subgroup of $G$, then $\|\cdot\|_M\asymp\|\cdot\|_G$.

\subsection{Measures}
We normalise the Haar measure $\dd^+x$ on $F$ so that $\vol(\Oc_F)=1$ if $F$ is non-archimedean, $\vol(|x|_F\leq 1)=2$ if $F\cong\RR$, and $\vol(|x|_F\leq 1)=\pi$ if $F\cong\CC$. We use the Haar measure $\dd^\times x$ on $F^\times$ defined by $\dd^\times x=|x|_F^{-1}\dd^+x$ if $F$ is archimedean, and $\dd^\times x=(1-q_F)^{-1}|x|_F^{-1}\dd^+x$ if $F$ is non-archimedean, so that $\vol(\Oc_F^\times)=1$.

Let $T$ be any torus over $F$. We define canonical normalisations of the Haar measures on $A_T(F)$, $\ak_T$, $A_T(F)^1$, and $T(F)$ as follows. Let $l$ be the split rank of $T$ and choose an isomorphism $A_T\cong\GG_m^l$, that is, a $\ZZ$-basis of $X^*(A_T)$. This determines isomorphisms $A_T(F)\cong(F^\times)^l$ and $\ak_T\cong\RR^l$. Transport the product Haar measure on $(F^\times)^l$ to $A_T(F)$ and the usual Haar measure on $\RR^l$ to $\ak_T$. The resulting Haar measures on $A_T(F)$ and $\ak_T$ do not depend on the choice of isomorphism $A_T\cong\GG_m^l$. We use the counting measure for lattices in $\ak_T$ and $\ak_T^*$. If $F$ is archimedean, then $\wt{\ak}_{T,F}=\ak_T$. If $F$ is non-archimedean, then $\wt{\ak}_{T,F}$ is a lattice in $\ak_T$; in fact, under our chosen isomorphism $\ak_T\cong\RR^l$, we have $\wt{\ak}_{T,F}\cong(\log q_F\ZZ)^l$ and thus $\vol(\ak_T/\wt{\ak}_{T,F})=(\log q_F)^l$. In any case, we have fixed a Haar measure on $\wt{\ak}_{T,F}$. We normalise the Haar measure of the maximal compact subgroup $A_T(F)^1$ of $A_T(F)$ so that with respect to the short exact sequence
\[
1\longrightarrow A_T(F)^1\longrightarrow A_T(F)\longrightarrow\wt{\ak}_{T,F}\longrightarrow 0
\]
it is compatible with the normalisations of the Haar measures on $A_T(F)$ and $\ak_T$, 
If $F\cong\RR$, then $A_T(F)^1\cong\{\pm1\}^l$ and the Haar measure on $A_T(F)^1$ is the counting measure. If $F\cong\CC$, then $A_T(F)^1\cong (S^1)^l$ and the Haar measure on $A_T(F)^1$ is normalised so that $\vol(A_T(F)^1)=(2\pi)^l$. If $F$ is non-archimedean, then $A_T(F)^1\cong(\Oc_F^\times)^l$ and the Haar measure on $A_T(F)^1$ is normalised so that $\vol(A_T(F)^1)=1$, and thus it coincides with the restriction of our Haar measure on $A_T(F)$. We normalise the Haar measure on $T(F)$ so that $\vol(T(F)/A_T(F))=1$.

Let $M$ be a Levi subgroup of $G$. Recall that we have fixed an inner product on $\ak_M$. We identify $i\ak_M^*$ with the Pontryagin dual of $\ak_M$ via $\lambda\mapsto e^{\langle\lambda,\cdot\rangle}$ and equip $i\ak_M^*$ with the Haar measure dual to the fixed Haar measure on $\ak_M$.

Let $P$ be a semistandard parabolic of $G$. The group $N_P(F)$ has a canonical Haar measure, which we now describe, following \cite{ArthurLocalTF}. Let $\delta_P$ be the modular function of $P$. We have 
\[
\delta_P(mn)=e^{\langle H_M(m),2\rho_P \rangle},
\]
for all $m\in M_P(F)$ and $n\in N_P(F)$, where $\rho_P$ is half the sum of the roots (counted with multiplicity) of $(P,A_{M_P})$. Let $\ol{P}$ be the parabolic subgroup opposite to $P$ with respect to $M_P$. For any Haar measure $\nu$ on $N_P(F)$, the integral
\[
\gamma(P,\nu)=\int_{N_P(F)}e^{\langle H_{\ol{P}}(n),2\rho_{\ol{P}} \rangle} \dd{\nu(n)} >0
\]
converges and we have $\gamma(P,c\nu)=c\gamma(P,\nu)$ for all $c>0$. Therefore $\gamma(P,\nu)^{-1}\nu$ is a canonical Haar measure on $N_P(F)$. From now on, we will use the canonical Haar measure on $N_P(F)$.

Let $M$ be a semistandard Levi subgroup of $G$. The Haar measures on $M(F)$ and $G(F)$ can be normalised so that for every parabolic subgroup $P$ of $G$ with Levi factor $M$, we have
\[
\int_{G(F)}f(g)\dd{g}=\int_{N_P(F)}\int_{M(F)}\int_{N_{\ol{P}}(F)}f(nm\ol{n})\delta_P(m)^{-1}\dd{\ol{n}}\dd{m}\dd{n}.
\]
Following \cite{ArthurLocalTF}, we say that the Haar measures $\dd{x}$ and $\dd{m}$ are compatible in this case. Let $L\subseteq M$ be a semistandard Levi subgroup of $G$. Then $L$ is a semistandard Levi subgroup of $M$. Let $\dd{g}$, $\dd{m}$, and $\dd{l}$ be Haar measures on $G(F)$, $M(F)$, and $L(F)$ respectively. If any two of the pairs $(\dd{g},\dd{m})$, $(\dd{m},\dd{l})$, and $(\dd{g},\dd{l})$ are compatible, then so is the third pair. We say that a collection of Haar measures $\{\mu_M\}_{M\in\Lc^G(M_0)}$ on the semistandard Levi subgroups $M(F)$ of $G(F)$ (including $G(F)$ itself) is compatible if each Haar measure $\mu_M$ is compatible with the Haar measure $\mu_G$ on $G(F)$, or equivalently $\dd{m}$ is compatible with $\dd{l}$ whenever $L\subseteq M$. If $\{\mu_M\}_{M\in\Lc^G(M_0)}$ and $\{\mu_{M}'\}_{M\in\Lc^G(M_0)}$ are two collections of compatible Haar measures on the semistandard Levi subgroups of $G(F)$, there exists $c>0$ such that $\mu_M'=c\mu_M$ for all $M\in\Lc^G(M_0)$. We choose a collection of compatible Haar measures $\dd{m}$ on the semistandard Levi subgroups $M(F)$ of $G(F)$. If the minimal Levi $M_0$ is a torus (necessarily maximal in $G$), then its Haar measure is already fixed and we must choose the Haar measures of the other semistandard Levis compatibly.

For a semistandard parabolic subgroup $P$ of $G$, we normalise the left Haar measure $\dd_lp$ on $P(F)$ so that
\[
\int_{P(F)}f(p)\dd_l{p}=\int_{M_P(F)}\int_{N_P(F)}f(mn)\dd{n}\dd{m}.
\]
The corresponding right Haar measure $\dd_rp$ on $P(F)$ is
\begin{align*}
\int_{P(F)}f(p)\dd_r{p}&=\int_{M_P(F)}\delta_P(m)^{1/2}\int_{N_P(F)}f(mn)\dd{n}\dd{m} \\
&=\int_{M_P(F)}\int_{N_P(F)}f(nm)\dd{n}\dd{m}.
\end{align*}

For each semistandard Levi subgroup $M$ of $G$, we normalise the Haar measure on the maximal compact subgroup $K_M\subseteq M(F)$ so that $\vol(K_M)=1$. In particular, $\vol(K)=1$. We have the integral formulas
\[
\int_Kf(k)\dd{k}=\int_{N_{\ol{P}}}f(k_P(\ol{n}))e^{2\rho_P(H_P(\ol{n}))}\dd{\ol{n}}
\]
and
\[
\int_{G(F)}f(g)\dd{g}=\int_K\int_{P(F)}f(pk)\dd{n}\dd_lp,
\]
as in \cite{ArthurLocalTF}. By conjugation, the Haar measures on semistandard Levi subgroups, parabolic subgroups, and their unipotent radicals can be transported canonically to arbitrary Levi subgroups, parabolic subgroups, and their unipotent radicals.

\subsection{Spaces of functions and distributions}

Let $\Zc$ denote a closed central subgroup of $G(F)$ equipped with a choice of Haar measure and let $\zeta:\Zc\to\CC^\times$ be a character of $\Zc$. The pair $(\Zc,\zeta)$ is called a central character datum or simply a central datum of $G$. We say that a central datum $(\Zc,\zeta)$ is unitary if $\zeta$ is unitary. The trivial central datum is the one with $\Zc=1$. Later, we will only use unitary central data and only consider the cases $\Zc=1$ or $\Zc=A_G(F)$. We may omit $\Zc$ and $\zeta$ from notation if the central datum $(\Zc,\zeta)$ is trivial in the sense that $\Zc=1$. For a function $f:G(F)\to\CC$ and $z\in\Zc$, we define $zf:G(F)\to\CC$ by $zf(g)=f(z^{-1}g)$.

We write $C_c(G,\zeta)$ for the space of continuous functions on $G(F)$ that are $\zeta^{-1}$-equivariant (i.e. $f(zg)=\zeta(z)^{-1}f(g)$, or equivalently $zf=\zeta(z)f$) and whose support is compact modulo $\Zc$. The space $C_c(G,\zeta)$ has a natural topology. For each closed subspace $B$ of $G(F)$ that is $\Zc$-stable and compact modulo $\Zc$, we define $C_B(G,\zeta)$ to be the subspace of functions in $C_c(G,\zeta)$ whose support is contained in $B$. It is a Banach space with respect to the sup norm. We have $C_c(G,\zeta)=\bigcup_BC_B(G,\zeta)$ and give $C_c(G,\zeta)$ the inductive limit topology in the category of locally convex spaces, making it a strict LF-space. The space of $\zeta$-equivariant Radon measures on $G(F)$ is $C_c(G,\zeta)'$, the continuous dual with the (standard) strong dual topology.

Similarly, we denote by $C_c^\infty(G,\zeta)$ the space of all functions $f:G(F)\to\CC$ that are smooth (i.e. locally constant if $F$ is non-archimedean), $\zeta^{-1}$-equivariant, and whose support is compact modulo $\Zc$. The space $C_c^\infty(G,\zeta)$ also has a natural topology. Suppose that $F$ is archimedean. Let $B$ be a closed subspace of $G(F)$ that is $\Zc$-stable and compact modulo $\Zc$. Define $C_B^\infty(G,\zeta)$ to be the subspace of functions in $C_c^\infty(G,\zeta)$ whose support is contained in $B$. We give $C_B^\infty(G,\zeta)$ the Fr\'echet space topology defined by the family of seminorms
\[
\|f\|_{D}\defeq\sup_{g\in G(F)}|Df(g)|
\]
for $D$ a $\Zc$-invariant differential operator on $D$. We have $C_c^\infty(G,\zeta)=\bigcup_{B}C_B^\infty(G,\zeta)$ and give $C_c^\infty(G,\zeta)$ the inductive limit topology in the category of locally convex spaces, making it a strict LF-space. Now, suppose that $F$ is non-archimedean. Then $C_c^\infty(G,\zeta)$ is naturally a countable increasing union of finite-dimensional subspaces, and we equip it with the locally convex limit topology, making it a strict LF-space. The topology on $C_c^\infty(G,\zeta)$ is simply the finest locally convex topology. The space of $\zeta$-equivariant distributions on $G(F)$ is $C_c^\infty(G,\zeta)'$. The inclusion $C_c^\infty(G,\zeta)\to C_c(G,\zeta)$ is continuous with dense image, and thus $\zeta$-equivariant Radon measures may be identified with $\zeta$-equivariant distributions.

We will also make use of the subspace $C_c^\infty(G,\zeta,K)$ of left and right $K$-finite functions in $C_c^\infty(G,\zeta)$. If $F$ is non-archimedean, then $C_c^\infty(G,\zeta,K)=C_c^\infty(G,\zeta)$. For $F$ archimedean, the space $C_c^\infty(G,\zeta,K)$ is naturally an inductive limit \cite{ArtIntResI}. Indeed, let $\Gamma\subseteq\wh{K}$ be a finite set of irreducible representations of $K$. We define $C_c^\infty(G,\zeta,K)_\Gamma$ to be the subspace of functions in $C_c^\infty(G,\zeta,K)$ that transform on each side under $K$ according to a finite direct sum of representations in $\Gamma$. For each compact subspace $B$ of $G$, we define $C_B^\infty(G,\zeta,K)_\Gamma=C_B^\infty(G,\zeta)\cap C_c^\infty(G,\zeta,K)_\Gamma$, which is a closed subspace of $C_B^\infty(G,\zeta)$ and thus a Fr\'echet space. We have $C_c^\infty(G,\zeta,K)_\Gamma=\colim_B C_B^\infty(G,\zeta,K)_\Gamma$ and $C_c^\infty(G,\zeta,K)=\colim_\Gamma C_c^\infty(G,\zeta,K)_\Gamma$, and we give each the locally convex inductive limit topology, which makes them strict LF-spaces. The inclusion $C_c^\infty(G,\zeta,K)\to C_c^\infty(G,\zeta)$ is continuous.

The most natural space for harmonic analysis is the Harish-Chandra Schwartz space. Recall that we have fixed a minimal parabolic $P_0$ of $G$, a Levi factor $M_0$ of $P_0$, and maximal compact subgroup $K$ of $G(F)$ in good position relative to $M_0$. Thus, the Iwasawa decomposition $G(F)=P_0(F)K$ holds. Let $e:G(F)\to\RR_{>0}$ be the unique function satisfying $e(K)=1$ and $e(p_0g)=\delta_{P_0}^{1/2}(p_0)e(g)$ for all $p_0\in P_0(F)$ and $g\in G(F)$. That is, $e$ is the unique smooth vector of the parabolically induced representation $I_{M_0,P_0}^G(1_{M_0})$ satisfying $e(K)=1$. Note that $e$ is $Z_G(F)$-invariant and right $K$-invariant. Although $e$ depends on the choice of $P_0$ and $K$, it does not depend on the choice of $M_0$. We define $\Xi=\Xi^G:G(F)\to\RR_{>0}$ by
\[
\Xi(g)=\int_Ke(kg)\dd{k}
\]
for all $g\in G(F)$. Note that $\Xi$ is $Z_G(F)$-invariant and bi-$K$-invariant. Recall that for $x\in G(F)$, we choose elements $p_0(x)=p_{P_0}(x)\in P_0(F)$ and $k_0(x)=k_{P_0}(x)\in K$ such that $x=p_0(x)k_0(x)$. We have $\Xi(g)=\int_K\delta_{P_0}^{1/2}(p_0(kg))\dd{k}$. Although $\Xi$ depends on the choices of $P_0$ and $K$, different choices result in a Xi-function $\Xi'$ with $\Xi\asymp\Xi'$. Since we will only use $\Xi$ in estimates, the choices made in the definition of $\Xi$ will have no impact on definitions made in terms of it.

A continuous function $f:G(F)\to\CC$ is said to be rapidly decreasing modulo $\Zc$ if $|f|$ is $\Zc$-invariant and it satisfies one of the following equivalent conditions \cite[\S5]{vigneras}.
\begin{enumerate}
    \item For all $N\in\ZZ_{>0}$, we have
    \[
    \|f\|_N\defeq\sup_{g\in G(F)}|f(g)|(1+\sigma^{\Zc}(g))^N\Xi(g)^{-1}<\infty.
    \]
    \item For all $N\in\ZZ_{>0}$, we have $f\in L^2(G(F)/\Zc,(1+\sigma^{\Zc}(g))^{N}\dd{g})$.
\end{enumerate}

We now define the Harish-Chandra Schwartz space $\Cc(G,\zeta)$ for a unitary central datum $(\Zc,\zeta)$ of $G(F)$.

Suppose that $F$ is non-archimedean. The Harish-Chandra Schwartz space $\Cc(G,\zeta)$ is the space of uniformly smooth complex-valued functions on $G(F)$ that are $\zeta^{-1}$-equivariant and rapidly decreasing modulo $\Zc$ (cf. \cite[\S4.4]{Silberger} and \cite[\S5]{vigneras}). Let $K_0$ be a compact open subgroup of $G(F)$ that is sufficiently small so that $\zeta(\Zc\cap K_0)=1$. For each compact open subgroup $K$ of $K_0$, let $\Cc_K(G,\zeta)$ denote the subspace of $\Cc(G,\zeta)$ that are left and right $K$-invariant. Then $\Cc(G,\zeta)=\bigcup_{K}\Cc_K(G,\zeta)$. The spaces $\Cc_K(G,\zeta)$ are Fr\'echet spaces when given the topology determined by the seminorms $\|\cdot\|_N$. The space $\Cc(G,\zeta)$ is given the inductive limit topology in the category of locally convex spaces and is an LF-space.

Suppose that $F$ is archimedean. The Harish-Chandra Schwartz space $\Cc(G,\zeta)$ is defined to be the space of all smooth functions $f$ on $G(F)$ such that all derivatives $ufv$ for $u,v\in\Uk(\gk_\CC)$ are rapidly decreasing modulo $\Zc$. Given $u,v\in\Uk(\gk_\CC)$ and $N\in\ZZ_{>0}$, for each $f\in C^\infty(G(F))$, we define
\[
\|f\|_{u,v,N}=\|ufv\|_N=\sup_{g\in G(F)}|ufv(g)|(1+\sigma^Z(g))^N\Xi(g)^{-1}.
\]
Then $\Cc(G,\zeta)$ is a Fr\'echet space when given the topology determined by the seminorms $\|\cdot\|_{u,v,N}$. There are other Schwartz spaces of $G(F)$ in the literature, but since we will only use the Harish-Chandra Schwartz space we may refer to its elements as Schwartz functions. We will write $\Cc_{c}(G,\zeta)=C_c^\infty(G,\zeta)$ and we will write $\Cc_{(c)}(G,\zeta)$ as a shorthand for $\Cc(G,\zeta)$ (resp. $\Cc_{c}(G,\zeta)$) so that we can make statements about both spaces at the same time.

The other characterisation above of a function being rapidly decreasing modulo $\Zc$ leads to a different way of defining the topology on the Schwartz space $\Cc(G,\zeta)$. Denote the $L^2$-norm on $L^2(G(F)/\Zc, (1+\sigma^{\Zc}(g))^N\dd{g})$ by $\|\cdot\|_{2,N}$. Replacing $\|\cdot\|_N$ with $\|\cdot\|_{2,N}$ in the definition of the topology on $\Cc(G,\zeta)$ given above yields in the same topology. 

We define the space of tempered $\zeta$-equivariant distributions to be $\Cc(G,\zeta)'$. Since the inclusion $C_c^\infty(G,\zeta)\to\Cc(G,\zeta)$ is continuous with dense image, we have a continuous injection $\Cc(G,\zeta)'\to C_c^\infty(G,\zeta)'$. Thus, we may identify tempered distributions with distributions, and a distribution $\Theta\in C_c^\infty(G,\zeta)'$ is tempered if and only if it extends to a continuous linear functional $\Theta:\Cc(G,\zeta)\to\CC$.

For $f:G(F)\to\CC$, we define the function $f^\zeta:G(F)\to\CC$ by
\[
f^\zeta(g)=\int_{\Zc}\zeta(z)f(zg)\dd{z},
\]
provided this integral converges absolutely. This defines open continuous surjections $C_c^\infty(G)\to C_c^\infty(G,\zeta)$, $C_c(G)\to C_c(G,\zeta)$, and $\Cc(G)\to\Cc(G,\zeta)$. Given a $\zeta$-equivariant distribution $D\in C_c^\infty(G,\zeta)'$ on $G(F)$, we may thus define a distribution $D^\zeta\in C_c^\infty(G)'$ by pullback: $D^\zeta(f)=D(f^\zeta)$. The distribution $D^\zeta$ is a Radon measure (resp. tempered distribution) on $G(F)$ if $D$ is. Moreover, the distribution $D^\zeta$ is $\zeta$-equivariant in the sense that $D^\zeta(zf)=\zeta(z)D^\zeta(f)$ for all $z\in\Zc$ and $f\in C_c^\infty(G)$.  Conversely, a distribution $\wt{D}\in C_c^\infty(G)'$ that is $\zeta$-equivariant descends to a distribution $D\in C_c^\infty(G,\zeta)'$, that is, a $\zeta$-equivariant distribution. In this way, one can identify $\zeta$-equivariant distributions on $G(F)$ with distributions on $G(F)$ that are $\zeta$-equivariant, and likewise for Radon measures and tempered distributions.

\section{Invariant Harmonic Analysis} \label{sec:invariant}

\subsection{Invariant distributions}
Let $G_\sr\subseteq G_\rs$ denote the open subvarieties of $G$ consisting of strongly regular (semisimple) elements and regular semisimple elements, respectively. Let
\[
\Gamma_\sr(G)\subseteq\Gamma_{\rs}(G)\subseteq\Gamma(G)
\]
denote the spaces of conjugacy classes of elements in $G_\sr(F)$, $G_\rs(F)$, and $G(F)$, respectively. Then $\Gamma_\sr(G)$ and $\Gamma_{\rs}(G)$ are open dense locally compact Hausdorff subspaces of $\Gamma(G)$, and $\Gamma_\sr(G)$ is naturally an $F$-analytic manifold. 

\subsubsection{Orbital integrals.}
Fix a central datum $(\Zc,\zeta)$ of $G$ and let $f\in C_c(G,\zeta)$. Let $\gamma\in G_\rs(F)$ and let $G_\gamma=C_G(\gamma)^\circ$, a maximal torus of $G$. The orbital integral of $f$ at $\gamma$ is defined by
\[
O_\gamma(f)=\int_{G_\gamma(F)\under G(F)}f(g^{-1}\gamma g)\dd{g}.
\]
The resulting linear functional $O_\gamma:C_c(G,\zeta)\to\CC$ is a $\zeta$-equivariant Radon measure on $G(F)$, it only depends on the conjugacy class $\gamma\in\Gamma_\rs(G)$, and the map $\gamma\mapsto O_\gamma$ on $\Gamma_\rs(G)$ is $\zeta^{-1}$-equivariant. Moreover, the $\zeta$-equivariant distribution $O_\gamma$ is tempered.

We write $D^G$ for the Weyl discriminant of $G$. For $\gamma\in G_\rs(F)$, if we write $T=G_\gamma$ we have
\[
D^G(\gamma)=\det(1-\Ad(\gamma)|_{\gk/\tk}).
\]
We recall the Weyl integration formula
\[
\int_{G(F)}f(g)\dd{g}=\sum_{\{T\}}|W_F(G,T)|^{-1}\int_{T(F)}|D^G(t)|O_t(f)\dd{t},
\]
where $f$ is an integrable function on $G(F)$ and $\{T\}$ runs over the set of $G(F)$-conjugacy classes of maximal tori of $G$. The Weyl integration formula makes it natural to define a Radon measure $\dd{\gamma}$ on $\Gamma_\rs(G)$ by
\[
\int_{\Gamma_\rs(G)}\alpha(\gamma)\dd{\gamma}=\sum_{\{T\}}|W_F(G,T)|^{-1}\int_{T(F)}\alpha(t)\dd{t}
\]
for all $\alpha\in C_c(\Gamma_\rs(G))$. The Weyl integration formula then becomes
\[
\int_{G(F)}f(g)\dd{g}=\int_{\Gamma_\rs(G)}|D^G(\gamma)|O_\gamma(f)\dd{\gamma}.
\]
Note that $\Gamma_{\sr}(G)$ has full measure in $\Gamma_{\rs}(G)$. 

For $f\in\Cc(G,\zeta)$ and $\gamma\in\Gamma_\rs(G)$, we write
\[
f_G(\gamma)=|D^G(\gamma)|^{1/2}O_\gamma(f)
\]
for the normalised orbital integral of $f$ on $\gamma$. The function $\gamma\mapsto O_\gamma(f)$ is continuous on $\Gamma_{\rs}(G)$ and smooth on $\Gamma_\sr(G)$. Thus, the same is true of the function $f_G$. Moreover, $f_G$ is locally bounded. We remark that since $\Gamma_{\sr}(G)$ is an open dense subset of full measure in $\Gamma_\rs(G)$, many definitions can be made either using $\Gamma_{\sr}(G)$ or $\Gamma_\rs(G)$.

We define the spaces of orbital integrals $\Ic(G,\zeta)$ and $\Ic_c(G,\zeta)$ by
\begin{align*}
\Ic_{(c)}(G,\zeta)&=\{f_G:\Gamma_\rs(G)\to\CC : f\in\Cc_{(c)}(G,\zeta)\} \\
&=\Cc_{(c)}(G,\zeta)/\Ann_{\Cc_{(c)}(G,\zeta)}(\{O_\gamma : \gamma\in\Gamma_\rs(G)\})
\end{align*}
with their natural quotient topologies. We define the space $\Ic_{c}(G,\zeta,K)$ of orbital integrals of elements of $C_c^\infty(G,\zeta,K)$ in an analogous way. This space does not depend on the choice of $K$. For non-archimedean $F$, this is because $C_c^\infty(G,\zeta,K)=C_c^\infty(G,\zeta)$. For archimedean $F$, it follows from the fact that all maximal compact subgroups of $G(F)$ are $G(F)$-conjugate. Thus, we will write $\Ic_{f}(G,\zeta)=\Ic_{c}(G,\zeta,K)$.

We recall that a maximal torus is said to be elliptic if $T/A_G$ is anisotropic, or equivalently if $A_T=A_G$. A semisimple element $\gamma\in G(F)$ is said to be elliptic if $\gamma\in T(F)$ for some elliptic maximal torus $T$ of $G$. We have open subvarieties $G_{\sr,\el}\subseteq G_{\rs,\el}\subseteq G$ and we write $\Gamma_{\sr,\el}(G)\subseteq\Gamma_{\rs,\el}(G)$ for the conjugacy classes of elements in $G_{\sr,\el}(F)$ and $G_{\rs,\el}(F)$, respectively. A function $f\in\Cc(G,\zeta)$ is said to be cuspidal if $f_G(\gamma)=0$ for all $\gamma\in\Gamma_\rs(G)\setminus\Gamma_{\rs,\el}(G)$. We denote the subspace of cuspidal functions in $\Cc_{(c)}(G,\zeta)$ by $\Cc_{(c),\cusp}(G,\zeta)$, and we denote its image in $\Ic_{(c)}(G,\zeta)$ by $\Ic_{(c),\cusp}(G,\zeta)$. We denote the subspace of cuspidal functions in $C_c^\infty(G,\zeta,K)$ by $C_{c,\cusp}^\infty(G,\zeta,K)$, and its image in $\Ic_f(G,\zeta)$ by $\Ic_{f,\cusp}(G,\zeta)$.

\subsubsection{Invariance.}
We refer to elements of $\Ic_c(G,\zeta)'$ (resp. $\Ic(G,\zeta)'$) as invariant ($\zeta$-equivariant) distributions (resp. invariant tempered ($\zeta$-equivariant) distributions). Note that we have a continuous linear injection $\Ic_c(G,\zeta)\to\Ic(G,\zeta)$ with dense image. Its transpose is a continuous linear injection $\Ic(G,\zeta)'\to\Ic_c(G,\zeta)'$, which enables us to identify each invariant tempered distribution with an invariant distribution.

We may identify $\Ic_c(G,\zeta)'$ with a subspace of $\Cc_{(c)}(G,\zeta)'$ via the transpose of the quotient map $\Cc_{(c)}(G,\zeta)\to\Ic_{(c)}(G,\zeta)$. As vector spaces, we have 
\begin{align*}
    \Ic_{(c)}(G,\zeta)'&=\Ann_{\Cc_{(c)}(G,\zeta)'}(\Ann_{\Cc_{(c)}(G,\zeta)}(\{O_\gamma : \gamma\in\Gamma_\rs(G)\})) \\
    &=\cl_{\Cc_{(c)}(G,\zeta)',\textrm{weak-}*}(\{O_\gamma : \gamma\in\Gamma_\rs(G)\}).
\end{align*}
That is, a distribution in $\Cc_{(c)}(G,\zeta)'$ belongs to $\Ic_{(c)}(G,\zeta)'$ if and only if it lies in the weak-$*$ closure in $\Cc_{(c)}(G,\zeta)'$ of the linear span of the set of regular semisimple orbital integrals of $G$. A locally integrable function $\Theta$ on $G(F)$ that is continuous on $G_\sr(F)$ defines an invariant distribution if and only if $\Theta$ is conjugation invariant on $G_\sr(F)$.

One can also define the notion of conjugation invariant distributions. For $y\in G(F)$ and $f\in\Cc_{(c)}(G,\zeta)$, we define $\pre{y}{f}\in\Cc_{(c)}(G,\zeta)$ by $\pre{y}{f}(x)=f(y^{-1}xy)$. This defines a left action of $G(F)$ on $\Cc_{(c)}(G,\zeta)$. The left action of $G(F)$ on $\Cc_{(c)}(G,\zeta)$ in turn leads to a right action of $G(F)$ on the associated space of distributions: for all $y\in G(F)$ and $u\in\Cc_{(c)}(G,\zeta)'$, we define $u^y\in\Cc_{(c)}(G,\zeta)'$ by $u^y(f)=u(\pre{y}{f})$.

Let $\Cc_{(c)}(G,\zeta)_G$ denote the quotient of $\Cc_{(c)}(G,\zeta)$ by the smallest closed subspace of $\Cc_{(c)}(G,\zeta)$ containing all functions of the form $\pre{y}{f}-f$ for $y\in G(F)$ and $f\in\Cc_{(c)}(G,\zeta)$. We define $(\Cc_c(G,\zeta)_G)'$ (resp. $(\Cc(G,\zeta)_G)'$) to be the space of conjugation invariant distributions (resp. tempered distributions). We have a natural injection $(\Cc_{(c)}(G,\zeta)_G)'\to\Cc_{(c)}(G,\zeta)'$, and its image is the subspace of distributions in $\Cc_{(c)}(G,\zeta)'$ such that $u^y=u$. 

Every orbital integral is conjugation invariant. Thus, if $f\in\Cc_{(c)}(G,\zeta)$ is annihilated by all conjugation invariant functions, then it is annihilated by all orbital integrals. It follows that the quotient $\Cc_{(c)}(G,\zeta)\to\Ic_{(c)}(G,\zeta)$ descends to a quotient $\Cc_{(c)}(G,\zeta)_G\to\Ic_{(c)}(G,\zeta)$. It is known that for $f\in\Cc_c(G)$, if $f$ is annihilated by all regular semisimple orbital integrals, then $f$ is annihilated by all conjugation invariant distributions. This was proved for real groups by Bouaziz \cite[Theorem 3.2.1]{Bouaziz} and for $p$-adic groups by Harish-Chandra \cite[Theorem 10]{HCAdmissible}. Thus, $\Cc_c(G)_G=\Ic_c(G)$, and conjugation invariant distributions are the same as invariant distributions. 

\subsubsection{Representations and characters.}
Let $\Pi(G)$ denote the set of equivalence classes of irreducible admissible representations of $G(F)$. We denote the central character of $\pi\in\Pi(G)$ by $\zeta_\pi$. Let $(\Zc,\zeta)$ be a central datum of $G(F)$. We denote by $\Pi(G,\zeta)$ the set of equivalence classes of irreducible admissible representations $\pi$ of $G(F)$ with $\Zc$-character $\zeta_\pi|_{\Zc}=\zeta$. We have the subsets
\[
\Pi_u(G,\zeta)\supseteq\Pi_\temp(G,\zeta)\supseteq\Pi_2(G,\zeta)
\]
of equivalence classes of irreducible unitary, tempered, and (relatively) square-integrable (or discrete series) representations, respectively. These are only non-empty if $\zeta$ is unitary. We denote the space of virtual representations of $G(F)$ with $\Zc$-character $\zeta$ by $D_\spec(G,\zeta)=\CC\Pi(G,\zeta)$, and we denote the subspace of virtual tempered representations by $D_\temp(G,\zeta)=\CC\Pi_\temp(G,\zeta)$.

Let $\pi$ be a finite-length admissible representation of $G(F)$ with $\Zc$-character $\zeta$. For $f\in C_c^\infty(G,\zeta)$, one can form the operator
\[
\pi(f)=\int_{G(F)/\Zc}f(g)\pi(g)\dd{g},
\]
which has a well-defined trace. The resulting $\zeta$-equivariant distribution $\Theta_\pi=\tr\pi:C_c^\infty(G,\zeta)\to\CC$ is the character of $\pi$. If $\pi$ is tempered (and $\zeta$ is unitary), then $\pi(f)$ is well-defined for $f\in\Cc(G,\zeta)$ and $\Theta_\pi$ is a tempered $\zeta$-equivariant distribution.

We identify a virtual representation $\pi=\sum_{i=1}^mc_i\pi_i\in D_\spec(G,\zeta)$ with its (virtual) character $\Theta_\pi=\sum_{i=1}^mc_i\Theta_{\pi_i}$. Thus, the space $D_\spec(G,\zeta)$ is identified with a subspace of the space of $\zeta$-equivariant distributions. Harish-Chandra's regularity theorem tells us that $\Theta_\pi$ is a smooth---even analytic if $F$ is archimedean---conjugation invariant $\zeta$-equivariant function on $G_\rs(F)$; that $\Theta_\pi\in L_\loc^1(G(F))$; and that the normalised character $|D^G(x)|^{1/2}\Theta_\pi(x)$ is locally bounded on $G(F)$.

A simple consequence of the Harish-Chandra regularity theorem and the Weyl integration formula is the following. Let $\pi\in\Pi(G,\zeta)$ (resp. $\pi\in\Pi_\temp(G,\zeta)$). If $f\in\Cc_c(G,\zeta)$ (resp. $f\in\Cc(G,\zeta)$), with $f_G(\gamma)=0$ for all $\gamma\in\Gamma_\sr(G)$, then $f_G(\pi)=0$. Therefore $\Theta_\pi\in\Ic_c(G,\zeta)'$ (resp. $\Theta_\pi\in\Ic(G,\zeta)'$). Consequently, we have $D_\spec(G,\zeta)\subseteq\Ic_c(G,\zeta)'$ and $D_\temp(G,\zeta)\subseteq\Ic(G,\zeta)'$.

\subsubsection{Parabolic descent.}
Let $(P,M)$ be a semi-standard parabolic pair. For $f\in\Cc_{(c)}(G,\zeta)$ (assuming $\zeta$ is unitary in the case of Schwartz functions), one defines $f^{(P)}:M(F)\to\CC$ by
\[
f^{(P)}(m)=\delta_P(m)^{1/2}\int_{N_P(F)}\int_{K}f(k^{-1}mnk)\dd{n}\dd{k}.
\]
Then $f^{(P)}\in\Cc_{(c)}(M,\zeta)$. Note that $f^{(P)}$ depends on the choice of $K$. This results in a continuous linear operator $\Cc_{(c)}(G,\zeta)\to\Cc_{(c)}(M,\zeta)$ called parabolic descent. If $\sigma$ is a finite-length admissible representation of $M(F)$ with $\Zc$-character $\zeta$, then
\[
\langle\Theta_\sigma, f^{(P)} \rangle=\langle \Theta_{I_{M,P}^G\sigma},f \rangle
\]
for all $f\in C_c^\infty(G,\zeta)$. If $\sigma$ is tempered, then this holds for all $f\in\Cc(G,\zeta)$. The parabolic descent map $f\mapsto f^{(P)}$ descends to a continuous map 
\begin{align*}
    \Ic_{(c)}(G,\zeta)&\longrightarrow\Ic_{(c)}(M,\zeta)^{W^G(M)} \\
    f_G&\longmapsto f_M
\end{align*}
which we also call parabolic descent. Let $\Gamma_{G\dash\rs}(M)$ denote the set of $G$-regular semisimple conjugacy classes in $M(F)$. For all $\gamma\in\Gamma_{G\dash\rs}(M)$ we have $f_M(\gamma)=f_G(\gamma)$. Consequently, although $f^{(P)}$ depends on the choice of $P$ and $K$, the function $f_M$ does not. For $f\in\Cc(G,\zeta)$, we have that $f$ is cuspidal if and only if $f_M=0$ for all $M\in\Lc^G(M_0)$ with $M\neq G$.

The transpose of parabolic descent $\Ic_{(c)}(G,\zeta)\to\Ic_{(c)}(M,\zeta)^{W^G(M)}$ is a continuous linear map $I_M^G:\Ic_{(c)}(M,\zeta)'/W^G(M)\to\Ic_{(c)}(G,\zeta)'$, which we call parabolic induction since it extends the parabolic induction of characters. If $f\in\Cc_\cusp(G,\zeta)$, then every invariant distribution that is parabolically induced from a proper Levi of $G$ annihilates $f$.

\subsubsection{The invariant Fourier transform.}
Let $(\Zc,\zeta)$ be a unitary central datum of $G$. Let $f\in\Cc(G,\zeta)$. The operator-valued Fourier transform of $f$ is the section of $\coprod_{\pi\in\Pi_\temp(G,\zeta)}\End(V_\pi)\to\Pi_\temp(G,\zeta)$ defined by $\pi\mapsto \pi(f)$. It is the subject of harmonic analysis on $G(F)$, upon which invariant harmonic analysis is built. Invariant harmonic analysis is the study of the invariant (or scalar-valued) Fourier transform of $f$, which is the function $f_G:\Pi_\temp(G,\zeta)\to\CC$ defined by $f_G(\pi)=\Theta_\pi(f)$.

We have the space of invariant Fourier transforms
\begin{align*}
    \wh{\Ic_{(c)}}(G,\zeta)&=\{f_G : f\in\Cc_{(c)}(G,\zeta)\} \\
    &=\Cc_{(c)}(G,\zeta)/\Ann_{\Cc_{(c)}(G,\zeta)}(\{\Theta_\pi : \pi\in\Pi_\temp(G,\zeta)\})
\end{align*}
with its natural quotient topology. The space $\wh{\Ic_{c,f}}(G,\zeta)=\wh{\Ic_c}(G,\zeta,K)$ of invariant Fourier transforms of elements of $C_c^\infty(G,\zeta,K)$ is defined in an analogous way, and $\wh{\Ic_{c,f}}(G,\zeta)$ does not depend on the choice of $K$. If $F$ is non-archimedean, we have $\wh{\Ic_{c,f}}(G,\zeta)=\wh{\Ic_c}(G,\zeta)$.

The invariant Fourier transform 
\[
\Ic_{(c)}(G,\zeta)\longrightarrow\wh{\Ic_{(c)}}(G,\zeta)
\]
is a surjective linear map by definition of $\wh{\Ic}_{(c)}(G,\zeta)$. Since tempered characters are invariant, it descends to a continuous surjective linear map
\[
\Fc:\Ic_{(c)}(G,\zeta)\longrightarrow\wh{\Ic_{(c)}}(G,\zeta).
\]
This restricts to a continuous surjective linear map
\[
\Fc:\Ic_{c,f}(G,\zeta)\longrightarrow\wh{\Ic_{c,f}}(G,\zeta).
\]

Let $f\in\Cc(G,\zeta)$. The invariant Fourier transform $f_G:\Pi_\temp(G,\zeta)\to\CC$ of $f$ has a unique extension to a $\CC$-linear form $f_G:D_\temp(G,\zeta)\to\CC$ on the space of virtual tempered representations $D_\temp(G,\zeta)$. The invariant Fourier transform is just the representation of this linear form $f_G$ with respect to the basis $\Pi_\temp(G,\zeta)$ of $D_\temp(G,\zeta)$. In \cite{ArthurElliptic}, Arthur defines a set of virtual tempered representations $\{\pi_{\tau}\}_{\tau\in T_\temp(G,\zeta)}$ parametrised by a set $T_\temp(G,\zeta)$ of $W_0^G$-orbits of equivalence classes of certain triples. These virtual representations arise naturally in the theory of the $R$-group. There is a natural action of $\CC^1$ on $T_\temp(G,\zeta)$, and the assignment $\tau\mapsto\pi_{\tau}$ is equivariant with respect to the action. That is, $\pi_{z\tau}=z\pi_{\tau}$ for all $z\in\CC^1$ and $\tau\in T_\temp(G,\zeta)$. A choice of representative $\tau$ for each $\CC^1$-orbit $\CC^1\tau$ in $T_\temp(G,\zeta)/\CC^1$ gives a basis $\{\pi_{\tau}\}_{\CC^1\tau\in T_\temp(G,\zeta)/\CC^1}$ of $\CC\Pi_\temp(G)$. For some purposes, the virtual tempered representations $\{\pi_{\tau}\}_{\tau\in\wt{T}_\temp(G,\zeta)}$ are more natural than $\Pi_\temp(G,\zeta)$ and can be regarded as the spectral objects that play the role dual to conjugacy classes in invariant harmonic analysis (cf. \cite{ArthurFourier}). It is useful to view the Fourier transform of $f$ as the $\CC^1$-equivariant function $f_G:T_\temp(G,\zeta)\to\CC$ defined by $f_G(\tau)=f_G(\pi_{\tau})$. This representation of the Fourier transform allows for more natural formulations of Paley--Wiener theorems. In the following subsections, we recall the theory of the $R$-group, Arthur's virtual tempered representations, and the invariant Paley--Wiener theorems that we will stabilise in the next section.

\subsection{$R$-groups}
Let $M\in\Lc^G(M_0)$ and let $\sigma$ be a Hilbert space representation of $M(F)$. One can form the normalised parabolically induced representation Hilbert space representation $I_P(\sigma)=I_P^G(\sigma)=I_{M,P}^G(\sigma)$, defined to be the space of equivalence classes of measurable functions $f:G(F)\to V_\sigma$ such that $f(pg)=\sigma(p)\delta_P(p)^{1/2}f(g)$ for all $p\in P(F)$ and $g\in G(F)$ and
\[
\int_K\|f(k)\|^2\dd{k}<\infty.
\]
By restricting the functions in $I_P^G(\sigma)$ to $K$, we obtain the so-called compact model of $I_P(\sigma)$, which is used in the construction of the standard intertwining operators.

Let $\sigma\in\Pi_2(M)$. Then $I_P(\sigma)$ decomposes as a finite direct sum of irreducible tempered representations of $G(F)$ and the theory of the $R$-group introduced by Knapp and Stein gives a parametrisation of the irreducible summands of $I_P(\sigma)$. We will follow the presentation in \cite[\S1.9--1.11]{MWLocal} with some minor modifications.

For any $P,Q\in\Pc(M)$, we have a normalised standard intertwining operator 
\[
R_{Q|P}(\sigma):I_P(\sigma)\longrightarrow I_Q(\sigma),
\]
which is defined by meromorphically continuing the standard intertwining operator defined by a certain integral, and then renormalising it. The renormalisation involves a choice of a scalar normalising factor, which is a meromorphic function 
\begin{align*}
\ak_{M,\CC}^*&\longrightarrow\CC \\
\lambda&\longmapsto r_{Q|P}(\sigma_\lambda)
\end{align*}
They can be chosen to satisfy several conditions, including the conditions (R1)--(R8) of \cite[Theorem 2.1]{ArthurFourier}.

We write $\wt{w}\in N_{G(F)}(M)$ for an element with image $w\in W^G(M)$. We have a unitary intertwining operator 
\[
\wt{w}:I_P(\sigma)\longrightarrow I_{w\cdot P}(\wt{w}\cdot\sigma)
\]
defined by $\phi\mapsto\phi(\wt{w}^{-1}\cdot)$. Let $N_{G(F)}(M)_\sigma=\{\wt{w}\in N_{G(F)}(M) : \wt{w}\cdot\sigma\cong\sigma\}$. Suppose that $\wt{w}\in N_{{G(F)}}(M)_\sigma$. For each unitary intertwining operator $A:\wt{w}\cdot\sigma\to\sigma$, we have a unitary intertwining operator 
\[
I_{w\cdot P}(A):I_{w\cdot P}(\wt{w}\cdot\sigma)\longrightarrow I_{w\cdot P}(\sigma)
\]
since $I_{w\cdot P}$ is a functor.

We define the unitary intertwining operator $R_P(A,\wt{w}):I_{P}(\sigma)\to I_{P}(\sigma)$ to be the composition $R_P(A,\wt{w})=R_{P|w\cdot P}(\sigma)\circ I_{w\cdot P}(A)\circ\wt{w}$:
\[
\begin{tikzcd}
    R_P(A,\wt{w}):I_{P}(\sigma) \arrow{r}{\wt{w}} & I_{w\cdot P}(\wt{w}\cdot\sigma) \arrow{r}{I_{w\cdot P}(A)} & I_{w\cdot P}(\sigma) \arrow{r}{R_{P|w\cdot P}(\sigma)} & I_{P}(\sigma).
\end{tikzcd}
\]

Define $\Nc^G(\sigma)$ to be the set of pairs $(A,\wt{w})$ where $\wt{w}\in N_{G(F)}(M)_{\sigma}$ and $A:\wt{w}\cdot\sigma\to\sigma$ is a unitary intertwining operator. The set $\Nc^G(\sigma)$ is naturally a group with multiplication defined by $(A_1,\wt{w_1})(A_2,\wt{w_2})=(A_1A_2,\wt{w_1}\wt{w_2})$. The map $M(F)\to\Nc^G(\sigma), m\mapsto (\sigma(m),m)$ is an injective homomorphism with normal image, and we denote the resulting quotient group by $\Wc^G(\sigma)=\Nc^G(\sigma)/M(F)$. Note that for each $(A,\wt{w})\in\Nc^G(\sigma)$, the intertwining operator $I_{w\cdot P}(A)\circ \wt{w}:I_P(\sigma)\to I_{w\cdot P}(\sigma)$, and thus also $R_P(A,\wt{w})$, only depends on the image of $(A,\wt{w})$ in $\Wc^G(\sigma)$. The map $R_P$ from $\Wc^G(\sigma)$ to the group of unitary intertwining operators on $I_P(\sigma)$ is a homomorphism.

For each $z\in\CC^1$, we have that $(z,1)$ lies in the centre of $\Nc^G(\sigma)$. Note that the maps $\CC^1\to\Nc^G(\sigma)$ and $\CC^1\to\Wc^G(\sigma)$ are injective homomorphisms. Let 
\begin{align*}
W^G(\sigma)&=\{w\in W^G(M) : w\cdot\sigma\cong\sigma\} \\
&\cong\{w\in W_0^G : w\cdot M=M, w\cdot\sigma\cong\sigma\}/W_0^M.
\end{align*}
We have a canonical surjective homomorphism $\Wc^G(\sigma)\to W^G(\sigma)$ with kernel $\CC^1$. That is, we have the central extension
\[
1\longrightarrow\CC^1\longrightarrow\Wc^G(\sigma)\longrightarrow W^G(\sigma)\longrightarrow1.
\]
Consequently, $\Wc^G(\sigma)$ is naturally a compact group. Moreover, $R_P$ is a (continuous) unitary representation of $\Wc^G(\sigma)$ on $I_P(V_\sigma)$ through which $\CC^1$ acts by multiplication.

The definition of $R_P$ depends in a simple manner on the choice of normalising factors and the parabolic subgroup $P\in\Pc^G(M)$. Consider a different choice of normalising factors and the resulting representation $\ul{R}_P$ of $\Wc^G(\sigma)$. There exists a unitary character $\chi$ of $W^G(\sigma)$ such that for all $(A,\wt{w})\in\Wc^G(\sigma)$ we have $\ul{R}_P(A,\wt{w})=\chi(A,\wt{w})R_P(A,\wt{w})$. We can define an automorphism $\alpha_\chi$ of $\Wc^G(\sigma)$ by $\alpha_\chi(A,\wt{w})=(\chi(A,\wt{w})A,\wt{w})$. Then $\ul{R}_P=R_P\circ\alpha_\chi$. If $P'\in\Pc^G(M)$ is another choice of parabolic, then $R_{P|P'}(\sigma)\circ R_{P'}(A,\wt{w})=R_P(A,\wt{w})\circ R_{P|P'}(\sigma)$ for all $w\in\Wc^G(\sigma)$, so $R_{P'}$ is unitarily equivalent to $R_P$. Consequently, the kernel $W^G(\sigma)_0$ of $R_P$ is independent of the choice of normalising factors and parabolic $P\in\Pc^G(M)$. 

The subgroup $W^G(\sigma)_0$ of $\Wc^G(\sigma)$ injects into $W^G(\sigma)$, and we also denote its image by $W^G(\sigma)_0$. We define $\Rc^G(\sigma)=\Wc^G(\sigma)/W^G(\sigma)_0$. The Knapp--Stein $R$-group of $\sigma$ is defined to be $R^G(\sigma)=W^G(\sigma)/W^G(\sigma)_0$. We thus have a central extension
\[
1\longrightarrow\CC^1\longrightarrow\Rc^G(\sigma)\longrightarrow R^G(\sigma)\longrightarrow1.
\]
Since $\CC^1$ acts by multiplication through $R_P$, the representation $R_P$ of $\Wc^G(\sigma)$ descends to a representation $R_P$ of $\Rc^G(\sigma)$ through which $\CC^1$ acts by multiplication.

We denote by $\Pi(\Rc^G(\sigma),\id_{\CC^1})$ the set of irreducible representations of $\Rc^G(\sigma)$ through which $\CC^1$ acts by multiplication. Let $\Rc_P$ denote the representation of $\Rc^G(\sigma)\times G(F)$ on $I_P(V_\sigma)$ defined by $\Rc_P(\wt{r},x)=R_P(\wt{r})I_P(\sigma,x)$.

Harish-Chandra's commuting algebra theorem (\cite[Theorem 38.1]{HCHAIII} for real groups and \cite[Theorem 5.5.3.2]{Silberger} for $p$-adic groups) states that the operators $\{R_P(\wt{r}) : \wt{r}\in\Rc^G(\sigma)\}$ span $\End(I_P(\sigma))$. The dimension theorem (due to Knapp--Stein for real groups \cite[Theorem 2]{KnappSteinSingularIV}, \cite[Theorem 13.4]{KnappSteinInterII} and Silberger for $p$-adic groups \cite{SilbergerDimension}) states that the dimension of $\End(I_P(\sigma))$ is the cardinality of $\Rc^G(\sigma)/\CC^1=R^G(\sigma)$, and thus
\[
\End(I_P(\sigma))=\bigoplus_{\CC^1\wt{r}\in\Rc^G(\sigma)/\CC^1}\CC R_P(\wt{r}).
\]
As a corollary, one obtains the main theorem of the theory of $R$-groups, namely that that there is a bijection $\rho\mapsto\pi_\rho$ between $\Pi(\Rc^G(\sigma),\id_{\CC^1})$ and the set $\Pi_\sigma(G)$ of irreducible summands of $I_P^G(\sigma)$ characterised by the decomposition
\[
\Rc\cong\bigoplus_{\rho\in\Pi(\Rc^G(\sigma),\id_{\CC^1})}\rho\boxtimes\pi_\rho.
\]

If $P'\in\Pc^G(M)$, then $R_{P|P'}(\sigma)$ intertwines the representation $\Rc_{P'}$ with $\Rc_P$, so the bijection $\rho\mapsto\pi_\rho$ is independent of the choice of $P\in\Pc^G(M)$. Consider a different choice of normalising factors and the resulting representation $\ul{R}_P$, which is of the form $\ul{R}_P=\chi R_P$ for a unitary character $\chi$ of $W^G(\sigma)$ as above. The resulting representation $\ul{\Rc}_P\defeq\ul{R}_P I_P(\sigma)=\chi\Rc_P$ of $\Rc^G(\sigma)\times G(F)$ decomposes as
\[
\bigoplus_{\rho\in\Pi(\Rc^G(\sigma),\id_{\CC^1})}\chi\rho\boxtimes\pi_\rho,
\]
and therefore determines the bijection $\rho\mapsto\pi_{\chi^{-1}\rho}$.

The group $\Nc^G(\sigma)$ and the related objects depend on the representation $\sigma$ itself, not just its isomorphism class. This point is often glossed over in the literature. However, suppose that $T:\sigma\to\sigma'$ is an isomorphism of unitary representations. Then we obtain an isomorphism
\begin{align*}
T:\Nc^G(\sigma)&\longrightarrow\Nc^G(\sigma') \\
(A,\wt{w})&\longmapsto(T\circ A\circ T^{-1},\wt{w})
\end{align*}
that induces isomorphisms $\Wc^G(\sigma)\to\Wc^G(\sigma')$, $W^G(\sigma)_0\to W^G(\sigma')_0$, $\Rc^G(\sigma)\to\Rc^G(\sigma')$, and $R^G(\sigma)\to R^G(\sigma')$, which we also denote by $T$. Moreover, we have the following commutative diagram
\[
\begin{tikzcd}
    \Rc^G(\sigma) \arrow{r}{T} \arrow{d}{R_P} & \Rc^G(\sigma') \arrow{d}{R_P} \\
    \Aut(I_P(V_\sigma)) \arrow{r}{I_P(T)} & \Aut(I_P(V_{\sigma'}))
\end{tikzcd}
\]

\subsection{Arthur's virtual tempered representations} 
Consider the set of triples $\tau=(M,\sigma,\wt{r})$, where $M\in\Lc^G(M_0)$, $\sigma$ is a discrete series representation of $M(F)$, and $\wt{r}\in\Rc^G(\sigma)$. We define two triples $(M,\sigma,\wt{r})$ and $(M,\sigma',\wt{r}')$ to be equivalent if there exists a unitary isomorphism $T:\sigma\to\sigma'$ such that $T(\wt{r})=\wt{r}'$.

Let $g\in N_{G(F)}(M_0)$ and consider $g\cdot M\in \Lc^G(M_0)$ and $g\cdot\sigma$. We have an isomorphism
\begin{align*}
    \Nc^G(\sigma)&\longrightarrow\Nc^G(g\cdot\sigma) \\
    (A,\wt{w})&\longmapsto g\cdot(A,\wt{w})=(A,g\wt{w}g^{-1})
\end{align*}
and this determines isomorphisms $\Wc^G(\sigma)\to\Wc^G(g\cdot\sigma)$, $W^G(\sigma)\to W^G(g\cdot\sigma)$, $\Rc^G(\sigma)\to\Rc^G(g\cdot\sigma)$, and $R^G(\sigma)\to R^G(g\cdot\sigma)$. Thus, the group $N_{G(F)}(M_0)$ acts on the set of triples $\tau=(M,\sigma,\wt{r})$ by $g\cdot\tau=(g\cdot M,g\cdot\sigma,g\cdot\wt{r})$. This descends to an action on the set of equivalence classes of triples, and $M_0(F)$ acts trivially. Thus, we obtain an action of $W_0^G$ on the set of equivalence classes of triples.

Let $\tau=(M,\sigma,\wt{r})$ be a triple and recall the representation $\Rc$ from the previous subsection, with its decomposition
\[
\Rc\cong\bigoplus_{\rho\in\Pi(\Rc^G(\sigma),\id_{\CC^1})}\rho\boxtimes\pi_\rho.
\]
Define the virtual tempered representation
\[
\pi_{\tau}=\sum_{\rho\in\Pi(\Rc^G(\sigma),\id_{\CC^1})}\tr(\rho(\wt{r}))\pi_\rho.
\]
The invariant tempered distribution character of $\pi_{\tau}$ is given by
\[
\Theta_{\tau}(f)=\tr(R_P(\wt{r})I_P(f))=\sum_{\rho\in\Pi(\Rc^G(\sigma),\id_{\CC^1})}\tr(\rho(\wt{r}))\Theta_{\pi_\rho}(f), \quad f\in\Cc(G).
\]

The irreducible tempered representations $\pi_\rho$ all have the central character $\zeta_\sigma$. We define $\zeta_\tau=\zeta_\sigma$ and call it the central character of $\tau$. In particular, $\tau$ has a well-defined $\Zc$-character, namely the restriction $\zeta_\tau|_\Zc$, and $\pi_\tau\in D_\temp(G,\zeta_\tau|_{\Zc})$.

In the archimedean case, the representation $I_P(\sigma)$ has an infinitesimal character obtained from $\mu_\sigma$ in the manner described in our review of parabolic induction. Thus, the irreducible tempered representations $\pi_\rho$ all have the same infinitesimal character, which we call the infinitesimal character of $\tau$ and denote by $\mu_\tau$.

The virtual tempered representation $\pi_{\tau}$ only depends on the $W_0^G$-orbit of the equivalence class of $\tau$. Moreover, $\pi_{\tau}$ does not depend on the choice of parabolic $P\in\Pc(M)$, but it does depend on the choice of normalising factors. A different choice of normalising factors results in the virtual tempered representation $\chi(\wt{r})\pi_{\tau}$ for some unitary character $\chi$ of $W^G(\sigma)$. Thus, $\pi_{\tau}$ is uniquely determined up to multiplication by an element of $\CC^1$.

The group $\CC^1$ acts on the set of triples by $z(M,\sigma,z\wt{r})=(M,\sigma,z\wt{r})$, and this descends to an action on equivalence classes that commutes with the action of $W_0^G$. The map $\tau\mapsto\pi_{\tau}$ is equivariant with respect to the action of $\CC^1$ action: $\pi_{z\tau}=z\pi_\tau$. 

Consider an equivalence class of triples $\tau$. It can happen that there exists a non-trivial $z\in\CC^1$ such that $z\tau=w\cdot\tau$ for some $w\in W_0^G$. When this happens, we have $z\pi_{\tau}=\pi_{z\tau}=\pi_{\tau}$, and thus $\pi_{\tau}=0$. One says that $\tau$ is essential if this does not happen. The set of essential equivalence classes of triples is stable under the actions of $W_0^G$ and $\CC^1$. We define $\wt{T}_\temp(G,\zeta)$ to be the set of essential equivalence classes of triples $(M,\sigma,\wt{r})$ with $\Zc$-character $\zeta$, and we define $T_\temp(G,\zeta)=\wt{T}_\temp(G,\zeta)/W_0^G$. When $\Zc$ is trivial, we omit $\zeta$ from the notation. Arthur proved the following (cf. \cite[Proposition 2.9]{MWLocal}).

\begin{proposition}
    For each $\tau\in T_\temp(G,\zeta)$, we have $\pi_\tau\neq0$, and the space of virtual tempered representations of $G(F)$ is 
    \[
    D_\temp(G,\zeta)=\bigoplus_{\CC^1\tau\in T_\temp(G,\zeta)/\CC^1}\CC\pi_\tau.
    \]
\end{proposition}

In particular, the $\CC^1$-equivariant map $T_\temp(G,\zeta)\to D_\temp(G,\zeta), \tau\mapsto\pi_\tau$ is injective. We will identify $\tau$ with $\pi_\tau$, and thus also $\Theta_\tau$.

\subsubsection{Elliptic virtual tempered representations.}
Suppose that $L\in\Lc^G(M_0)$. We can relate $T_\temp(L,\zeta)$ and $T_\temp(G,\zeta)$ for $L\in\Lc^G(M_0)$ in the following way. If $M\in\Lc^L(M_0)$ and $\sigma$ is a discrete series representation of $M(F)$, then there is a natural injective homomorphism $\Rc^L(\sigma)\to\Rc^G(\sigma)$. This gives rise to an embedding
\[
\iota_L^G:\wt{T}_\temp(L,\zeta)\longrightarrow\wt{T}_\temp(G,\zeta)
\]
that is equivariant with respect to the action of $\CC^1$ and the action of $W_0^L\subseteq W_0^G$.
This map descends to a $\CC^1$-equivariant map
\[
\iota_L^G:T_\temp(L,\zeta)\longrightarrow T_\temp(G,\zeta).
\]
The group $W^G(L)$ acts on $T_\temp(L,\zeta)$ and this map is the quotient for this action. This map is compatible with parabolic induction in the following sense. If $\tau\in T_\temp(L,\zeta)$, then 
\[
I_L^G(\pi_\tau)=\pi_{\iota_L^G(\tau)}.
\]

We say that $\tau\in \wt{T}_\temp(G,\zeta)$ is elliptic if it does not lie in the image of $\iota_L^G:\wt{T}_\temp(L,\zeta)\to\wt{T}_\temp(G,\zeta)$ for any $L\neq G$. Let $\wt{T}_\el(G,\zeta)$ denote the set of elliptic elements in $\wt{T}_\temp(G,\zeta)$ and let $T_\el(G,\zeta)=\wt{T}_\el(G,\zeta)/W_0^G$. We can also define elliptic triples as follows. Let $\tau=(M,\sigma,\wt{r})\in\wt{T}_\temp(G,\zeta)$ and let $r\in R^G(\sigma)$ be the image of $\wt{r}$ in $R^G(\sigma)$. We define $W_\reg^G(\sigma)$ to be the set of $w\in W^G(\sigma)$ such that the space $\ak_M^w$ of $w$-invariants in $\ak_M$ is equal to $\ak_G$. Then $\tau$ is elliptic if and only if $W^G(\sigma)_0=1$ (in which case $R^G(\sigma)=W^G(\sigma)$, so $r\in W^G(\sigma)$) and $r\in W_\reg^G(\sigma)$ (i.e. $\ak_M^r=\ak_G$). We have an injection $\Pi_2(G,\zeta)\to T_\el(G,\zeta)$ defined by $\pi\mapsto(G,\pi,1)$, and we identify $\Pi_2(G,\zeta)$ with its image in $T_\el(G,\zeta)$.

We have
\[
\wt{T}_\temp(G,\zeta)=\coprod_{L\in\Lc^G(M_0)}\wt{T}_\el(L,\zeta)
\]
and
\[
T_\temp(G,\zeta)=\Bigg(\coprod_{L\in\Lc^G(M_0)}T_\el(L,\zeta)\Bigg)^{W_0^G}=\coprod_{L\in\Lc^G(M_0)/W_0^G}T_\el(L,\zeta)/W^G(L).
\]
We define $D_\el(G,\zeta)$ to be the subspace of $D_\temp(G,\zeta)$ generated by $T_\el(G,\zeta)$, that is,
\[
D_\el(G,\zeta)=\bigoplus_{\CC^1\tau\in T_\el(G,\zeta)/\CC^1}\CC\tau.
\]

\subsubsection{The space of Arthur's virtual tempered representations.}
For $\lambda\in i\ak_G^*$, the map
\begin{align*}
    \Nc^G(\sigma)&\longrightarrow\Nc^G(\sigma_\lambda) \\
    (A,n)&\longmapsto (Ae^{\langle\lambda,H_G(n)\rangle}, n)
\end{align*}
is an isomorphism. It descends to an isomorphism $\Rc^G(\sigma)\to\Rc^G(\sigma_\lambda),\wt{r}\mapsto\wt{r}_\lambda$ compatible with $R(\sigma)=R(\sigma_\lambda)$. We obtain an action of $i\ak_G^*$ on the set of triples $\tau=(M,\sigma,\wt{r})$ defined by 
\[
(\lambda,\tau)\longmapsto \tau_\lambda=(M,\sigma_\lambda,\wt{r}_\lambda),
\]
for all $\lambda\in i\ak_G^*$. The set $\wt{T}_\temp(G)$ of equivalence classes of triples and $\wt{T}_\el(G)$ are stable under the action of $i\ak_G^*$. Moreover, for $\tau\in T_\temp(G)$ and $\lambda\in i\ak_G^*$, we have $\pi_{\tau_\lambda}=(\pi_\tau)_\lambda$.

We have encountered three actions on $\wt{T}_\el(G)$, namely the actions of $i\ak_{G}^*$, $W_0^G$, and $\CC^1$. Let us summarise these here. If $\tau=[(M,\sigma,(A,n))]\in\wt{T}_\el(G)$, $\lambda\in i\ak_{G}^*$, $w\in W_0^G$, and $z\in\CC^1$, then
\begin{align*}
    \tau_\lambda&=[(M,|\cdot|_M^\lambda\sigma, (|n|_G^\lambda A,n))] \\
    w\cdot\tau&=[(wMw^{-1},\sigma\circ\Int(w^{-1}),(A,wnw^{-1}))] \\
    z\cdot\tau&=[(M,\sigma,(zA,n))].
\end{align*}
Any two of the above three actions commute with each other. We thus have commuting actions of $i\ak_{G}^*$ and $\CC^1$ on $T_\el(G)$. The action of $\ak_{G,\CC}^*$ on $\Pi(G)$ extends to a linear action on $D_\spec(G)$, and $i\ak_G^*$ preserves $D_\temp(G)$ and $D_\el(G)$. The injection $T_\el(G)\to D_\el(G)$ is equivariant with respect to the actions of $i\ak_G^*$ and $\CC^1$.

We write $T_\el(G)_\CC$ for the contracted product
\[
T_\el(G)_\CC=T_\el(G)\times^{i\ak_G^*}\ak_{G,\CC}^*,
\]
the quotient of $T_\el(G)\times\ak_{G,\CC}^*$ by the relation $(\tau_{\lambda_0},\lambda)\sim(\tau,\lambda+\lambda_0)$, where $\lambda_0\in i\ak_G^*$. The injection $T_\el(G)\to D_\el(G)$ extends to an injection $T_\el(G)_\CC\to D_\spec(G)$ that is equivariant with respect to the actions of $\ak_{G,\CC}^*$ and $\CC^1$.

Let $\tau\in T_\el(L)_\CC$. We denote the isotropy subgroup of $\tau$ in $\ak_{L,\CC}^*$ by $\ak_{L,\tau}^\vee$. We have
\[
\ak_{L,F}^\vee\subseteq \ak_{L,\tau}^\vee\subseteq \wt{\ak}_{L,F}^\vee\,.
\]
Thus, if $F$ is archimedean we have $\ak_{L,\tau}^\vee=0$, and if $F$ is non-archimedean we have that $\ak_{L,\tau}^\vee$ is a full lattice in $i\ak_L^*$. The set $T_\el(L)$ is naturally a smooth manifold with uncountably many connected components $i\ak_{L}^*\cdot\tau=i\ak_{L}^*/\ak_{L,\tau}^\vee$, which are Euclidean spaces if $F$ is archimedean and compact tori if $F$ is non-archimedean. The manifold $T_\el(L)$ has uncountably many components because of the action of $\CC^1$. Moreover, $T_\el(L)_\CC$ is the complexification of $T_\el(L)$ with connected components $\ak_{L,\CC}^*\cdot\tau=\ak_{L,\CC}^*/\ak_{L,\tau}^\vee$. Since
\[
T_\temp(G)=\coprod_{L\in\Lc^G(M_0)/W_0^G}T_\el(L)/W^G(L),
\]
we have that $T_\temp(G)$ is naturally a topological space. We say that a function on $T_\temp(G)$ is smooth if its pullback to each $T_\el(L)$ is smooth. 

\subsubsection{The elliptic inner product.}
We recall the elliptic inner product from \cite[7.3]{MWLocal}. Consider the restriction of our Radon measure $\dd{\gamma}$ on $\Gamma_{\rs}(G)$ to the open subspace $\Gamma_{\rs,\el}(G)$. This Radon measure and our Haar measure on $A_G(F)$ determine a Radon measure on $\Gamma_{\rs,\el}(G)/A_G(F)$. Identifying $\Gamma_{\rs,\el}(G)/A_G(F)=\Gamma_{\rs,\el}(G/A_G)$, this Radon measure is the restriction of the Radon measure on $\Gamma_{\rs}(G/A_G)$ determined by our choices of measures. 

For each $\gamma\in\Gamma_{\rs,\el}(G)$, define $m(\gamma)=\vol(G_\gamma(F)/A_G(F))$. We have a canonical Radon measure $\dd_\el{\gamma}$ on $\Gamma_{\rs,\el}(G)/A_G(F)$ defined as follows. For $\alpha\in C_c(\Gamma_{\rs,\el}(G)/A_G(F))$, we define
\begin{align*}
&\int_{\Gamma_{\rs,\el}(G)/A_G(F)}\alpha(\gamma)\dd_\el{\gamma} \\
&=\int_{\Gamma_{\rs,\el}(G)/A_G(F)}\alpha(\gamma)m(\gamma)^{-1}\dd{\gamma} \\
&=\sum_{\{T\}}|W_F(G,T)|^{-1}\vol(T(F)/A_G(F))^{-1}\int_{T(F)/A_G(F)}\alpha(t)\dd{t},
\end{align*}
where $\{T\}$ runs over the conjugacy classes of elliptic maximal tori of $G$. This integral does not depend on any choice of measure.

Let $\zeta$ be a unitary character of $A_G(F)$. For $\zeta$-equivariant functions $\alpha,\alpha':\Gamma_{\rs,\el}(G)\to\CC$, we define
\[
\langle \alpha,\alpha'\rangle_\el=\int_{\Gamma_{\rs,\el}(G)/A_G(F)} \alpha(\gamma)\ol{\alpha'(\gamma)} \dd_\el{\gamma},
\]
provided the integral converges absolutely. For $\pi,\pi'\in D_\el(G,\zeta)$, we define
\[
\langle \pi,\pi'\rangle_\el = \langle|D^G|^{1/2}\Theta_\pi,|D^G|^{1/2}\Theta_{\pi'}\rangle_\el.
\]
For $\tau=(M,\sigma,\wt{r})\in T_\el(G)$, one defines $\iota(\tau)=\big|\det(1-\wt{r})|_{\ak_{M}^G}\big|^{-1}$. We have the following orthogonality relations due to Arthur \cite[Theorem 7.3]{MWLocal}.

\begin{theorem}\label{OrthogonalityRelations}
    Let $\tau,\tau'\in T_\el(G,\zeta)$. We have 
    \[
    \langle\tau,\tau'\rangle_\el=|W^G(\tau)|\iota(\tau)^{-1}\delta_{\tau,\tau'}.
    \]
\end{theorem}

Thus, $\langle\cdot,\cdot\rangle_\el$ is an inner product on $D_\el(G,\zeta)$ and the decomposition
\[
D_\el(G,\zeta)=\bigoplus_{\CC^1\tau\in T_\el(G,\zeta)/\CC^1}\CC\tau
\]
is an orthogonal direct sum. It follows from the above theorem that $\{\|\tau\|_\el:\tau\in T_\el(G)\}$ is bounded. Note that when $\tau=\pi\in\Pi_2(G)$, that is, $\tau=(G,\pi,1)$, we have $\|\pi\|_\el=1$.

We have $D_\el(G)=\bigoplus_\zeta D_\el(G,\zeta)$. Thus, we obtain an inner product $\langle\cdot,\cdot\rangle_\el$ on $D_\el(G)$ as the orthogonal direct sum of the inner products on the summands. This inner product is called the elliptic inner product.

\subsection{Paley--Wiener and Schwartz spaces} \label{sec:spaces}

In this subsection we define abstract Paley--Wiener spaces and Schwartz spaces of the sort needed for the invariant and stable harmonic analysis. See \cite[Ch. IV]{MWI} for a similar presentation of abstract Paley--Wiener spaces suitable for invariant and stable harmonic analysis on a real group.

Let $V$ be a Euclidean space and let $\Lambda=\{\Lambda_e\}_{e\in E}$ be a countable family of one of the following types:
\begin{enumerate}
    \item archimedean type: for each $e\in E$, $\Lambda_e=iV^*$ equipped with a non-negative real number $\|e\|$.
    \item non-archimedean type: for each $e\in E$, $\Lambda_e=iV^*/\Gamma_e^\vee$, where $\Gamma_e$ is a lattice in $V$ and $\Gamma_e^\vee=\Hom(\Gamma_e,2\pi i\ZZ)$. Thus, $\Lambda_e$ is a compact torus of dimension $\dim V$.
\end{enumerate}
The terminology comes from the types of spaces that appear in Paley--Wiener theorems for $G$ when $F$ is archimedean or non-archimedean. We will also use $\Lambda$ to denote the disjoint union $\Lambda=\coprod_{e\in E}\Lambda_e$. The function spaces we define will be spaces of certain smooth functions on $\Lambda$. 

Note that for each $e\in E$, the space $\Lambda_e$ has a natural complexification $\Lambda_{e,\CC}$, which is $V_{\CC}^*$ in the archimedean case and $V_{\CC}^*/\Gamma_e^\vee$ in the non-archimedean case. Thus, the space $\Lambda$ has a natural complexification $\Lambda_\CC=\coprod_{e\in E}\Lambda_{e,\CC}$. For each $e\in E$, we extend the inner product on $V$ to a Hermitian inner product on $V_{\CC}$. Note that a function $\varphi:\Lambda\to\CC$ can be identified with a family of functions $\{\varphi_e\}_{e\in E}$.

We define the Paley--Wiener space $PW(\Lambda)$ on $\Lambda$ to be the vector space of smooth functions $\varphi:\Lambda\to\CC$ such that the following hold.
\begin{enumerate}
    \item In the non-archimedean case we require that $\varphi$ is supported on finitely many connected components $\Lambda_e$.
    \item $\varphi$ extends to an entire function on $\Lambda_\CC$.
    \item $\varphi$ satisfies a growth condition:
    \begin{enumerate}
        \item in the archimedean case we require that there exists $r>0$ such that for all $N\in\ZZ_{>0}$ we have
        \[
        \|\varphi\|_{r,N}\defeq\sup_{e\in E, \lambda\in\Lambda_{e,\CC}} |\varphi(\lambda)| (1+\|e\|+\|\lambda\|)^{N}e^{-r\|\Re(\lambda)\|}
        \]
        is finite;
        \item In the non-archimedean case we require that there exists $r>0$ such that
        \[
        \|\varphi\|_r\defeq\sup_{\lambda\in\Lambda_\CC}|\varphi(\lambda)|e^{-r\|\Re(\lambda)\|}
        \]
        is finite.
    \end{enumerate}
\end{enumerate}
We now define a locally convex topology on $PW(\Lambda)$. Suppose first that we are in the archimedean case. 
For each $r>0$, we define $PW^r(\Lambda)$ to be the subspace of all $\varphi\in PW(\Lambda)$ such that $\|\varphi\|_{r,N}<\infty$ for all $N\in\ZZ_{>0}$, and we give $PW^r(\Lambda)$ the topology defined by the family of norms $\|\cdot\|_{r,N}$ with $N\in\ZZ_{>0}$. It is a Fr\'echet space. We give $PW(\Lambda)=\bigcup_r PW^r(\Lambda)$ the inductive limit topology in the category of locally convex spaces, making it a strict LF-space.

In general, we define $PW_f(\Lambda)$ to be the linear subspace of $PW(\Lambda)$ consisting of all $\varphi\in PW(\Lambda)$ that are supported on finitely many components $\Lambda_e$. Note that in the non-archimedean case we have $PW_f(\Lambda)=PW(\Lambda)$. We topologise $PW_f$ as follows. First, for each finite set $E_0\subseteq E$, define $PW_{E_0}(\Lambda)$ to be the subspace of all $\varphi\in PW(\Lambda)$ that are supported on $\coprod_{e\in E_0}\Lambda_e$. For each $r>0$, define $PW_{E_0}^r(\Lambda)$ to be the subspace of all $\varphi\in PW_{E_0}(\Lambda)$ such that $\|\varphi\|_r<\infty$ in the non-archimedean case and $\|\varphi\|_{r,N}<\infty$ for all $N\in\ZZ_{>0}$ in the archimedean case. We give $PW_{E_0}^r(\Lambda)$ the Banach space topology defined by the norm $\|\cdot\|_r$ in the non-archimedean case and the Fr\'echet space topology defined by the family of norms $\|\cdot\|_{r,N}$ in the archimedean case.  We give $PW_{E_0}(\Lambda)=\bigcup_r PW_{E_0}^r(\Lambda)$ the inductive limit topology in the category of locally convex spaces, making it a strict LF-space. Finally, we give $PW_f(\Lambda)=\bigcup_{E_0}PW_{E_0}(\Lambda)$ the inductive limit topology in the category of locally convex spaces, making it also a strict LF-space. In the archimedean case, the topology on $PW_f(\Lambda)$ is at least as fine as the subspace topology inherited from $PW(\Lambda)$, that is, the injection $PW_f(\Lambda)\to PW(\Lambda)$ is continuous.

We have defined the topology on $PW_f(\Lambda)$ in an analogous way to how the topology is defined on $C_c^\infty(G,K)$. A simpler way of describing the topology on $PW_f(\Lambda)$ is as follows. We write $PW_e(\Lambda)=PW_{\{e\}}(\Lambda)$. The subspace $PW_e(\Lambda)$ is the classical Paley--Wiener space $PW(\Lambda_e)$ on $\Lambda_e$. Observe that we have the locally convex direct sum decomposition $PW_{E_0}(\Lambda)=\bigoplus_{e\in E_0}PW_e(\Lambda)$, where $PW_e(\Lambda)=PW_{\{e\}}(\Lambda)$. This can be shown by checking that $PW_{E_0}(\Lambda)$ satisfies the correct universal property. Alternatively, this follows from the fact that $\bigoplus_{e\in E_0}PW_{e}(\Lambda)$ is a Fr\'echet space (even a Banach space in the non-archimedean case) and applying a suitable version of the open mapping theorem to the continuous bijection $\bigoplus_{e\in E_0}PW_{e}(\Lambda)\to PW_{E_0}(\Lambda)$. Choosing an enumeration $e_1,e_2,\dots$ of $E_0$, we have 
\[
PW_f(\Lambda)=\bigcup_{n=1}^\infty PW_{\{e_1,\dots,e_n\}}(\Lambda)=\bigcup_{n=1}^\infty\bigoplus_{i=1}^nPW_{e_i}(\Lambda)
\]
with the inductive limit topology in the category of locally convex spaces, and consequently we have the locally convex direct sum decomposition
\[
PW_f(\Lambda)=\bigoplus_{e\in E}PW_e(\Lambda).
\]

To define the Schwartz space $\Ss(\Lambda)$ on $\Lambda$, in the archimedean case we need to make use of certain constant coefficient differential operators on $\Lambda$. Recall in the archimedean case we have an identification of each $\Lambda_e$ with a fixed vector space $iV^*$. Thus, we have a notion of a differential operator on $\Lambda$ that is the same on almost all (all but finitely many) components $\Lambda_e$. 

We define the Schwartz space $\Ss(\Lambda)$ on $\Lambda$ to be the space of smooth functions $\varphi:\Lambda\to\CC$ such that the following holds.
\begin{enumerate}
    \item In the non-archimedean case we require that $\varphi$ is supported on finitely many connected components $\Lambda_e$.
    \item $\varphi$ satisfies a decay condition:
    \begin{enumerate}
        \item In the archimedean case, we require that for each differential operator $D$ on $\Lambda$ that is the same on almost all components and for each $N\in\ZZ_{>0}$, we have
        \[
        \|\varphi\|_{D,N}\defeq \sup_{e\in E, \lambda\in\Lambda_e}|D\varphi(\lambda)| (1+\|e\|+\|\lambda\|)^N
        \]
        is finite.
        \item In the non-archimedean case, we require that for each differential operator $D$ on $\Lambda$, we have
        \[
        \|\varphi\|_{D}\defeq \sup_{\lambda\in\Lambda}|D\varphi(\lambda)|
        \]
        is finite.
    \end{enumerate}
\end{enumerate}
In the archimedean case, $\Ss(\Lambda)$ is a Fr\'echet space with the topology defined by the family of seminorms $\|\cdot\|_{D,N}$. Note that in the definition of $\Ss(\Lambda)$ in the archimedean case, we could have restricted attention to differential operators $D$ that are the same on all components of $\Lambda$; this is how such Schwartz spaces are often defined in the literature.

For each finite set $E_0\subseteq E$ we define $\Ss_{E_0}(\Lambda)$ to be the subspace of all $\varphi\in \Ss(\Lambda)$ that are supported on $\coprod_{e\in E_0}\Lambda_e$. In the archimedean case, $\Ss_{E_0}(\Lambda)$ is a closed subspace of $\Ss(\Lambda)$ and is thus a Fr\'echet space with respect to the topology defined by the family of seminorms $\|\cdot\|_{D,N}$. In the non-archimedean case, we give $\Ss_{E_0}(\Lambda)$ the topology defined by the family of seminorms $\|\cdot\|_{D}$. It is a Fr\'echet space. We define $\Ss_e(\Lambda)=\Ss_{\{e\}}(\Lambda)$. Note that $\Ss_e(\Lambda)$ is the classical Schwartz space $\Ss(\Lambda_e)$ on $\Lambda_e$.

In the non-archimedean case, we give $\Ss(\Lambda)=\bigcup_{E_0}\Ss_{E_0}(\Lambda)$ the inductive limit topology in the category of locally convex spaces, making it a strict LF-space, and we have $\Ss(\Lambda)=C_c^\infty(\Lambda)$.

We could also define a space $\Ss_f(\Lambda)$ in the archimedean case in a manner similar to how we defined $PW_f(\Lambda)$, although we will not make use of this space. Such Schwartz spaces would presumably be used in an invariant Paley--Wiener theorem for $K$-finite Schwartz functions, but to our knowledge such a theorem has not yet been established.

\subsubsection{The Fourier transform.}
Fix a space $\Lambda$ as in \Cref{sec:spaces}. It is naturally dual to the space $X=\coprod_{e\in E}X_e$, where $X_e=V$ in the archimedean case and $X_e=\Gamma_e$ in the non-archimedean case. 

We will define function spaces $C_c^\infty(X)$ and $\Ss(X)$ and a Fourier transform that gives isomorphisms of topological vector spaces $C_c^\infty(X)\to PW(\Lambda)$ and $\Ss(X)\to\Ss(\Lambda)$. We remark that in the archimedean case, $C_c^\infty(X)$ will not be the space of compactly supported smooth functions on $X$ if $X$ has infinitely many components.

First, we define the spaces $C_c^\infty(X)$ and $\Ss(X)$ in the non-archimedean case. We define $\Ss(X)=\bigoplus_{e\in E}\Ss(\Gamma_e)$, where $\Ss(\Gamma_e)$ is the space of functions $\phi:\Gamma_e\to\CC$ that are rapidly decreasing, i.e. for all $N\in\ZZ_{>0}$ we have
\[
\|\phi\|_N\defeq\sup_{x\in X}|\phi(x)|(1+\|x\|)^N<\infty.
\]
The space $\Ss(\Gamma_e)$ is a Fr\'echet space with respect to the topology defined by the seminorms $\|\cdot\|_N$, and $\Ss(X)$ is a strict LF-space.

We define $C_c^\infty(X)=\bigoplus_{e\in E} C_c^\infty(\Gamma_e)$. Here, $C_c^\infty(\Gamma_e)$ is the space of compactly supported (i.e. finitely supported) functions on $\Gamma_e$ equipped with the finest locally convex topology. Thus, $C_c^\infty(X)$ also has the finest locally convex topology. For $r>0$, we define $C_r^\infty(\Gamma_e)$ to be the space of functions on $\Gamma_e$ with support contained in $\{x\in\Gamma_e : \|x\|\leq r\}$. It is a finite-dimensional space, and we give it the locally convex topology. Then $C_r^\infty(\Gamma_e)$ is a closed subspace of $C_c^\infty(\Gamma_e)$. Moreover, we have an increasing union $C_c^\infty(\Gamma_e)=\bigcup_{r>0}C_r^\infty(\Gamma_e)$ and $C_c^\infty(\Gamma_e)$ is the inductive limit of the $C_r^\infty(\Gamma_e)$ in the category of locally convex spaces. It is a strict LF-space. We define $C_r^\infty(X)=\bigoplus_{e\in E}C_r^\infty(\Gamma_e)$. Again, we have that $C_c^\infty(X)$ is the increasing union $C_c^\infty(X)=\bigcup_{r>0} C_r^\infty(X)$, it is the inductive limit of the $C_r^\infty(X)$ in the category of locally convex spaces, and it is a strict LF-space.

Now, we treat the archimedean case. We may view $V=i(iV^*)^*$, and thus we have a definition of $\Ss(X)$. It remains for us to define $C_c^\infty(X)$ in the archimedean case. For $r>0$, we define $C_{r}^\infty(X)$ to be the space of smooth functions $\phi:X\to\CC$ such that:
\begin{enumerate}
    \item the support of $\phi_e$ is contained in the closed ball of radius $r$ in $X_e$;
    \item for all  $N>0$ and all differential operators $D$ that are the same on almost all components $X_e$, we have
    \[
    \|\phi\|_{D,N}\defeq\sup_{e\in E, x\in X_e}|D\phi_e(x)|(1+\|e\|)^N
    \]
    is finite.
\end{enumerate}
We give $C_{r}^\infty(X)$ the topology defined by the family of seminorms $\|\cdot\|_{D,N}$, which makes $C_r^\infty(X)$ a Fr\'echet space. (Note that in the definition of $C_r^\infty(X)$ and its topology, it suffices to use differential operators that are the same on all components of $X$.) We define $C_c^\infty(X)$ to be the increasing union $C_c^\infty(X)=\bigcup_{r>0}C_r^\infty(X)$ and we give it the inductive limit topology in the category of locally convex spaces, making it a strict LF-space.

Note that $C_c^\infty(X)$ is a dense subspace of $\Ss(X)$ in all cases and the inclusion map $C_c^\infty(X)\to\Ss(X)$ is continuous.

We now fix measures on $X$ and $\Lambda$. We fix Pontryagin dual measures on $V$ and $iV^*$. These determine Pontryagin dual measures on $X_e$ and $\Lambda_e$ in the archimedean case. In the non-archimedean case, we give each $X_e=\Gamma_e$ the counting measure and $\Lambda_e=iV^*/\Gamma_e^\vee$ the Pontryagin dual measure.

For $\phi\in\Ss(X)$, we define its Fourier transform $\wh{\phi}=(\wh{\phi_{e}})_{e\in E}:\Lambda\to\CC$ by taking the Fourier transform on each component $X_e$, that is, $\wh{\phi_{e}}(\lambda)=\int_{X_{e}}\phi_{e}(x)e^{-2\pi\langle x,\lambda \rangle}\dd{x}$.

\begin{lemma}
    The Fourier transform gives isomorphisms of topological vector spaces $\Ss(X)\to\Ss(\Lambda)$, $C_{r/2\pi}^\infty(X)\to PW^r(\Lambda)$, and $C_c^\infty(X)\to PW(\Lambda)$.
\end{lemma}

This follows easily from the classical version of this lemma when $E$ is a singleton, together with the following basic inequalities: $(1+x+y)\leq(1+x)(1+y)\leq(1+x+y)^2$ for $x,y\geq0$; and for $N\in\ZZ_{>0}$, we have $(1+x)^N\ll(1+x^2)^N\ll(1+x)^{2N}$ for $x\geq0$.

\begin{corollary}\label{PWDenseSchwartz}
    $PW(\Lambda)$ is a dense subspace of $\Ss(\Lambda)$ and the inclusion map 
    \[
    PW(\Lambda)\longrightarrow\Ss(\Lambda)
    \]
    is continuous.
\end{corollary}

\subsubsection{Pullback mappings.}
Let $V_1$ and $V_2$ be Euclidean spaces and let $T:iV_1^*\to iV_2^*$ be an injective linear map. Then pullback defines continuous linear maps $T^*:\Ss(iV_2^*)\to\Ss(iV_1^*)$ and $T^*:PW(iV_2^*)\to PW(iV_1^*)$. Similarly, if $\Gamma_1$ and $\Gamma_2$ are lattices in $V_1$ and $V_2$, respectively, and $T:iV_1^*/\Gamma_1^\vee\to iV_2^*/\Gamma_2^\vee$ is a smooth homomorphism, then pullback along $T$ defines continuous linear maps $T^*:\Ss(iV_2^*/\Gamma_2^\vee)\to\Ss(iV_1^*/\Gamma_1^\vee)$ and $T^*:PW(iV_2^*/\Gamma_2^\vee)\to PW(iV_1^*/\Gamma_1^\vee)$.

There is a simple generalisation of this to the Schwartz and Paley--Wiener spaces $\Ss(\Lambda)$, $PW(\Lambda)$, and $PW_f(\Lambda)$ defined above.

\begin{lemma}\label{pullback}
    For $i=1,2$, let $\Lambda_i=\coprod_{e\in E_i}\Lambda_e$ be a space as defined in \Cref{sec:spaces} with $V=V_i$.
    
    Let $T_E:E_1\to E_2$ be a partially defined map and let $T_V:iV_1^*\to iV_2^*$ be a linear map. In the archimedean case, assume that $\|T_E (e_1)\|\gg\|e_1\|$ and that $T_V$ is injective. In the non-archimedean case, assume that $T_E$ has finite fibres and that $T_V(\Gamma_1^\vee)\subseteq\Gamma_2^\vee$ so that $T_V$ descends to a smooth homomorphism $T_V:iV_1^*/\Gamma_1^\vee\to iV_2^*/\Gamma_2^\vee$.
    
    Let $\varphi=(\varphi_{e_2})_{e_2\in E_2}\in \Ss(\Lambda_2)$. Define $T^*\varphi=((T^*\varphi)_{e_1})_{e_1\in E_1}$ by $(T^*\varphi)_{e_1}=0$ if $T_E(e_1)$ is undefined and $(T^*\varphi)_{e_1}=T_V^*\varphi_{T_E(e_1)}$ otherwise. To simplify notation, we write $\varphi_{T_E(e_1)}=0$ if $T_E(e_1)$ is not defined. Then $(T^*\varphi)_{e_1}=T_V^*\varphi_{T_E(e_1)}$ for all $e_1\in E_1$.
    
    For all $\varphi\in \Ss(\Lambda_2)$, we have $T^*\varphi\in\Ss(\Lambda_1)$ and 
    \[
    T^*:\Ss(\Lambda_2)\longrightarrow\Ss(\Lambda_1)
    \]
    is a continuous linear map. Moreover, $T^*$ restricts to a continuous linear map
    \[
    T^*:PW(\Lambda_2)\longrightarrow PW(\Lambda_1).
    \]
    In the archimedean case, if $T_E$ in addition has finite fibres, then $T^*$ restricts further to a continuous linear map
    \[
    T^*:PW_f(\Lambda_2)\longrightarrow PW_f(\Lambda_1).
    \]
\end{lemma}

\begin{proof}
    The non-archimedean case and the last claim is straightforward. We prove the first claim in the archimedean case, the second being similar. Let $D$ be an invariant differential operator on $iV_1^*$ and let $N\in\ZZ_{>0}$. We will also use $D$ to denote the extension of $D$ along $T_V$ to an invariant differential operator on $iV_2^*$. Let $\varphi\in\Ss(\Lambda_2)$. We have
    \[
    \|T^*\varphi\|_{D,N}=\sup_{\substack{e_1\in E_1 \\
    \lambda\in\Lambda_{e_1}}} |D\varphi_{T_E(e_1)}(T_V\lambda)|(1+\|e_1\|+\|\lambda\|)^N.
    \]
    For all $N'\in\ZZ_{>0}$, we have
    \[
    |D\varphi_{T_E(e_1)}(T_V\lambda)|\leq\|\varphi\|_{D,N'}(1+\|T_E(e_1)\|+\|T_V\lambda\|)^{-N'}.
    \]
    Since $T_V:iV_1^*\to iV_2^*$ is injective, we have that $\lambda\mapsto \|T_V\lambda\|$ is a norm on $iV_1^*$ and thus $\|T_V\lambda\|\asymp\|\lambda\|$. Since $\|T_E(e_1)\|\gg\|e_1\|$, we have
    \[
    (1+\|T_E(e_1)\|+\|T_V\lambda\|)^{-N'}\ll (1+\|e_1\|+\|\lambda\|)^N
    \]
    for $N'$ sufficiently large. Thus,
    \[
    \|T^*\varphi\|_{D,N}\ll\|\varphi\|_{D,N'}
    \]
    for $N'$ sufficiently large. Thus, we have $T^*\varphi\in\Ss(\Lambda_1)$ and $T^*$ defines a continuous linear map $T^*:\Ss(\Lambda_2)\to\Ss(\Lambda_1)$. That $T^*$ restricts to a continuous linear map $T^*:PW(\Lambda_2)\to PW(\Lambda_1)$ follows in a similar way.
\end{proof}

The proof of \Cref{thm:pullback}, from which the existence of stable transfer follows, makes use of this lemma.

\subsection{Invariant Paley--Wiener Theorems}

As explained above, we view the invariant Fourier transform of $f\in\Cc(G)$ as the $\CC^\times$-equivariant function $f_G:T_\temp(G)\to\CC$ defined by $f_G(\tau)=f_G(\pi_\tau)$. We thus view the space of invariant Fourier transforms of elements of $\Ic_{(c)}(G)$ as
\begin{align*}
    \wh{\Ic_{(c)}}(G)&=\{f_G:T_\temp(G)\to\CC : f\in\Cc_{(c)}(G)\} \\
    &=\Cc_{(c)}(G)/\Ann_{\Cc_{(c)}(G)}(T_\temp(G)).
\end{align*}
The invariant Paley--Wiener theorem for $\Ic_{(c)}(G)$ asserts that the invariant Fourier transform $\Fc:\Ic_{(c)}(G)\to\wh{\Ic_{(c)}}(G)$ is an isomorphism of topological vector spaces and gives a characterisation of $\wh{\Ic_{(c)}}(G)$ as a Schwartz (resp. Paley--Wiener) space. In this subsection we review this and other invariant Paley--Wiener theorems that we will stabilise. 

Let $L\in\Lc^G(M_0)$. Recall that if $F$ is archimedean we have defined a ``norm'' on the set of infinitesimal characters of $L$. We define $\|\tau\|=\|\mu_\tau\|$ for all $\tau\in\wt{T}_\el(L)$. For any countable set $E\subseteq T_\el(L)$, we have a Schwartz space $\Ss(\Lambda)$ and a Paley--Wiener space $PW(\Lambda)$ defined as in \Cref{sec:spaces} on the space $\Lambda=\coprod_{\tau\in E}\Lambda_\tau$ with $\Lambda_\tau=i\ak_{L}^*/\ak_{L,\tau}^\vee$.

We define $\Ss_\el(L)$ to be the space of smooth $\CC^\times$-equivariant functions $\varphi:T_\el(L)\to\CC$ such that for some (and hence any) choice of representatives $E_\el(L)\subseteq T_\el(L)$ for the connected components of $T_\el(L)/\CC^\times$, we have 
\[
\varphi\in \Ss\Bigg(\coprod_{\tau\in E_\el(L)}i\ak_{L}^*/\ak_{L,\tau}^\vee\Bigg).
\]
We define $PW_\el(L)$ (resp. $PW_{\el,f}(L)$) in the same way as $\Ss_\el(L)$, except that we replace $\Ss(\cdot)$ by $PW(\cdot)$ (resp. $PW_f(\cdot)$). We define 
\[
\Ss(G)=\Bigg(\bigoplus_{L\in\Lc^G(M_0)}PW_\el(L)\Bigg)^{W_0^G}=\bigoplus_{L\in\Lc^G(M_0)/W_0^G}PW_\el(L)^{W^G(L)}
\] 
and similarly we define $PW(G)$ and $PW_f(G)$. It follows from the decomposition of $T_\temp(G)$ in terms of the $T_\el(L)$, that $\Ss(G)$, $PW(G)$, and $PW_f(G)$ are naturally spaces of smooth $\CC^1$-equivariant functions on $T_\temp(G)$. We have the following invariant Paley--Wiener theorems.

\begin{theorem}
    The invariant Fourier transform is an isomorphism of topological vector spaces
    \[
    \Ic(G)\longrightarrow\Ss(G).
    \]
    It restricts to isomorphism of topological vector spaces
    \[
    \Ic_c(G)\longrightarrow PW(G).
    \]
    This further restricts to an isomorphism of topological vector spaces
    \[
    \Ic_f(G)\longrightarrow PW_f(G)
    \]
    Thus, we have $\wh{\Ic}(G)=\Ss(G)$, $\wh{\Ic_c}(G)=PW(G)$, and $\wh{\Ic_{f}}(G)=PW_f(G)$.
\end{theorem}

Injectivity of the first map is equivalent to the assertion that for $f\in\Cc(G)$, if $f_G$ vanishes on $D_\temp(G)$, then $f_G(\gamma)=0$ for all $\gamma\in\Gamma_\rs(G)$. This is called spectral density. Kazhdan proved spectral density follows for $p$-adic groups \cite[Theorem J.(a)]{KazCusp}. Spectral density for both real and $p$-adic groups can also be seen as a corollary of Arthur's Fourier inversion theorem for orbital integrals \cite[Theorem 4.1 with $M=G$]{ArthurFourier}, which we will recall in a moment. The first map in the theorem was proved to be an open continuous surjection by Arthur in \cite{ArthurPW}. For real groups, the second statement in the theorem is proved in a more general twisted form in \cite{MWI} using a twisted invariant Paley--Wiener theorem due Renard, which generalises a theorem due to Bouaziz in the non-twisted case. For $p$-adic groups, the second statement is proved in \cite{ArthurElliptic}. The third statement is only different from the second in the case of real groups, and it is also proved in \cite{ArthurElliptic}. (See also \cite[\S6.2]{MWLocal}.)

To state Arthur's Fourier inversion theorem for orbital integrals, we need a measure on $T_\temp(G)/\CC^\times$. First, we define a measure on $T_\el(L)/\CC^\times$ for each $L\in\Lc^G(M_0)$ by
\[
\int_{T_\el(L)/\CC^\times}\alpha(\tau)\dd{\tau}=\sum_{\tau\in T_\el(L)/(\CC^\times\times i\ak_{L}^*)}\int_{i\ak_{L}^*/\ak_{L,\tau}^\vee}\alpha(\tau_\lambda)\dd{\lambda},
\]
for all $\alpha\in C_c(T_\el(L)/\CC^\times)$. Then, we define a measure on $T_\temp(G)/\CC^\times$ by
\[
\int_{T_\temp(G)/\CC^\times}\alpha(\tau)\dd{\tau}=\sum_{L\in\Lc^G(M_0)/W_0^G}|W^G(L)|^{-1}\int_{T_\el(L)/\CC^\times}\alpha(\tau)\dd{\tau}
\]
for all $\alpha\in C_c(T_\temp(G)/\CC^\times)$. The following is Arthur's Fourier inversion theorem for orbital integrals.

\begin{theorem} \label{thm:inversion}
    There exists a smooth function $I_G(\gamma,\tau)$ on $\Gamma_{\sr}(G)\times T_\temp(G)$, which satisfies $I_G(\gamma,z\tau)=z^{-1}I_G(\gamma,\tau)$ for all $z\in\CC^\times$, and satisfies
    \[
    f_G(\gamma)=\int_{T_\temp(G)/\CC^\times}I_G(\gamma,\tau)f_G(\tau)\dd{\tau}
    \]
    for all $f\in\Cc(G)$ and $\gamma\in\Gamma_{\sr}(G)$.
\end{theorem}

This is \cite[Theorem 4.1 with $M=G$]{ArthurFourier}. See the discussion before Lemma 6.3 in \cite{ArthurRelations}, where this specialisation of \cite[Theorem 4.1]{ArthurFourier} is discussed. Writing $I_G(\tau,\gamma)$ for the normalised character $I_G(\tau,\gamma)=|D^G(\gamma)|^{1/2}\Theta_\tau(\gamma)$, we have the dual formula
\[
f_G(\tau)=\int_{\Gamma_\rs(G)}I_G(\tau,\gamma)f_G(\gamma)\dd{\gamma}.
\]

Recall that $\Ic_{\cusp}(G)$ (resp. $\Ic_{c,\cusp}(G)$, $\Ic_{f,\cusp}(G)$) denotes the image of the space $\Cc_\cusp(G)$ (resp. $C_{c,\cusp}^\infty(G)$, $C_{c,\cusp}^\infty(G,K)$) in $\Ic(G)$ (resp. $\Ic_c(G)$, $\Ic_c(G, K)$). For $f\in\Cc(G)$, it follows from spectral density that $f\in\Cc_\cusp(G)$ if and only if its invariant Fourier transform is supported on $T_\el(G)$. We obtain the following invariant Paley--Wiener theorems for cuspidal functions as a corollary of the above invariant Paley--Wiener theorem.

\begin{corollary}
    The invariant Fourier transform is an isomorphism of topological vector spaces
    \[
    \Ic_\cusp(G)\longrightarrow \Ss_\el(G).
    \]
    It restricts to isomorphisms of topological vector spaces
    \[
    \Ic_{c,\cusp}(G)\longrightarrow PW_\el(G)
    \]
    and
    \[
    \Ic_{f,\cusp}(G)\longrightarrow PW_{\el,f}(G).
    \]
\end{corollary}

\subsubsection{Pseudocoefficients.}
Let $\zeta$ be a unitary character of $A_G(F)$. The finite group $\wt{\ak}_{G,F}^\vee/\ak_{G,F}^\vee$ acts on $T_\el(G,\zeta)$ and the $i\ak_G^*$-orbits in $T_\el(G)$ meet $T_\el(G,\zeta)$ in $(\wt{\ak}_{G,F}^\vee/\ak_{G,F}^\vee)$-orbits. Thus, it is natural to regard $T_\el(G,\zeta)$ as a discrete space. Define $PW_{\el,f}(G,\zeta)$ to be the space of all $\CC^1$-equivariant functions $\varphi:T_\el(G,\zeta)\to\CC$ that are supported on finitely many $\CC^1$-orbits. Let $E_\el(G,\zeta)\subseteq T_\el(G,\zeta)$ be a set of representatives for the countable set $T_\el(G,\zeta)/\CC^1$. Then we have an identification
\[
PW_{\el,f}(G,\zeta)=\bigoplus_{\tau\in E_\el(G,\zeta)}\CC,
\]
defined by mapping $\varphi$ to $(\varphi(\tau))_{\tau\in E_\el(G,\zeta)}$. The direct sum $\bigoplus_{\tau\in E_\el(G,\zeta)}\CC$ can be thought of as the $PW_f$-space on the countable 0-dimensional manifold $E_\el(G,\zeta)$. Recall that we have identified $\Pi_2(G,\zeta)$ with a subset of $T_\el(G,\zeta)$. Thus, we may assume that $\Pi_2(G,\zeta)\subseteq E_\el(G,\zeta)$.

There is a Paley--Wiener theorem for the invariant Fourier transform on the space $\Ic_{f,\cusp}(G,\zeta)$, which we now recall from \cite[\S7.2]{MWLocal}. 

\begin{theorem}
    The invariant Fourier transform gives an isomorphism of topological vector spaces
    \[
    \Ic_{f,\cusp}(G,\zeta)\longrightarrow PW_{\el,f}(G,\zeta).
    \]
\end{theorem}

A consequence is that for each $\tau\in T_\el(G,\zeta)$, there exists a unique function $f[\tau]_G\in\Ic_{f,\cusp}(G,\zeta)$ such that $f[\tau]_G(z\tau)=z\|\tau\|_\el^2$ for $z\in\CC^1$, and $f[\tau]_G(\tau')=0$ for $\tau'\in T_\el(G,\zeta)$ with $\tau'\not\in\CC^1\tau$. The function $f[\tau]_G$ is called the pseudocoefficient of $\tau$ in $\Ic_{f,\cusp}(G,\zeta)$. Note that for all $\tau\in T_\el(G,\zeta)$ and $z\in\CC^1$, we have $f[z\tau]_G=z^{-1}f[\tau]_G$. Moreover, the set $\{f[\tau]_G\}_{\tau\in E_\el(G,\zeta)}$ is a basis of $\Ic_{f,\cusp}(G,\zeta)$. It follows from the simple form of the local trace formula that for $\gamma\in\Gamma_{\rs,\el}(G)$ we have
\[
f[\tau]_G(\gamma)=m(\gamma)^{-1}|D^G(\gamma)|^{1/2}\ol{\Theta_{\tau}(\gamma)}.
\]
Recall that $m(\gamma)=\vol(G_\gamma(F)/A_G(F))$. (See \cite[\S7.2]{MWLocal}.)

Since $\{f[\tau]_G\}_{\tau\in E_\el(G,\zeta)}$ is a basis of $\Ic_{f,\cusp}(G,\zeta)$, there is a unique conjugate-linear isomorphism
\[
D_\el(G,\zeta)\longrightarrow\Ic_{f,\cusp}(G,\zeta)
\]
given on the basis $E_\el(G,\zeta)$ by $\tau\mapsto f[\tau]_G$. By conjugate-linearity, for all $z\in\CC^1$ we have $z\tau\mapsto z^{-1}f[\tau]_G=f[z\tau]_G$. Consequently, the isomorphism does not depend on the choice of representatives $E_\el(G,\zeta)$. For all $\pi\in D_\el(G,\zeta)$ we write $f[\pi]_G\in\Ic_{f,\cusp}(G,\zeta)$ for the corresponding element under the above isomorphism. We call $f[\pi]_G$ the pseudocoefficient of $\pi$ in $\Ic_{f,\cusp}(G,\zeta)$.

For all $f_G\in\Ic_\cusp(G,\zeta)$ and $\pi\in D_\el(G,\zeta)$, we have $\langle mf_G,mf[\pi]_G\rangle_\el = f_G(\pi)$. Indeed, it suffices to check this for $\pi=\tau\in T_\el(G,\zeta)$, and this in turn follows from the formula for $f[\tau]_G$ given above. Note that the map $\pi\mapsto mf[\pi]_G$ is unitary with respect to the elliptic inner product in the sense that for all $\pi_1,\pi_2\in D_\el(G,\zeta)$, we have $\langle \pi_1,\pi_2 \rangle_\el=\langle mf[\pi_1]_G,mf[\pi_2]_G \rangle_\el$. If $B$ is an orthogonal basis of $D_\el(G,\zeta)$ (with respect to the elliptic inner product), then for all $b,b'\in B$, we have
\[
f[b]_G(b')=\langle mf[b]_G, mf[b']_G \rangle_\el = \langle b,b' \rangle_\el=\|b\|_\el^2\delta_b(b').
\]

\section{Stable Harmonic Analysis} \label{sec:stable}

\subsection{Stable distributions}
Stable harmonic analysis is concerned with stable orbital integrals, which are integrals over stable conjugacy classes. We begin by recalling these notions. Let $x,x'\in G(F)$. If there exists $g\in G$ such that $x'=g^{-1}xg$, then $g\sigma(g)^{-1}\in G^{x_s}=C_G(x_s)$, where $x_s$ is the semisimple component of $x$ in its Jordan decomposition. The elements $x,x'\in G(F)$ are said to be stably conjugate if there exists $g\in G$ such that $x'=g^{-1}xg$ and $g\sigma(g)^{-1}\in G_{x_s}=C_G(x_s)^\circ$ for all $\sigma\in\Gamma_F$ (see \cite[\S3]{KotConj}). Stable conjugacy is an equivalence relation on $G(F)$, which is intermediate in strength between $G(F)$-conjugacy and $G$-conjugacy. Each stable conjugacy class (or stable class) is a finite union of $G(F)$-conjugacy classes. Note that for complex groups, stable conjugacy is the same as conjugacy. Two strongly regular elements of $G(F)$ are stably conjugate if and only if they are $G$-conjugate, and thus the stable conjugacy class of a strongly regular element $x\in G(F)$ is simply the intersection $x^{G}\cap G(F)$ of its $G$-conjugacy class with $G(F)$. The Weyl discriminant $D^G$ is stably invariant, that is, constant on stable conjugacy classes. If $T$ is a maximal torus of $G$, then two $G$-regular elements of $T(F)$ are conjugate if and only if they lie in the same orbit of the stable Weyl group $W(G,T)(F)$ of $T$. Let 
\[
\Delta_\sr(G)\subseteq\Delta_\rs(G)\subseteq \Delta(G)
\]
denote the spaces of stable conjugacy classes in $G_\sr(F)$, $G_\rs(F)$, and $G(F)$, respectively. Then $\Delta_\sr(G)$ and $\Delta_{\rs}(G)$ are open dense locally compact Hausdorff subspaces of $\Delta(G)$, and $\Delta_\sr(G)$ is naturally an $F$-analytic manifold.

The stable orbital integral at $\delta\in\Delta_{\rs}(G)$ is the $\zeta$-equivariant Radon measure and tempered distribution on $G(F)$ defined by
\[
SO_\delta=\sum_{\gamma\subseteq\delta}O_\gamma,
\]
where the sum is over all conjugacy classes $\gamma$ in the stable conjugacy class $\delta$. Let $f\in\Cc(G,\zeta)$. The normalised stable orbital integral of $f$ is defined by 
\[
f^G(\delta)=|D^G(\delta)|^{1/2}SO_\delta(f)=\sum_{\gamma\subseteq\delta}f_G(\gamma).
\]
The functions $\delta\mapsto SO_\delta(f)$ and $f^G$ are continuous on $\Delta_\rs(G)$ and smooth on $\Delta_\sr(G)$, and $f^G$ is locally bounded.

Two maximal tori $T,T'$ of $G$ are said to be stably conjugate if there exists $g\in G$ such that $\Int(g):T\to T'$ is defined over $F$. We have a stable version of the Weyl integration formula,
\[
\int_{G(F)}f(x)\dd{x}=\sum_{\{T\}_{\st}}|W(G,T)(F)|^{-1}\int_{T(F)}|D^G(t)|SO_t(f)\dd{t},
\]
where $\{T\}_{\st}$ runs over the set of stable conjugacy classes of maximal tori of $G$. As in the invariant case, one can define a Radon measure $\dd{\delta}$ on $\Delta_\rs(G)$ by
\[
\int_{\Delta_\rs(G)}\alpha(\delta)\dd{\delta}=\sum_{\{T\}_{\st}}|W(G,T)(F)|^{-1}\int_{T(F)}\alpha(t)\dd{t},
\]
for all $\alpha\in C_c(\Delta_\rs(G))$, where the sum runs over a set of representatives of the stable conjugacy classes of maximal tori of $G$. The stable Weyl integration formula then becomes
\[
\int_{G(F)}f(x)\dd{x}=\int_{\Delta_\rs(G)}|D^G(\delta)|SO_\delta(f)\dd{\delta}.
\]
Note that $\Delta_\sr(G)$ has full measure in $\Delta_\rs(G)$. 

We define the subspace of unstable functions in $\Cc_{(c)}(G,\zeta)$ by
\begin{align*}
    \Cc_{(c)}^\unst(G,\zeta)&=\Ann_{\Cc_{(c)}(G,\zeta)}(\{SO_\delta : \delta\in\Delta_\rs(G)\}) \\
    &=\{f\in\Cc_{(c)}(G,\zeta) : f^G(\delta)=0,\,\forall\delta\in\Delta_\rs(G)\}
\end{align*}
and its image in $\Ic_{(c)}(G,\zeta)$,
\[
\Ic_{(c)}^\unst(G,\zeta)=\Ann_{\Ic_{(c)}(G,\zeta)}(\{SO_\delta : \delta\in\Delta_\rs(G)\}).
\]
We define the space of stable orbital integrals 
\[
\Sc_{(c)}(G,\zeta)=\{f^G : f\in\Cc_{(c)}(G,\zeta)\}.
\]
We have
\[
\Sc_{(c)}(G,\zeta)=\Cc_{(c)}(G,\zeta)/\Cc_{(c)}^\unst(G,\zeta)=\Ic_{(c)}(G,\zeta)/\Ic_{(c)}^\unst(G,\zeta)
\]
and we give it the natural quotient topology. We also define the spaces $\Cc_{c}^\unst(G,\zeta,K)$, $\Ic_{f}^\unst(G,\zeta)=\Ic_{c}^\unst(G,\zeta,K)$, and $\Sc_{f}(G,\zeta)=\Sc_{c}(G,\zeta,K)$ in a similar way.

The space of $\zeta$-equivariant stable distributions is defined to be $\Sc_c(G,\zeta)'$ and the space of tempered $\zeta$-equivariant stable distributions is defined to be $\Sc(G,\zeta)'$. Note that we have a continuous linear injection $\Sc_c(G,\zeta)\to\Sc(G,\zeta)$ with dense image. Its transpose is a continuous linear injection $\Sc(G,\zeta)'\to\Sc_c(G,\zeta)'$, which enables us to identify each tempered $\zeta$-equivariant stable distribution with a $\zeta$-equivariant stable distribution.

We may identify $\Sc_{(c)}(G,\zeta)'$ with a subspace of $\Ic_{(c)}(G,\zeta)'$ via the transpose of the quotient map $\Ic_{(c)}(G,\zeta)\to\Sc_{(c)}(G,\zeta)$. Thus, we have a further identification of $\Sc_{(c)}(G,\zeta)'$ as a subspace of $\Cc_{(c)}(G,\zeta)'$. As vector spaces, we have 
\begin{align*}
    \Sc_{(c)}(G,\zeta)'&=\Ann_{\Cc_{(c)}(G,\zeta)'}(\Ann_{\Cc_{(c)}(G,\zeta)}(\{SO_\delta : \delta\in\Delta_\rs(G)\})) \\
    &=\cl_{\Cc_{(c)}(G,\zeta)',\textrm{weak-}*}(\{SO_\delta : \delta\in\Delta_\rs(G)\}).
\end{align*}
That is, an distribution in $\Cc_{(c)}(G,\zeta)'$ belongs to $\Sc_{(c)}(G,\zeta)'$ if and only if it lies in the weak-$*$ closure in $\Cc_{(c)}(G,\zeta)'$ of the linear span of the set of regular semisimple stable orbital integrals of $G$. A locally integrable function $\Theta$ on $G(F)$ that is continuous on $G_\sr(F)$ defines a stable distribution of $G$ if and only if $\Theta$ is stably invariant on $G_\sr(F)$.

We say that a virtual representation is stable if its character, which we identify it with, is stable. Let $D_\spec^\st(G,\zeta)$ (resp. $D_\temp^\st(G,\zeta)$, $D_\el^\st(G,\zeta)$) denote the subspace of stable elements in $D_\spec(G,\zeta)$ (resp. $D_\temp(G,\zeta)$, $D_\el(G,\zeta)$).

The parabolic descent map $\Ic_{(c)}(G,\zeta)\to\Ic_{(c)}(M,\zeta)^{W^G(M)}$ descends to a continuous map 
\[
\Sc_{(c)}(G,\zeta)\longrightarrow\Sc_{(c)}(M,\zeta)^{W^G(M)},
\]
which we write as $f^G\mapsto f^M$ and also call parabolic descent. Thus, the parabolic induction map $I_M^G:\Ic_{(c)}(M,\zeta)'/W^G(M)\to\Ic_{(c)}(G,\zeta)'$ restricts to a continuous map $I_M^G:\Sc_{(c)}(M,\zeta)'/W^G(M)\to\Sc_{(c)}(G,\zeta)'$, namely the transpose of the parabolic descent map 
\[
\Sc_{(c)}(G,\zeta)\longrightarrow\Sc_{(c)}(M,\zeta)^{W^G(M)}.
\]
Consequently, parabolic induction $I_M^G$ maps stable distributions on $M(F)$ to stable distributions on $G(F)$. Let $\Delta_{G\dash\rs}(M)$ denote the set of $G$-regular semisimple stable conjugacy classes in $M(F)$. For all $\delta\in\Delta_{G\dash\rs}(M)$ we have $f^M(\delta)=f^G(\delta)$.

\subsection{$L$-groups and $L$-parameters}
We refer to a homomorphism $\eta:G\to G'$ with a normal image as a normal homomorphism. Let $\Red_{F,\Out}$ be the category whose objects are connected reductive groups $G$ over $F$, and whose morphisms are normal homomorphisms $\eta:G\to H$, taken up to $H_{\ad}(F)$-conjugacy, which we call equivalence. We write $\Psi(G)$ for the based root datum $\Psi(G_{\ol{F}})$ of $G_{\ol{F}}$ equipped with the locally constant homomorphism $\Gamma_F\to\Aut(\Psi(G_{\ol{F}}))$. We write $\dual{\Psi(G)}$ for the based root datum $\dual{\Psi(G_{\ol{F}})}$ dual to $\Psi(G_{\ol{F}})$ equipped with the locally constant homomorphism $\Gamma_F\to\Aut(\dual{\Psi(G_{\ol{F}})})$.

Let $\Red_{\CC,L,\Out}$ be the category whose objects are connected reductive groups $\dual{G}$ over $\CC$ equipped with an $L$-action (a locally constant homomorphism $\Gamma_F\to\Aut(\dual{G})$ that preserves a pinning), and whose morphisms are $\Gamma_F$-equivariant normal homomorphisms $\dual{\eta}:\dual{H}\to\dual{G}$ taken up to $\dual{G}^{\Gamma_F}$-conjugacy, which we call equivalence. We write $\Psi(\dual{G})$ for the based root datum of $\dual{G}$ equipped with the locally constant homomorphism $\Gamma_F\to\Aut(\Psi(\dual{G}))$.

There is an exact contravariant Langlands dual group functor
\[
  \Red_{F,\Out}\longrightarrow\Red_{\CC,L,\Out}
\]
such that we have an identification of based root data $\dual{\Psi(G)}=\Psi(\dual{G})$ equipped with $\Gamma_F$-actions (cf. \cite{Ngo}). This functor is unique up to natural isomorphism. We denote this functor on objects by $G\mapsto\dual{G}$ and on morphisms by $\eta\mapsto\dual{\eta}$. Moreover, the Langlands dual group functor has a section to the full subcategory of $\Red_{F,\Out}$ consisting of quasisplit groups. The Langlands dual group $\dual{G}$ of $G$ comes equipped with an $L$-action $\rho_G:\Gamma_F\to\Aut(\dual{G})$.

We have a canonical $\Gamma_F$-equivariant isomorphism
\[
X^*(G_{\ol{F}})\longrightarrow X_*(Z(\dual{G})^\circ)
\]
and 
\[
X^*(G)\longrightarrow X_*(Z(\dual{G})^{\Gamma_F,\circ}),
\]
which we write as $\theta\mapsto\dual{\theta}$.

The identification $\dual{\Psi(G)}=\Psi(\dual{G})$ gives rise to a canonical $\Gamma_F$-equivariant bijection between the set of conjugacy classes of parabolic (resp. Levi) subgroups of $G_{\ol{F}}$ and those of $\dual{G}$. This restricts to a bijection between the set of $\Gamma_F$-stable conjugacy classes of parabolic (resp. Levi) subgroups of $G_{\ol{F}}$ and those of $\dual{G}$. A parabolic (resp. Levi) subgroup of $\dual{G}$ is said to be $G$-relevant if its conjugacy class corresponds to the conjugacy class of a parabolic (resp. Levi) subgroup of $G_{\ol{F}}$ that is defined over $F$. If a parabolic (resp. Levi) subgroup of $\dual{G}$ is $G$-relevant, then so is every parabolic (resp. Levi) subgroup of $G$ containing it.

As mentioned in the introduction, we use the Weil form of the $L$-group: $\Lgp{G}=\dual{G}\rtimes W_F$. The $L$-groups lie in a natural category of what we refer to as $\lambda$-groups based on the notation in \cite{LanBE}.

Given a surjective homomorphism $\Gc\to W_F$, we will write $\Gc^0=\ker(\Gc\to W_F)$. Note that we have used a superscript ``0'' instead of a superscript ``$\circ$'', which we use to denote identity components. Let $\Gc$ be a second countable locally compact group together with a continuous surjective homomorphism $\Gc\to W_F$. By the open mapping theorem, the homomorphism $\Gc\to W_F$ is open. Suppose that the kernel $\Gc^0=\ker(\Gc\to W_F)$ is the group of $\CC$-points of a connected reductive group $\Gc^0$ over $\CC$ and that for all $g\in\Gc$ the automorphism $\Int(g)|_{\Gc^0}:\Gc^0\to\Gc^0$ is algebraic. The resulting homomorphism $\Gc\to\Aut(\Psi(\Gc^0))$ then factors through $W_F$. Suppose that it further factors through $W_F/W_K=\Gamma_{K/F}$ for some finite Galois subextension $K/F$ of $F_s/F$, and thus extends to a continuous homomorphism $\Gamma_F\to\Aut(\Psi(\Gc^0))$ along $W_F\to\Gamma_F$. We then say that $\Gc\to W_F$, or just $\Gc$, is a $\lambda$-group.

An element $g\in\Gc$ is said to be semisimple if $\Int(g)|_{\Gc^0}$ is a semisimple automorphism of $\Gc^0$, that is, after embedding $\Gc^0$ in a general linear group it can be realised as conjugation by a semisimple element. If $g\in\Gc^0$, then $g$ is semisimple as an element of $\Gc$ if and only if it is semisimple as an element of $\Gc^0$. If $\Gc$ and $\Gc'$ are $\lambda$-groups, an isomorphism of topological groups $\Hc\to\Gc$ over $W_F$ maps semisimple elements to semisimple elements.

Note that an $L$-group $\Lgp{G}=\dual{G}\rtimes W_F$ of a connected reductive group $G$ over $F$ together with the projection $\Lgp{G}\to W_F$ is a $\lambda$-group. The semisimple elements of $\Lgp{G}$ are those of the form $(g,w)$, where $g$ is a semisimple element of $\dual{G}$.

Let $\Hc,\Gc$ be two $\lambda$-groups. A continuous homomorphism $\xi:\Hc\to\Gc$ over $W_F$ is said to be an $L$-homomorphism if its restriction $\xi_0:\Hc^0\to\Gc^0$ is algebraic and semisimple in the sense that it maps semisimple elements to semisimple elements. Two $L$-homomorphisms $\xi,\xi':\Hc\to\Gc$ are said to be equivalent if there exists $g\in\Gc^0$ such that $\xi'=\Int(g)\circ\xi$. We define a morphism of $\lambda$-groups to be an equivalence class of $L$-homomorphisms and denote the category of $\lambda$-groups by $\lambda\Gp$. We denote the full subcategory of $L$-groups by $L\Gp$. We call an $L$-homomorphism $\xi:\Hc\to\Gc$ an $L$-embedding if it is a topological embedding, in which case $\xi(\Hc)$ is a closed subgroup of $\Gc$ and $\xi_0:\Hc^0\to\Gc^0$ is a closed embedding of algebraic groups. We will often switch between thinking of $\xi$ as an equivalence class and a single $L$-homomorphism.

For a normal homomorphism $\eta:G\to H$, we have an equivalence class of $L$-homomorphisms $\Lgp{\eta}=\dual{\eta}\rtimes\id_{W_F}:\Lgp{H}\to\Lgp{G}$ and this gives an exact contravariant functor $\Red_{F,\Out}\to\lambda\Gp$. Its essential image is $L\Gp$, and a $\lambda$-group $\Gc$ lies in $L\Gp$ if and only if there is a splitting $c:W_F\to\Gc$ (i.e. continuous homomorphic section of $\Gc\to W_F$) such that the action $\rho_c:W_F\to\Aut(\Gc^0)$ defined by $\rho_c(w)=\Int(c(w))|_{\Gc^0}$ preserves a pinning, in which case $\rho_c$ factors through $\Gamma_{K/F}$ if $W_F\to\Aut(\Psi(\Gc^0))$ does.

We recall that a parabolic subgroup of $\Lgp{G}$ is defined to be a closed subgroup $\Pc$ that maps onto $W_F$ such that $\Pc^0$ is a parabolic subgroup of $\dual{G}$. Similarly, a Levi subgroup of $\Lgp{G}$ is defined to be a closed subgroup $\Mcal$ that maps onto $W_F$ such that $\Mcal^0$ is a Levi subgroup of $\dual{G}$. We say that $\Pc$ (resp. $\Mcal$) is $G$ relevant if $\Pc^0$ (resp. $\Mcal^0$) is $G$-relevant.

For each Levi subgroup of $G$, we have a canonical equivalence class of $\Gamma_F$-equivariant embeddings $\iota_M^G:\dual{M}\to\dual{G}$, and these extend to give a canonical equivalence class of tempered $L$-embeddings $\iota_M^G:\Lgp{M}\to\Lgp{G}$. The $\dual{G}$-conjugacy classes of groups $\iota_M^G(\dual{M})$ (resp. $\iota_M^G(\Lgp{M})$) are precisely the $\dual{G}$-conjugacy classes of $G$-relevant Levi subgroups of $\dual{G}$ (resp. $\Lgp{G}$). If $L\subseteq M$ are Levi subgroups of $G$, then $\iota_M^G\circ\iota_L^M=\iota_L^G$.

Let $\Hc$ be a subgroup of $\Lgp{G}$ that maps onto $W_F$. By \cite[Proposition 3.6]{Borel} the Levi subgroups $\Mcal_\Hc$ of $\Lgp{G}$ that contain $\Hc$ minimally are all conjugate by $C_{\dual{G}}(\Hc)$. We say that $\Hc$ is $G$-relevant if $\Hc$ is only contained in $G$-relevant Levi subgroups of $\Lgp{G}$, or equivalently if $\Mcal_\Hc$ is $G$-relevant. We say that $\Hc$ is elliptic if it is not contained in a proper Levi subgroup of $\Lgp{G}$. If $\Hc$ is a $\lambda$-group, an $L$-homomorphism $\xi:\Hc\to\Lgp{G}$ is said to be $G$-relevant (resp. elliptic) if its image is $G$-relevant (resp. elliptic), and we write $\Mcal_\xi=\Mcal_{\xi(\Hc)}$.

Every $L$-homomorphism $\xi:\Lgp{H}\to\Lgp{G}$ can be written in the form $\xi(h,w)=(\xi_0(h)a_\xi(w),w)$, where $\xi_0:\dual{H}\to\dual{G}$ is a morphism of algebraic groups and $a_\xi\in Z_c^1(W_F,\dual{G})$. The cohomology class $a_\xi\in H_c^1(W_F,\dual{G})$ only depends on the equivalence class of $\xi$. One says that $\xi$ is tempered (or bounded, or of unitary type) if the image of $a_\xi$ in $\dual{G}$ is bounded, that is, has compact closure. We will use the notation $a_\xi$ and $\xi_0$ without comment.

Recall that $L_F$ denotes the local Langlands group of $F$, which defined by 
\[
L_F=\begin{cases}
    \Lgp{1}=W_F & \textrm{if $F$ is archimedean}, \\
    \Lgp{\PGL_2}=\SL_2(\CC)\times W_F & \textrm{if $F$ is non-archimedean}.
\end{cases}
\]
We write the homomorphism $L_F\to W_F$ as $l\mapsto w(l)$. For an $L$-homomorphism $\phi:L_F\to\Lgp{G}$, we write $\phi(l)=(a_\phi(l),w(l))$.

The set $\Phi(G)$ of $L$-parameters of $G$ is defined to be the set of equivalence classes of $G$-relevant $L$-homomorphisms $\phi:L_F\to\Lgp{G}$. We have the subset $\Phi_\temp(G)\subseteq\Phi(G)$ of tempered $L$-parameters of $G$ the further subset $\Phi_2(G)\subseteq\Phi_\temp(G)$ of discrete $L$-parameters of $G$, which are those that are tempered and elliptic. 

Let $\xi:\Lgp{H}\to\Lgp{G}$ be an $L$-homomorphism. For each $\phi\in\Phi(H)$ we obtain an equivalence class of $L$-homomorphisms $\xi_*\phi=\xi\circ\phi:L_F\to\Lgp{G}$, which only depends on the equivalence class of $\xi$. We thus obtain a partially defined map $\xi_*:\Phi(H)\to\Phi(G)$, whose domain is the set $\phi\in\Phi(H)$ such that $\xi_*\phi$ is $G$-relevant. If $G$ is quasisplit or $\xi$ is an isomorphism, then $\xi_*$ is defined on all of $\Phi(H)$. If $\xi=\Lgp{\eta}$ for a normal homomorphism $\eta:G\to H$ with abelian kernel and cokernel, then $\xi_*$ is also defined on all of $\Phi(H)$.

\subsubsection{Central and cocentral characters and the Langlands pairing.}
In order to formulate various properties of the local Langlands correspondence we need two constructions originally due to Langlands \cite{LanReal}. A more intrinsic approach is given by Kaletha in \cite{KalethaEpipelagic} using the cohomology of crossed modules. The first is the central character map
\[
\zeta:H_c^1(W_F,\dual{G})\longrightarrow\Pi(Z_G(F)).
\]
The classes $H_{c,\bdd}^1(W_F,\dual{G})$ represented by bounded 1-cocycles map into the group $\Pi_u(Z_G(F))$ of unitary central characters. We have a map $\Phi(G)\to H_c^1(W_F,\dual{G})$ defined by $\phi\mapsto a_\phi$, and this restricts to a map $\Phi_\temp(G)\to H_{c,\bdd}^1(W_F,\dual{G})$. Thus, we have a central character map 
\begin{align*}
    \zeta:\Phi_{(\temp)}(G)&\longrightarrow\Pi_{(u)}(Z_G(F)). \\
    \phi&\longmapsto \zeta_\phi
\end{align*}
For a central datum $(\Zc,\zeta)$ of $G(F)$, we can thus define $\Phi(G,\zeta)=\{\phi\in\Phi(G) : \zeta_\phi|_\Zc=\zeta\}$. We define $\Phi_\temp(G,\zeta)=\Phi_\temp(G)\cap\Phi(G,\zeta)$ and $\Phi_2(G,\zeta)=\Phi_2(G)\cap\Phi(G,\zeta)$, which are empty unless $\zeta$ is unitary.

The second construction is the cocentral character homomorphism
\begin{align*}
\chi:H_c^1(W_F,Z(\dual{G}))&\longrightarrow\Hom_c(G(F),\CC^\times) \\
a&\longmapsto \chi_a
\end{align*}
The corresponding pairing $H_c^1(W_F,Z(\dual{G}))\times G(F)\to\CC^\times$ is called the Langlands pairing. (See also \cite[\S5.1]{KS} and Appendix A by Labesse and Lapid in \cite{LapidMao}.) Let $Z(\dual{G})^1$ denote  the maximal compact subgroup of $Z(\dual{G})$. The cocentral character homomorphism restricts to a homomorphism
\[
\chi:H_c^1(W_F,Z(\dual{G})^1)\longrightarrow\Hom_c(G(F),\CC^1),
\]
and thus we have the corresponding pairing $H_c^1(W_F,Z(\dual{G})^1)\times G(F)\to\CC^1$. Following \cite{KalPra}, we write $G(F)^\natural=\im(G_\scd(F)\to G(F))$, where $G_\scd\to G_\der$ is the simply connected cover of the derived group of $G$. The construction in \cite{KalethaEpipelagic} makes it clear that the cocentral character homomorphism has image in the group $\Hom_c(G(F)/G(F)^\natural,\CC^\times)$ of cocentral characters of $G$.

The cocentral character homomorphism and its unitary restriction are isomorphisms if $G$ is quasisplit, in particular if $G$ is a torus. The homomorphism is injective if $F$ is non-archimedean. It is surjective if $G_{\scd}(F)$ is perfect, or equivalently if $F$ is archimedean or $G_{\scd}$ does not contain a simple factor of the form $\Res_{E/F}(\SL_1(D))$ for a finite-dimensional non-commutative division algebra $D$ over a finite separable (this works over positive characteristic) extension $E$ of $F$. (See Appendix A by Labesse and Lapid in \cite{LapidMao}.)

\subsubsection{The local Langlands correspondence for tori.}
Suppose that $G=T$ is a torus. Then we have a bijection $\Phi(T)\cong H_c^1(W_F,\dual{T})$ defined by $\phi\mapsto a_\phi$, and we transport the group structure from $H_c^1(W_F,\dual{T})$ to $\Phi(T)$ so that this bijection becomes a group isomorphism. It restricts to an isomorphism $\Phi_\temp(T)=H_c^1(W_F,(\dual{T})^1)$, where $(\dual{T})^1$ is the maximal compact subgroup of $\dual{T}$. Furthermore, we have $\Pi(T)=\Hom_c(T(F),\CC^\times)$ and $\Pi_\temp(T)=\Hom_c(T(F),\CC^1)$. The cocentral character homomorphism in this case thus gives us an isomorphism $\rec_T:\Phi(T)\to\Pi(T)$, which restricts to an isomorphism $\rec_T:\Phi_\temp(T)\to\Pi_\temp(T)$. It is called the local Langlands correspondence or local reciprocity map for tori and was first constructed in \cite{LanAbelian}. (See \cite[\S9]{Borel} for an overview and \cite{LabesseTori} for a slightly different approach.) 

\subsubsection{Unramified characters.}
Let $W_F^1=\ker(|\cdot|:W_F\to\RR_{>0})$ and consider
\[
H_c^1(W_F/W_F^1,Z(\dual{G})^{\Gamma_F,\circ})=\Hom_c(W_F/W_F^1,Z(\dual{G})^{\Gamma_F,\circ}).
\]
We call the homomorphism
\[
\begin{tikzcd}
    H_c^1(W_F/W_F^1,Z(\dual{G})^{\Gamma_F,\circ}) \arrow[r] & H_c^1(W_F,Z(\dual{G})) \arrow[r,"\chi"] & \Hom_c(G(F),\CC^\times)
\end{tikzcd}
\]
the unramified cocentral character homomorphism. Its image is the group $X^\nr(G)$ of unramified characters of $G$, and we recall why this is the case by relating it to the isomorphism $\ak_{G,\CC}^*/\ak_{G,F}^\vee\to X^\nr(G),\lambda\mapsto|\cdot|_G^\lambda$.

Recall the canonical isomorphism 
\begin{align*}
X^*(G)&\longrightarrow X_*(Z(\dual{G})^{\Gamma_F,\circ}) \\
\theta&\longmapsto\theta^\vee
\end{align*}
which comes from Langlands duality. This induces an isomorphism
\[
\begin{tikzcd}
    \ak_{G,\CC}^*=X^*(G)\otimes_\ZZ\CC \arrow[r,"\sim"] & X_*(Z(\dual{G})^{\Gamma_F,\circ})\otimes_\ZZ\CC=\Lie(Z(\dual{G})^{\Gamma_F,\circ}),
\end{tikzcd}
\]
which we write as $\lambda\mapsto\dual{\lambda}$. Here, we have used that for any complex torus $T$ we have identifications
\[
\begin{tikzcd}
    X_*(T)\otimes_\ZZ\CC^\times \arrow[r,"\sim"] & T \\
    X_*(T)\otimes_\ZZ\CC \arrow[r,"\sim"] \arrow[u,"\id\otimes\exp"] & \Lie(T) \arrow[u,swap,"\exp"]
\end{tikzcd}
\]
For $\lambda\in\ak_{G,\CC}^*$, we define $\|\cdot\|^{\dual{\lambda}}\in H_c^1(W_F/W_F^1,Z(\dual{G})^{\Gamma_F,\circ})$ by
\[
\|w\|^{\dual{\lambda}}=\exp((\log\|w\|)\dual{\lambda})
\]
for all $w\in W_F$. This defines a homomorphism
\[
\ak_{G,\CC}^*\longrightarrow H_c^1(W_F/W_F^1,Z(\dual{G})^{\Gamma_F,\circ}).
\]
For $\lambda\in\ak_{G,\CC}^*$, let $a_\lambda$ be the image of $\|\cdot\|^{\dual{\lambda}}$ in $H_c^1(W_F,Z(\dual{G}))$, yielding a homomorphism
\[
\ak_{G,\CC}^*\longrightarrow H_c^1(W_F,Z(\dual{G})).
\]
We have 
\[
\chi_{a_\lambda}=|\cdot|_G^\lambda.
\]
(See the discussion below 4.7 in \cite{SilZin}.) Thus, the unramified cocentral character homomorphism maps onto $X^\nr(G)$ and the following diagram commutes
\[
\begin{tikzcd}[column sep = small]
    &\ak_{G,\CC}^* \arrow[dl,two heads] \arrow[dr,two heads] \\
    H_c^1(W_F/W_F^1,Z(\dual{G})^{\Gamma_F,\circ}) \arrow[rr,two heads] && X^\nr(G)
\end{tikzcd}
\]

We will now proceed to describe the unramified cocentral character homomorphism and the homomorphism $\lambda\mapsto a_\lambda$ in more detail. In particular, we will see that the kernel of $\lambda\mapsto a_\lambda$ is trivial when $F$ is archimedean, and is equal to $\ak_{G,F}^\vee$ when $F$ is non-archimedean. We will make use of these facts in the following subsection.

Suppose that $F$ is archimedean. We have $W_F/W_F^1=\RR_{>0}$ and thus
\[
H_c^1(W_F/W_F^1,Z(\dual{G})^{\Gamma_F,\circ})=H_c^1(W_F/W_F^1,Z(\dual{G})^{\Gamma_F}).
\]
We may identify this with $\Lie(Z(\dual{G})^{\Gamma_F})=\Lie(Z(\dual{G})^{\Gamma_F,\circ})$. The homomorphism $\lambda\mapsto\|\cdot\|^{\lambda^\vee}$ is then identified with the isomorphism $\ak_{G,\CC}^*\to\Lie(Z(\dual{G})^{\Gamma_F}),\lambda\mapsto\lambda^\vee$. The homomorphism $\ak_{G,\CC}^*\to X^\nr(G)$ is also an isomorphism. Consequently, the unramified cocentral character homomorphism is identified with an isomorphism
\[
\begin{tikzcd}
    \Lie(Z(\dual{G})^{\Gamma_F}) \arrow[r,"\sim"] & X^\nr(G).
\end{tikzcd}
\]
The homomorphism $W_F\to\Gamma_F$ maps $W_F^1$ onto $\Gamma_F$. Therefore $Z(\dual{G})^{\Gamma_F}=Z(\dual{G})^{W_F^1}$. The unramified cocentral character homomorphism is the restriction of the cocentral character homomorphism along the inflation homomorphism
\[
\begin{tikzcd}
    \Lie(Z(\dual{G})^{\Gamma_F})=H_c^1(W_F/W_F^1,Z(\dual{G})^{W_F^1}) \arrow[r,hook] & H_c^1(W_F,Z(\dual{G})).
\end{tikzcd}
\]
In summary, we have the following commutative diagram.
\[
\begin{tikzcd}[column sep = small]
& \ak_{G,\CC}^* \arrow[dl,"\sim"{sloped,auto}] \arrow[dr,"\sim"{sloped,auto}] & \\
\Lie(Z(\dual{G})^{\Gamma_F}) \arrow[rr,"\sim"] \arrow[d,hook] && X^\nr(G) \arrow[d,hook] \\
H_c^1(W_F,Z(\dual{G})) \arrow[rr,"\chi"] && \Hom_c(G(F)/G(F)^\natural,\CC^\times)
\end{tikzcd}
\]
The composition on the left is the homomorphism $\lambda\mapsto a_\lambda$, which is thus injective.

Suppose that $F$ is non-archimedean. We have $W_F^1=I_F$ and $W_F/W_F^1=\langle\Fr\rangle$ and the norm $\|\cdot\|:W_F/W_F^1\to q_F^\ZZ$ is defined by $\|\Fr\|=q_F$. We may identify
\[
H_c^1(W_F/W_F^1,Z(\dual{G})^{\Gamma_F,\circ})=Z(\dual{G})^{\Gamma_F,\circ}.
\]
The homomorphism $\lambda\mapsto\|\cdot\|^{\lambda^\vee}$ is then identified with the surjective homomorphism $\ak_{G,\CC}^*\twoheadrightarrow Z(\dual{G})^{\Gamma_F,\circ},\lambda\mapsto q_F^{\lambda^\vee}$, which descends to an isomorphism 
\[
\begin{tikzcd}
    \ak_{G,\CC}^*/\frac{2\pi i}{\log q_F}X^*(G) \arrow[r,"\sim"] & Z(\dual{G})^{\Gamma_F,\circ}.
\end{tikzcd}
\]
The homomorphism $\ak_{G,\CC}^*\to X^\nr(G)$ has kernel $\ak_{G,F}^\vee$, and the kernel of the unramified cocentral character homomorphism
\[
\begin{tikzcd}
    Z(\dual{G})^{\Gamma_F,\circ}=H_c^1(W_F/I_F,Z(\dual{G})^{\Gamma_F,\circ}) \arrow[r,two heads] & X^\nr(G)
\end{tikzcd}
\]
is the finite subgroup of $Z(\dual{G})^{\Gamma_F,\circ}$ corresponding to the subgroup 
\[
\ak_{G,F}^\vee/\frac{2\pi i}{\log q_F}X^*(G)
\]
of $\ak_{G,\CC}^*/\frac{2\pi i}{\log q_F}X^*(G)$. 

As in \cite[\S3.3]{Haines}, we recall how the unramified cocentral character homomorphism can be described using Kottwitz homomorphism for $G$, which was introduced in \cite{KotIsoII}. (See \cite[Ch. 11]{KalPra} for a detailed exposition.) The Kottwitz homomorphism for $G$ is a surjective homomorphism
\[
\begin{tikzcd}
    \kappa_G:G(F) \arrow[r,two heads] & (X^*(Z(\dual{G}))_{I_F})^{\Fr}=X^*((Z(\dual{G})^{I_F})_\Fr).
\end{tikzcd}
\]
Its kernel $G(F)_1=\ker\kappa_G$ is an open subgroup of $G(F)^1$, and $G(F)_1$ is the subgroup of $G(F)$ generated by its parahoric subgroups. A (continuous) character of $G(F)$ is said to be weakly unramified if it is trivial on $G(F)_1$. We denote the group of weakly unramified characters of $G(F)$ by 
\[
X^\wnr(G)=\Hom(G(F)/G(F)_1,\CC^\times).
\]
Since Kottwitz homomorphism is functorial and trivial on simply connected groups \cite[Proof of Prop. 11.5.4]{KalPra}, it follows that $G(F)^\natural\subseteq G(F)_1$. Therefore weakly unramified characters are cocentral. By definition we have an isomorphism
\[
\begin{tikzcd}
    \kappa_G:G(F)/G(F)_1 \arrow[r,"\sim"] & X^*((Z(\dual{G})^{I_F})_\Fr).
\end{tikzcd}
\]
Consequently, we obtain an isomorphism defined by
\[
\begin{tikzcd}
   (Z(\dual{G})^{I_F})_\Fr \arrow[r,"\sim"] & \Hom(X^*((Z(\dual{G})^{I_F})_\Fr),\CC^\times) \arrow[r,"\kappa_G^*"] & X^\wnr(G).
\end{tikzcd}
\]
As explained in \cite[\S3.3.1]{Haines}, it follows from \cite[Prop. 4.5.2]{KalethaEpipelagic} that this isomorphism is the restriction of the cocentral character homomorphism $a\mapsto\chi_a$ along the inflation homomorphism 
\[
\begin{tikzcd}
    (Z(\dual{G})^{I_F})_\Fr=H_c^1(W_F/I_F,Z(\dual{G})^{I_F}) \arrow[r,hook] & H_c^1(W_F,Z(\dual{G})).
\end{tikzcd}
\]
If we further restrict the cocentral character homomorphism to 
\[
(Z(\dual{G})^{I_F})_{\Fr}^\circ\defeq((Z(\dual{G})^{I_F})_{\Fr})^\circ,
\]
then we obtain an isomorphism
\[
\begin{tikzcd}
    (Z(\dual{G})^{I_F})_{\Fr}^\circ \arrow[r,"\sim"] \arrow[r,"\sim"] & X^\nr(G).
\end{tikzcd}
\]
This follows since the Kottwitz homomorphism descends to an isomorphism
\[
\begin{tikzcd}
    \kappa_G:G(F)/G(F)^1 \arrow[r,"\sim"] & X^*((Z(\dual{G})^{I_F})_{\Fr})/{\Tor}=X^*((Z(\dual{G})^{I_F})_{\Fr}^\circ),
\end{tikzcd}
\]
where $\Tor$ denotes the torsion subgroup, and therefore $\chi$ restricts to the isomorphism
\[
\begin{tikzcd}
    (Z(\dual{G})^{I_F})_{\Fr}^\circ \arrow[r,"\sim"] & \Hom(X^*((Z(\dual{G})^{I_F})_{\Fr}^\circ),\CC^\times) \arrow[r,"\kappa_G^*"] & X^\nr(G).
\end{tikzcd}
\]
The unramified cocentral character homomorphism factors as the composition of the surjective homomorphism
\[
\begin{tikzcd}
    Z(\dual{G})^{\Gamma_F,\circ} \arrow[r,two heads] & (Z(\dual{G})^{I_F})_{\Fr}^\circ
\end{tikzcd}
\]
with the restriction of the cocentral character homomorphism to $(Z(\dual{G})^{I_F})_{\Fr}^\circ$. Surjectivity of the above homomorphism can be seen as a consequence of the following commutative diagram, which summarises the above discussion.
\[
\begin{tikzcd}
    Z(\dual{G})^{\Gamma_F,\circ} \arrow[d,two heads] \arrow[dr,two heads] & \ak_{G,\CC}^* \arrow[d,two heads] \arrow[l,two heads] \\
    (Z(\dual{G})^{I_F})_\Fr^\circ \arrow[d,hook] \arrow[r,"\sim"] & X^\nr(G) \arrow[d,hook] \\
    (Z(\dual{G})^{I_F})_\Fr \arrow[d,hook] \arrow[r,"\sim"] & X^\wnr(G) \arrow[d,hook] \\
    H_c^1(W_F,Z(\dual{G})) \arrow[r,"\chi"] & \Hom_c(G(F)/G(F)^\natural,\CC^\times) \\
\end{tikzcd}
\]
The homomorphism $\ak_{G,\CC}^*\to H_c^1(W_F,Z(\dual{G}))$ in this diagram is the the homomorphism $\lambda\mapsto a_\lambda$. It follows that its image is $(Z(\dual{G})^{I_F})_\Fr^\circ$ and its kernel is $\ak_{G,F}^\vee$.

\subsubsection{Twists.}
Let $a\in H_c^1(W_F,Z(\dual{G}))$ and $\phi\in\Phi(G)$. Recall that we write $\phi(l)=(a_\phi(l),w(l))$. We define $a\cdot\phi\in\Phi(G)$ by $(a\cdot\phi)(l)=(a(w(l))a_\phi(l),w(l))$, and this gives a well-defined action of the group $H_c^1(W_F,Z(\dual{G}))$ on the pointed set $\Phi(G)$. The action of $H_c^1(W_F,Z(\dual{G})^1)$ preserves $\Phi_2(G)$ and $\Phi_\temp(G)$.

Pulling back along the homomorphism $\ak_{G,\CC}^*/\ak_{G,F}^\vee\to H_c^1(W_F,Z(\dual{G})), \lambda\mapsto a_\lambda$, we obtain an action of $\ak_{G,\CC}^*/\ak_{G,F}^\vee$ on $\Phi(G)$, and the action of $i\ak_{G}^*/\ak_{G,F}^\vee$ preserves $\Phi_2(G)$ and $\Phi_\temp(G)$. We define
\[
\Phi_\temp(G)_\CC=\ak_{G,\CC}^*\cdot\Phi_\temp(G) \quad,\quad \Phi_2(G)_\CC=\ak_{G,\CC}^*\cdot\Phi_2(G).
\]
The set $\Phi_2(G)_\CC$ is precisely the set of elements of $\Phi(G)$ that are elliptic but not necessarily tempered. We call the elements of $\Phi_2(G)_\CC$ essentially discrete $L$-parameters and the elements of $\Phi_\temp(G)_\CC$ essentially tempered $L$-parameters. We caution the reader that sometimes in the literature $\Phi_2(G)_\CC$ and $\Phi_2(G)$ are denoted by $\Phi_{2}(G)$ and $\Phi_{2,\temp}(G)$, respectively.

\subsubsection{Infinitesimal characters and $L$-parameters of real groups.}

Assume that $G$ is a real group. Let $\phi\in\Phi(G)$. As explained in \cite{LanReal}, one can attach an infinitesimal character $\mu_\phi\in\Hom_{\CC\alg}(\Zk(\gk_\CC),\CC)$ of $G$ to $\phi$ as follows. Choose a Borel pair $(B,T)$ of $G_{\CC}$ and a $\Gamma_\RR$-stable Borel pair $(\Bc^0,\Tc^0)$ of $\dual{G}$. We choose a representative of $\phi$ such that $\phi(\CC^\times)\subseteq\Tc^0$. There exists $\mu,\nu\in X_*(\Tc^0)\otimes_\ZZ\CC=\Lie(\Tc^0)$ with $\mu-\nu\in X_*(\Tc^0)$ such that $\phi(z)=z^{\mu}\ol{z}^\nu$ for all $z\in\CC^\times$. The $W(\dual{G},\Tc^0)$-orbit of $\mu$ does not depend on the choice of representative of $\phi$. The Borel pairs and the identification $\Psi(G)^\vee=\Psi(\dual{G})$ give us an isomorphism $X_*(\Tc^0)\cong X^*(T)$ and a compatible isomorphism $W(\dual{G},\Tc^0)\cong W(G_\CC,T)$. Using the identification $\tk^*=\Lie(T)^*$ with $X^*(T)\otimes_\ZZ\CC$, we obtain an isomorphism $\Lie(\Tc^0)/W(\dual{G},\Tc^0)\cong\tk^*/W(G_\CC,T)$. Thus, we obtain an element $\mu_\phi\in\tk^*/W(G_\CC,T)$. Recall that the Harish-Chandra isomorphism $\Zk(\gk_\CC)\cong\Sym(\tk)^{W(G_\CC,T)}$ gives an isomorphism $\Hom_{\CC\alg}(\Zk(\gk_\CC),\CC)\cong\tk^*/W(G_\CC,T)$. Thus, we have an infinitesimal character $\mu_\phi$ of $G$. It does not depend on any of the choices made.

\subsubsection{Classification of $L$-parameters.}
For each $M\in\Lc^G(M_0)$, recall that we have a canonical equivalence class of tempered $L$-homomorphisms $\iota_M^G:\Lgp{M}\to\Lgp{G}$. This gives rise to a pushforward map $\iota_M^G=(\iota_M^G)_*$ on $L$-parameters preserving temperedness. The images of $\Phi_2(M)_\CC$ in $\Phi(G)$ and $\Phi_2(M)$ in $\Phi_\temp(G)$ are both quotients for the natural actions of $W^G(M)$. We have a decomposition
\[
\Phi(G)=\coprod_{M\in\Lc^G(M_0)/W_0^G}\Phi_2(M)_\CC/W^G(M)
\]
and this restricts to a decomposition
\[
\Phi_\temp(G)=\coprod_{M\in\Lc^G(M_0)/W_0^G}\Phi_2(M)/W^G(M).
\]
This latter decomposition is an analogue for tempered $L$-parameters of the Harish-Chandra classification of irreducible tempered representations. The first decomposition is analogous to the classification of irreducible admissible representations obtained by combining the Langlands classification with the Harish-Chandra classification.

Implicit in \cite{LanReal} is a classification of $\Phi(G)$ that is analogous to the Langlands classification of $\Pi(G)$. This was elaborated on and given an explicit formulation and detailed proof in \cite{SilZin}. We define a Langlands datum for $L$-parameters to be a triple $((P,M),\phi,\lambda)$, where $(P,M)$ is a parabolic pair, $\phi\in\Phi_\temp(M)$, and $\lambda\in(\ak_M^*)^{P,>0}$. To each Langlands datum $((P,M),\phi,\lambda)$, we may assign an $L$-parameter $\iota_M^G(\phi_\lambda)\in\Phi(G)$ and the Langlands classification for $L$-parameters asserts that this gives a bijection between the set of $G(F)$-conjugacy classes of Langlands data and $\Phi(G)$. See \cite{SilZin} for more details and a description of the inverse of this bijection.

\subsection{The local Langlands correspondence}

The local Langlands correspondence for $G$, which is still hypothetical for $p$-adic groups in general, is a natural surjective ``local reciprocity'' map
\[
\rec_G:\Pi(G)\longrightarrow\Phi(G)
\]
with finite fibres. The fibre of $\rec_G$ above an $L$-parameter $\phi\in\Phi(G)$ is called the $L$-packet of $\phi$ and denoted by $\Pi_\phi$. The local reciprocity map $\rec_G$ is natural in the sense that it satisfies several desiderata, which say that $\rec_G$ is compatible with various structures (cf. \cite{Borel, KalTai}).

\begin{desiderata}\label{des} \
\begin{enumerate}
    \item\label{des:1} Compatibility with central characters: for each $\pi\in\Pi(G)$, we have $\zeta_\pi=\zeta_{\phi_\pi}$.
    \item\label{des:2} Compatibility with cocentral twists: for each representation $\pi\in\Pi(G)$ and cohomology class $a\in H^1(W_F,Z(\dual{G}))$, we have $\phi_{\pi\otimes\chi_a}=a\cdot\phi_\pi$. In particular, we have compatibility with unramified twists: for each $\pi\in\Pi(G)$ and $\lambda\in\ak_{G,\CC}^*$, we have $\phi_{\pi_\lambda}=(\phi_\pi)_\lambda$.
    \item\label{des:3} Compatibility with temperedness and discreteness: for $\pi\in\Pi(G)$, we have $\pi\in\Pi_\temp(G)$ (resp. $\pi\in\Pi_2(G)$) if and only if $\phi_\pi\in\Phi_\temp(G)$ (resp. $\phi_\pi\in\Phi_2(G)$).
    \item\label{des:4} Naturality: if $\eta:G\to H$ is a normal homomorphism with abelian kernel and cokernel and if $\pi\in\Pi(H)$, then the irreducible constituents of the restriction $\pi\circ\eta$, which is a direct sum of finitely many irreducible admissible representations of $G(F)$, all belong to $\Pi_{\Lgp{\eta}\circ\phi}$. In particular, if $\eta:G\to H$ is an isomorphism, then the following diagram commutes:
    \[
    \begin{tikzcd}
        \Pi(H) \arrow[r, "\Pi(\eta)"] \arrow[d, "\rec_H", swap] & \Pi(G) \arrow[d, "\rec_G"] \\
        \Phi(H) \arrow[r, "\Phi(\eta)"] & \Phi(G)
    \end{tikzcd}
    \]
    where $\Pi(\eta)=\eta^*$ and $\Phi(\eta)=(\Lgp{\eta})_*$.
    \item\label{des:5} Compatibility with Weil restriction: If $K/F$ is a finite subextension of $\ol{F}/F$, $G'$ is a connected reductive group over $K$, and $G=\Res_{K/F}G'$, then we have a commutative diagram
    \[
    \begin{tikzcd}
        \Pi(G') \arrow[r, "\sim"] \arrow[d,"\rec_{G'}"] & \Pi(G) \arrow[d, "\rec_G"] \\
        \Phi(G') \arrow[r, "\sim"] & \Phi(G)
    \end{tikzcd}
    \]
    where the upper horizontal arrow is the isomorphism coming from the identification $G(F)=G'(K)$ and the lower horizontal arrow is a canonical bijection coming from Shapiro's lemma. See \cite[\S4, \S5, \S8.4]{Borel} for more details.
    \item\label{des:6} Compatibility with the Langlands classification of $\Pi(G)$: if $((P,M),\phi,\lambda)$ is a Langlands triple, then $\Pi_{\iota_M^G(\phi)}$ is the set of Langlands quotients $J((P,M),\sigma,\lambda)\in\Pi(G)$, for $\sigma\in\Pi_\phi$.
    \item\label{des:7} Compatibility with the Harish-Chandra classification of $\Pi_\temp(G)$: if $M$ is a Levi subgroup of $G$ and $\phi\in\Phi_2(M)$, then $\Pi_{\iota_M^G(\phi)}$ is the set of irreducible subrepresentations of the various parabolically induced representations $I_M^G(\sigma)$ for $\sigma\in\Pi_\phi$.
    \item\label{des:8} Stable tempered characters: for each $\phi\in\Phi_\temp(G)$, there exists a stable virtual tempered representation
    \[
    \pi_\phi=\sum_{\pi\in\Pi_\phi}c_{\phi,\pi}\pi
    \]
    with $c_{\phi,\pi}\in\ZZ_{>0}$ for all $\pi\in\Pi_\phi$, such that the following properties hold.
    \begin{enumerate}
        \item\label{des:8a} Compatibility with parabolic induction: If $M$ is a Levi subgroup of $G$ and $\phi\in\Phi_\temp(M)$, then $I_M^G(\pi_\phi)=\pi_{\iota_M^G(\phi)}$.
        \item\label{des:8b} The set $\{\pi_\phi : \phi\in\Phi_\temp(G)\}$ is a basis of the space $D_\temp^\st(G)$ of stable tempered virtual representations of $G$. Consequently, the set $\{\pi_\phi : \phi\in\Phi_2(G)\}$ is a basis of the space $D_\el^\st(G)$.
    \end{enumerate}
    The character $\Theta_\phi=\Theta_{\pi_\phi}$ of $\pi_\phi$ is a stable tempered distribution called the stable tempered character of $\phi$, and called a stable discrete series character if $\phi\in\Phi_2(G)$. Note that the stable tempered virtual representations are linearly independent. We will often identify $\phi=\pi_\phi=\Theta_{\phi}$.
    \item\label{des:9} Compatibility with infinitesimal characters in the case $F=\RR$: for each $\pi\in\Pi(G)$, we have $\mu_\pi=\mu_{\phi_\pi}$.
\end{enumerate}
\end{desiderata}

It follows from \Cref{des:5} that the local Langlands correspondence for groups over arbitrary local fields of characteristic zero is determined by and can be constructed from the local Langlands correspondence for groups over $F=\RR$ or $F=\QQ_p$ for a prime $p$. It follows from the Langlands classification of admissible representations and the parallel Langlands classification of $L$-parameters that a local Langlands correspondence $\rec_G:\Pi(G)\to\Phi(G)$ satisfying \Cref{des} is determined by its restricted tempered local Langlands correspondence $\rec_G:\Pi_\temp(G)\to\Phi_\temp(G)$. Furthermore, it follows from what we have called the Harish-Chandra classification of tempered representations and the parallel Harish-Chandra classification of tempered $L$-parameters that a local Langlands correspondence satisfying \Cref{des} is determined by its restricted discrete local Langlands correspondence $\rec_G:\Pi_2(G)\to\Phi_2(G)$. Conversely, given a discrete local reciprocity map $\rec_G:\Pi_2(G)\to\Phi_2(G)$ satisfying the parts of \Cref{des} that make sense, one can construct a local reciprocity map $\rec_G:\Pi(G)\to\Phi(G)$ satisfying \Cref{des}. (See \cite[\S11.7]{Borel}, and \cite[\S6.1.2]{KalTai} for more on this.)

Note that $\Pi_\phi$ can be recovered from $\Theta_\phi$ since the characters of irreducible admissible representations are linearly dependent. Thus the local Langlands correspondence is determined by the parametrisation $\phi\mapsto\Theta_\phi$ of stable discrete series characters and can be constructed by defining stable discrete series characters in such a way that the associated discrete reciprocity map satisfies \Cref{des}. 

For real groups, and thus for complex groups by restriction of scalars, Langlands constructed the local Langlands correspondence in the above manner. Suppose that $G$ is a connected reductive group over $\RR$ and $\phi\in\Phi_2(G)$. Langlands showed that $G$ has an elliptic maximal torus $T$. Moreover, a choice of a Borel subgroup $B$ of $G_\CC$ containing $T$ determines an $L$-embedding $\Lgp{T}\to\Lgp{G}$ through which $\phi$ factors as a tempered $L$-parameter of $\phi_T\in\Phi_\temp(T)$. The unitary character of $\theta_\phi$ of $T(\RR)$  corresponding to $\phi_T$ under the local Langlands correspondence for tori is dominant with respect to $B$ and its orbit under the absolute Weyl group $W(G,T)$ does not depend on $B$. In Harish-Chandra's papers on discrete series, he defined a stable tempered distribution $\Theta_{\theta_\phi}$ that is uniquely determined by its formula on the set $T(\RR)'$ of regular elements in $T(\RR)$:
\[
\Theta_{\theta_\phi}(t)=(-1)^{q(G)}\sum_{w\in W(G,T)}\frac{\theta_\phi(w^{-1}\cdot t)}{\prod_{\alpha>_B0}(1-\alpha(w^{-1}\cdot t)^{-1})},
\]
where $q(G)=1/2\dim(G(\RR)/K)$ for any maximal compact subgroup $K$ of $G(\RR)$ and $\alpha>_B0$ indicates that $\alpha$ is a positive root with respect to the positive system determined by $B$. Harish-Chandra showed that $\Theta_{\theta_\phi}$ is the sum of $\Theta_\pi$ for all $\pi\in\Pi_2(G)$ with $\zeta_\pi=\theta_\phi|_{Z_G(\RR)}$ and $\mu_\pi=d\theta_\phi$. Note that we have $\zeta_\phi=\theta_\phi|_{Z_G(\RR)}$ and $\mu_\phi=d\theta_\phi$. Langlands constructed the local Langlands correspondence for $G$ by requiring that $\Theta_\phi=\Theta_{\theta_\phi}$. Consequently, we have that $\Pi_\phi$ consists of all $\pi\in\Pi_2(G)$ with infinitesimal character $\mu_\pi=\mu_\phi$ and central character $\zeta_\pi=\zeta_\phi$, and we have $\Theta_\phi=\sum_{\pi\in\Pi_\phi}\Theta_\pi$. In fact, in \cite{AdamsVogan} Adams and Vogan prove that $\Pi_\phi$ consists of all $\pi\in\Pi_2(G)$ with infinitesimal character $\mu_\pi=\mu_\phi$ and what they call ``split radical character'' $\zeta_\pi|_{A_G(\RR)}=\zeta_\phi|_{A_G(\RR)}$. For $\phi\in\Phi_2(G)$, the cardinality of $\Pi_\phi$ is bounded by $|W(G,T)/W_\RR(G,T)|$, where $T$ is an elliptic torus of $G$. For $\phi\in\Phi_\temp(G)$, we also have $\pi_\phi=\sum_{\pi\in\Pi_\phi}\pi$. This follows as in the proof of \cite[Lemma 3.1]{She79} using \cite[Theorem]{KnaComm} and \cite[Theorem 2.9]{SpeVog}. \Cref{des:8b} is a corollary of \cite[Lemma 18.11]{ABV}, which follows from the proof of \cite[Lemma 5.2]{She79}. The $L$-packets for complex groups are singletons, which agrees with the fact that conjugacy and stable conjugacy coincide for complex groups.

For non-archimedean groups, it is expected that stable discrete series characters $\Theta_\phi$ can be constructed directly in terms of $\phi$, as was the case for real groups. This would then give rise to construction of the local Langlands correspondence and a characterisation of it. (See \cite{KalDouble,KalICM} for further discussion.)

We will make use of the following hypothesis on the local Langlands correspondence for non-archimedean groups.

\begin{hypothesis} \label{hypothesis}
    Let $G$ be a connected reductive group over a non-archimedean local field $F$. The group $G$ and its Levi subgroups have local reciprocity maps satisfying the following subset of \Cref{des}: \Cref{des:1}, \Cref{des:2} in the special case of unramified twists, \Cref{des:3}, \Cref{des:4} in the special case when $\eta$ is of the form $\Int(g):M\to g\cdot M$ for some $g\in G(F)$ and $M$ a Levi subgroup of $G$, \Cref{des:6}, \Cref{des:7}, and \Cref{des:8}.
\end{hypothesis}

\Cref{hypothesis} is known for tori. The local reciprocity map is the cocentral character isomorphism given by the Langlands pairing and the stable tempered character of a tempered $L$-parameter can be taken to simply be the corresponding character. 

In the case $G=\GL_n$, \Cref{hypothesis} was proved by Harris and Taylor, Henniart, and Scholze \cite{HarTay, Henniart, Scholze}. The $L$-packets are singletons and the stable tempered characters are simply the irreducible tempered characters.

If $G$ is an inner form of a quasisplit symplectic, odd special orthogonal, unitary, or odd general spin group, then \Cref{hypothesis} is satisfied. This follows from the works of Arthur, Mok, and Moeglin  \cite{ArthurEndo, MokEndo, MoegEndo}, which use the theory of twisted endoscopy and the local Langlands correspondence for $\GL_n$. See \cite{varma}, where this is established carefully.

\textbf{From now on, if $F$ is non-archimedean we assume \Cref{hypothesis} for all connected reductive groups over $F$.}

\subsubsection{Infinitesimal characters for $p$-adic groups.}

Suppose that $F$ is archimedean. An infinitesimal character (also called an infinitesimal parameter) of $G$ is defined to be a $\dual{G}$-conjugacy class of $L$-homomorphisms $\mu:W_F\to\Lgp{G}$. This notion originates in \cite{VoganLLC}. Another useful reference is \cite{CunninghamEtAl}. Every $L$-parameter $\phi:L_F=\SL_2(\CC)\times W_F\to\Lgp{G}$ of $G$ has an associated infinitesimal character $\mu_\phi:W_F\to\Lgp{G}$ defined by $\mu_\phi(w)=\phi(d_w,w)$, where $d_w=\diag(\|w\|^{1/2},\|w\|^{-1/2})$. Each $\phi\in\Phi_\temp(G)$ can be recovered from its infinitesimal character $\mu_\phi$ \cite[\S6.6]{varma}. In general, there are at most finitely many elements of $\Phi(G)$ with a given infinitesimal character \cite[Corollary 4.6]{VoganLLC}. 

For each $\pi\in\Pi(G)$, we define its infinitesimal character $\mu_\pi$ by $\mu_\pi=\mu_{\phi_\pi}$. It follows from the finiteness of $L$-packets that for each infinitesimal character $\mu$ of $G$, there are finitely many $\pi\in\Pi(G)$ with $\mu_\pi=\mu$.

Let $M$ be a Levi subgroup of $G$. If $\mu$ is an infinitesimal character of $M$, then $\iota_M^G\circ\mu$ is an infinitesimal character of $G$. Moreover, if $\mu_1$ and $\mu_2$ are infinitesimal characters of $M$, then $\iota_M^G\circ\mu_1=\iota_M^G\circ\mu_2$ if and only if $\mu_1$ and $\mu_2$ are in the same $W^G(M)$-orbit.

For each $\tau=(M,\sigma,\wt{r})\in T_\temp(G)$, we define its infinitesimal character to be $\mu_\tau=\iota_M^G\circ\mu_\sigma$. Note that $\mu_\tau$ only depends on the image of $\tau$ in $T_\temp(G)/\CC^1$. There are at most finitely many elements in $T_\temp(G)/\CC^1$ with a given infinitesimal character. 

We have an action of $\ak_{G,\CC}^*/\ak_{G,F}^\vee$ on the set of infinitesimal characters by $(\lambda,\mu)\mapsto a_\lambda\cdot\mu$. The assignment of infinitesimal characters to elements of $\Pi(G)$ (resp. $T_\temp(G)$) is equivariant with respect to the actions of $i\ak_G^*$ (resp. $\ak_{G,\CC}^*$).

\subsection{Stable Paley--Wiener Theorems} \label{sec:stablePW}

\subsubsection{The stable Fourier transform.}
For $f\in\Cc(G,\zeta)$, we define its stable Fourier transform to be the function $f^G:\Phi_\temp(G,\zeta)\to\CC$ defined by $f^G(\phi)=\Theta_\phi(f)$.

We define the space of stable Fourier transforms
\[
\wh{\Sc_{(c)}}(G,\zeta)=\{f^G : f\in\Cc_{(c)}(G,\zeta)\}.
\]
We have
\[
\wh{\Sc_{(c)}}(G,\zeta)=\Cc_{(c)}(G,\zeta)/\Ann_{\Cc_{(c)}(G,\zeta)}(\{\Theta_\phi : \phi\in\Phi_\temp(G,\zeta)\})
\]
and we give it the natural quotient topology. The stable Fourier transform is a continuous surjective linear map $\Fc^\st:\Cc_{(c)}(G,\zeta)\to\wh{\Sc_{(c)}}(G,\zeta)$.

A simple consequence of the Harish-Chandra regularity theorem and the stable Weyl integration formula is the following. Let $\phi\in\Phi(G,\zeta)$ (resp. $\phi\in\Phi_\temp(G,\zeta)$). If $f\in\Cc_c^\unst(G,\zeta)$ (resp. $f\in\Cc^\unst(G,\zeta)$), then $f^G(\phi)=0$. Consequently, we have $\Theta_\phi\in\Sc_c(G,\zeta)'$ (resp. $\Theta_\phi\in\Sc(G,\zeta)'$).

It follows that the stable Fourier transform $\Fc^\st:\Cc_{(c)}(G,\zeta)\to\wh{\Sc_{(c)}}(G,\zeta)$ descends to a continuous surjective linear map
\[
\Fc^\st:\Sc_{(c)}(G,\zeta)\longrightarrow\wh{\Sc_{(c)}}(G,\zeta).
\]
We call the property of injectivity of this map  stable spectral density for $\Sc_{(c)}(G,\zeta)$. Stable spectral density for $\Sc_{(c)}(G,\zeta)$ is equivalent to
\[
\Cc_{(c)}^\unst(G,\zeta)=\Ann_{\Cc_{(c)}(G,\zeta)}(\{\Theta_\phi : \phi\in\Phi_\temp(G,\zeta)\}).
\]
It follows from stable spectral density for $\Cc_{(c)}(G,\zeta)$ that a distribution in $\Cc_{(c)}(G,\zeta)'$ is stable if and only if it lies in the weak-$*$ closure of the linear span of $\{\Theta_\phi\}_{\phi\in\Phi_\temp(G,\zeta)}$. 

For archimedean $F$, stable spectral density for $\Cc(G,\zeta)$ was proved by Shelstad in \cite[Lemma 5.3]{She79} for $\Zc=1$ and in \cite[Theorem 4.1]{SheTERGI} for $\Zc=Z(F)$, where $Z$ is a central torus of $G$. For non-archimedean $F$, stable spectral density for $\Cc_c(G,\zeta)$ is an immediate consequence of \cite[Theorem 6.1, Theorem 6.2]{ArthurRelations} when $G$ is quasi-split and $\Zc=Z(F)$ for $Z$ a central induced torus of $G$. The assumption that $G$ is quasisplit can be removed using endoscopic transfer between a non-quasisplit group $G$ and its quasisplit inner form $G^*$, as explained in \cite[\S3.2]{varma}. Arthur's proof uses a global argument involving the simple trace formula, and it is generalised to the twisted setting in \cite[XI.5.2]{MWII} in the case of a trivial central character. The reader may find this a helpful reference since the spectral density result and proof are stated explicitly there.  Thus, we have the following result.

\begin{theorem} \label{StableDensity}
    Let $(\Zc,\zeta)$ be a unitary central datum with $\Zc=Z(F)$ for $Z$ a central induced torus of $G$. Stable spectral density for $C_c^\infty(G,\zeta)$ holds.
\end{theorem}

In particular, we have stable spectral density for $\Cc_c(G)$ and $C_{c,\cusp}^\infty(G,\zeta,K)$ when $Z=A_G(F)$, and we will use these to prove the stable Paley--Wiener theorems below. Stable spectral density for $\Cc(G)$ in the non-archimedean case does not appear to be in the literature. It is established below (\Cref{StableSpectralDensity}) in the course of the proof of the stable Paley--Wiener theorem for $\Cc(G)$.

\subsubsection{The space of $L$-parameters.}
In order to define Paley--Wiener and Schwartz spaces on $\Phi_\temp(G)$, we must recall how it is naturally a topological space. Recall that we have an action of $\ak_{G,\CC}^*$ on $\Phi(G)$, and that the action of $i\ak_{G}^*$ preserves $\Phi_\temp(G)$ and $\Phi_2(G)$. We denote the isotropy subgroup of $\phi\in\Phi(G,\zeta)$ in $\ak_{G,\CC}^*$ by $\ak_{G,\phi}^\vee$. Just as was the case for $\ak_{G,\tau}^\vee$, we have
\[
\ak_{G,F}^\vee\subseteq\ak_{G,\phi}^\vee\subseteq\wt{\ak}_{G,F}^\vee
\]
Thus, if $F$ is archimedean we have $\ak_{G,\phi}^\vee=0$, and if $F$ is non-archimedean we have that $\ak_{G,\phi}^\vee$ is a full lattice in $i\ak_G^*$. Consequently, $\Phi_2(G)$ is naturally a smooth manifold with countably many components $i\ak_{G}^*\cdot\phi=i\ak_{G}^*/\ak_{G,\phi}^\vee$, which are Euclidean spaces if $F$ is archimedean and compact tori if $F$ is non-archimedean. Moreover, $\Phi_2(G)_\CC$ is the complexification of $\Phi_2(G)$ with connected components $\ak_{G,\CC}^*\cdot\phi=\ak_{G,\CC}^*/\ak_{G,\phi}^\vee$. Since
\[
\Phi_\temp(G)=\coprod_{L\in\Lc^G(M_0)/W_0^G}\Phi_2(L)/W^G(L),
\]
we have that $\Phi_\temp(G)$, and similarly $\Phi(G)$ is naturally a topological space. We say that a function on $\Phi_\temp(G)$ is smooth if it pulls back to a smooth function on each $\Phi_2(L)$.

\subsubsection{The stable Paley--Wiener Theorems.}
We define $\|\phi\|=\|\mu_\phi\|$ in the archimedean case. Note that for any countable set $E\subseteq\Phi_2(L)$, we have Paley--Wiener and Schwartz spaces defined on the space $\Lambda=\bigcup_{\phi\in E}\Lambda_\phi$ with $\Lambda_\phi=i\ak_{L}^*/\ak_{L,\phi}^\vee$.

We define $\Ss_\el^\st(L)$ to be the space of smooth functions $\varphi:\Phi_2(L)$ such that for some (and hence any) set of representatives $B_\el^\st(L)\subseteq\Phi_2(L)$ for the connected components of $\Phi_2(L)$, we have 
\[
\varphi\in \Ss\Bigg(\coprod_{\phi\in B_\el^\st(L)}i\ak_{L}^*/\ak_{L,\phi}^\vee\Bigg).
\]
We define $PW_\el^\st(L)$ (resp. $PW_{\el,f}^\st(L)$) in the same way as $\Ss_\el^\st(L)$, except that we replace $\Ss(\cdot)$ by $PW(\cdot)$ (resp. $PW_f(\cdot)$).

We define
\[
\Ss^\st(G)=\Bigg(\bigoplus_{L\in\Lc^G(M_0)}\Ss_\el^\st(L)\Bigg)^{W_0^G}=\bigoplus_{L\in\Lc^G(M_0)/W_0^G}\Ss_\el^\st(L)^{W^G(L)}
\]
and similarly we define $PW^\st(G)$ and $PW_f^\st(G)$. These are naturally spaces of smooth functions on $\Phi_\temp(G)$. 

\begin{example}
    We describe the above spaces explicitly in the case when $G=\SL_2$ and $F=\RR$. Let $T=\SO_2$, an anisotropic maximal torus of $\SL_2$. We write the elements of $T(\RR)$ as
    \[
    t(\theta)=
    \begin{pmatrix}
        \cos\theta & -\sin\theta \\
        \sin\theta & \cos\theta
    \end{pmatrix}.
    \]
    Let $S$ be the subgroup of diagonal elements of $\SL_2$, a split maximal torus. For $x\in\RR^\times$, we write $s(x)=\diag(x,x^{-1})\in S(\RR)$. The set $\{S,\SL_2\}$ is a set of representatives for the conjugacy classes of Levi subgroups of $G$. The Weyl group $W_0^G=W^G(S)$ is of order two, with the non-trivial element acting by $s(x)\mapsto s(x^{-1})$.
    
    The space $\Phi_2(G)=\{\phi_n\}_{n=1}^\infty$ is countable and discrete, with $\phi_n$ denoting the $L$-parameter whose $L$-packet consists of the discrete series representations $\pi_{\pm n}$, whose characters are given on the regular elements of $T(\RR)$ by the formula
    \[
    \Theta_{\pi_{\pm n}}(t(\theta))=-\frac{\pm e^{\pm in\theta}}{e^{i\theta}-e^{-i\theta}}.
    \]
    The stable discrete series character of $\phi_n$ is
    \[
    \Theta_{\phi_n}(t(\theta))=-\frac{e^{in\theta}-e^{-in\theta}}{e^{i\theta}-e^{-i\theta}}.
    \]
    We identify $\Phi_2(G)=\ZZ_{>0}$ via $\phi_n\mapsto n$. We have
    \[
    \Ss_\el^\st(G)=PW_\el^\st(G)=\Ss(\ZZ_{>0}),
    \]
    the space of rapidly decreasing functions $\varphi:\ZZ_{>0}\to\CC$. Furthermore,
    \[
    PW_{\el,f}^\st(G)=C_c(\ZZ_{>0}),
    \]
    the space of finitely supported functions $\varphi:\ZZ_{>0}\to\CC$.

    By the local Langlands correspondence for $S$, we may identify $\Phi_2(S)$ with $\Pi_u(S)$, the set of unitary characters of $S(\RR)\cong\RR^\times$. The unitary unramified characters of $S(\RR)$ are those of the form $s(x)\mapsto|x|^{i\lambda}$ for $\lambda\in\RR$. An arbitrary unitary character of $S(\RR)$ can be written as
    \[
    \chi_{m,\lambda}(s(x))=\sgn(x)^m|x|^{i\lambda},
    \]
    for unique $m\in\{0,1\}$ and $\lambda\in\RR$. The non-trivial element of $W^G(S)$ acts $\Pi_u(S)$ by $\chi_{m,\lambda}\mapsto\chi_{m,-\lambda}$. We identify $\Phi_2(S)=\{0,1\}\times\RR=\RR\coprod\RR$ via $\chi_{m,\lambda}\mapsto (m,\lambda)$. Then $\Phi_2(S)/W_0^G=\RR/\{\pm1\}\coprod\RR/\{\pm1\}=\RR_{\geq0}\coprod\RR_{\geq0}$. We have
    \[
    \Ss_\el^\st(S)=\Ss(\RR)\oplus\Ss(\RR)
    \]
    and
    \[
    \Ss_\el^\st(S)^{W^G(S)}=\Ss(\RR)^{\{\pm1\}}\oplus\Ss(\RR)^{\{\pm1\}}.
    \]
    Similarly, we have 
    \[
    PW_\el^\st(S)=PW_{\el,f}^\st(S)=PW(\RR)\oplus PW(\RR)
    \]
    and
    \[
    PW_\el^\st(S)^{W^G(S)}=PW_{\el,f}^\st(S)^{W^G(S)}=PW(\RR)^{\{\pm1\}}\oplus PW(\RR)^{\{\pm1\}}.
    \]
    Finally, we have
    \begin{align*}
        \Ss^\st(G)&=\Ss(\ZZ_{>0})\oplus\Ss(\RR)^{\{\pm1\}}\oplus\Ss(\RR)^{\{\pm1\}}, \\
        PW^\st(G)&=\Ss(\ZZ_{>0})\oplus PW(\RR)^{\{\pm1\}}\oplus PW(\RR)^{\{\pm1\}}, \\
        PW_f^\st(G)&=C_c(\ZZ_{>0})\oplus PW(\RR)^{\{\pm1\}}\oplus PW(\RR)^{\{\pm1\}}.
    \end{align*}
\end{example}

The aim of this section is to prove the following result, which contains three stable Paley--Wiener theorems. This was stated as \Cref{stablePW} in \Cref{sec:intro}.

\begin{theorem} \label{thm:stablePW}
    If $F$ is non-archimedean, we assume \Cref{hypothesis} for $G$. The stable Fourier transform is an isomorphism of topological vector spaces
    \[
    \Sc(G)\longrightarrow\Ss^\st(G)
    \]
    and restricts to isomorphisms of topological vector spaces
    \[
    \Sc_c(G)\longrightarrow PW^\st(G)
    \]
    and
    \[
    \Sc_f(G)\longrightarrow PW_f^\st(G).
    \]
\end{theorem}

Arthur proved the stable Paley--Wiener theorem for test functions on a quasisplit $p$-adic group as a consequence of his more general result \cite[Theorem 6.1]{ArthurRelations} on collective endoscopic transfer (see the discussion below the statement of Theorem 6.2 in \cite{ArthurRelations}). Moeglin--Waldspurger proved the stable Paley--Wiener theorem for test functions and $K$-finite test functions on a quasisplit real group (see the theorem in \cite[Ch. IV, \S4.2.3]{MWI} and the corollary in \cite[Ch. IV, \S4.2.9]{MWI}). In fact, they work in the more general setting of twisted spaces (also called twisted groups). What is new in our theorem is the stable Paley--Wiener theorem for Schwartz functions, i.e. the first isomorphism in our theorem, and the treatment of groups that are not quasisplit. Generalising Arthur's proof to Schwartz functions would require generalising several results that we do not need. Instead, we follow the strategy in the proof by Moeglin--Waldspurger, which is simplified because we work with groups (not general twisted groups) and make use of stable spectral density for test functions. This necessitates revisiting the cases of test functions and $K$-finite test functions.

Before turning to the proofs of the stable Paley--Wiener theorems, we remark that there is a stable Fourier inversion theorem for stable orbital integrals. Define a measure on $\Phi_2(L)$ for each $L\in\Lc^G(M_0)$ by
\[
\int_{\Phi_2(L)}\alpha(\phi)\dd{\phi}=\sum_{\phi\in\Phi_2(L)/i\ak_L^*}\int_{i\ak_L^*/\ak_{L,\phi}^\vee}\alpha(\phi_\lambda)\dd{\lambda},
\]
for all $\alpha\in C_c(\Phi_2(L))$. Then, define a measure on $\Phi_\temp(G)$ by
\[
\int_{\Phi_\temp(G)}\alpha(\phi)\dd{\phi}=\sum_{L\in\Lc^G(M_0)/W_0^G}|W^G(L)|^{-1}\int_{\Phi_2(L)}\alpha(\phi)\dd{\phi}
\]
for all $\alpha\in C_c(\Phi_\temp(G))$. As in the proof of \cite[Lemma 6.3]{ArthurRelations}, one can use \Cref{thm:inversion} to prove the following.

\begin{proposition} \label{thm:stableinversion}
    There exists a smooth function $S^G(\delta,\phi)$ on $\Delta_\sr(G)\times\Phi_\temp(G)$ such that
    \[
    f^G(\delta)=\int_{\Phi_\temp(G)}S^G(\delta,\phi)f^G(\phi)\dd{\phi}
    \]
    for all $f\in\Cc(G)$ and $\delta\in\Delta_\sr(G)$.
\end{proposition}

Writing $S^G(\phi,\delta)$ for the normalised stable tempered character $S^G(\phi,\delta)=|D^G(\delta)|^{1/2}\Theta_\phi(\delta)$, we have the dual formula
\[
f^G(\phi)=\int_{\Delta_\rs(G)}S^G(\phi,\delta)f^G(\delta)\dd{\delta}.
\]

\subsection{Proofs of the stable Paley--Wiener theorems}

\subsubsection{Stable and unstable elliptic tempered characters.}

Recall that the stable tempered characters are linearly independent and that we identify tempered $L$-parameters $\phi\in\Phi_\temp(G,\zeta)$ with their associated virtual tempered representations $\pi_\phi=\sum_{\pi\in\Pi_\phi}c_\phi(\pi)\pi$ and with their stable tempered distribution characters $\Theta_\phi=\sum_{\pi\in\Pi_\phi}c_\phi(\pi)\Theta_\pi$.

By \Cref{hypothesis}, we have $D_\temp^\st(G,\zeta)=\CC\Phi_\temp(G,\zeta)$ and $D_\el^\st(G,\zeta)=\CC\Phi_2(G,\zeta)$. We define the space $D_\el^\unst(G,\zeta)$ to be the subspace of $D_\el(G,\zeta)$ consisting of all elements orthogonal to $D_\el^\st(G,\zeta)$ with respect to the elliptic inner product.

\begin{lemma}
    We have $D_\el(G,\zeta)=D_\el^\st(G,\zeta)\oplus D_\el^\unst(G,\zeta)$. 
\end{lemma}

Let us set some notation and terminology. Let $V$ be an inner product space (not necessarily complete) and let $W$ be a subspace of $V$. We say that $W$ has an orthogonal complement in $V$ if there exists a subspace $X$ of $V$ such that $X$ is orthogonal to $W$ and $V=W\oplus X$, in which case $X$ is
\[
W^\perp=\Ann_V(W)=\{v\in V : \langle W,v\rangle=0\},
\]
the largest subspace of $V$ orthogonal to $W$. If $V$ is complete and $W$ is closed (in particular, if $V$ is finite-dimensional), then $W$ has an orthogonal complement in $V$.

\begin{proof}
    We have $D_\el^\st(G,\zeta)=\CC\Phi_2(G,\zeta)=\bigoplus_{\phi\in\Phi_2(G,\zeta)}\CC\phi$. For each $\phi\in\Phi_2(G,\zeta)$, the space $\CC\phi$ has an orthogonal complement in the finite-dimensional space $\CC\Pi_\phi$. Since 
    \[
    \CC\Pi_2(G,\zeta)=\bigoplus_{\phi\in\Phi_2(G,\zeta)}\CC\Pi_\phi,
    \]
    it follows that $D_\el^\st(G,\zeta)$ has an orthogonal complement in $\CC\Pi_2(G,\zeta)$, namely 
    \[
    \Ann_{\CC\Pi_2(G,\zeta)}(D_\el^\st(G,\zeta))=\bigoplus_{\phi\in\Phi_2(G,\zeta)}\Ann_{\CC\Pi_\phi}(\CC\phi).
    \]
    Since 
    \[
    D_\el(G,\zeta)=\CC\Pi_2(G,\zeta)\oplus\CC(T_\el(G,\zeta)\setminus\CC^1\Pi_2(G,\zeta))
    \]
    it follows that
    \[
    D_\el^\unst(G,\zeta)=\Ann_{\CC\Pi_2(G,\zeta)}(D_\el^\st(G,\zeta))\oplus\CC(T_\el(G,\zeta)\setminus\CC^1\Pi_2(G,\zeta))
    \]
    and that $D_\el(G,\zeta)=D_\el^\st(G,\zeta)\oplus D_\el^\unst(G,\zeta)$.
\end{proof}

\subsubsection{Stable pseudocoefficients.}
Let $\zeta$ be a unitary character of $A_G(F)$. Recall the Paley--Wiener space $PW_{\el,f}(G,\zeta)$. Fix a set of representatives $E_\el(G,\zeta)\subseteq T_\el(G,\zeta)$ for $T_\el(G,\zeta)/\CC^1$ such that $\Pi_2(G,\zeta)\subseteq E_\el(G,\zeta)$. Then $E_\el(G,\zeta)$ is a basis of $D_\el(G,\zeta)$. As explained above, we have an identification $PW_{\el,f}(G,\zeta)=\bigoplus_{\tau\in E_{\el}(G,\zeta)}\CC$ defined by $\varphi\mapsto(\varphi(\tau))_{\tau\in E_\el(G,\zeta)}$.

Let $B_\el^\st(G,\zeta)=\Phi_2(G,\zeta)$, which is an orthogonal basis of $D_\el^\st(G,\zeta)$ (with respect to the elliptic inner product). Let $B_\el^\unst(G,\zeta)$ be an orthogonal basis of $D_\el^\unst(
G,\zeta)$ and define $B_\el(G,\zeta)=B_\el^\st(G,\zeta)\cup B_\el^\unst(G,\zeta)$. Then $B_\el(G,\zeta)$ is an orthogonal basis of $D_\el(G,\zeta)$. For $?=\st,\unst$, we define
\[
PW_{\el,f}^?(G,\zeta)=\bigoplus_{b\in B_\el^?(G,\zeta)}\CC.
\]

We identify an element of $PW_{\el,f}(G,\zeta)$ (resp. $PW_{\el,f}^?(G,\zeta)$) with the linear functional on $D_\el(G,\zeta)$ (resp. $D_\el^?(G,\zeta)$) it defines by linear extension. Thus, for each $\varphi\in PW_{\el,f}(G,\zeta)$ the value $\varphi(b)$ is defined for all $b\in B_\el(G,\zeta)$. Explicitly, for each $b\in B_\el(G,\zeta)$, let us write $b=\sum_{\tau\in E_\el(G,\zeta)}c_{b,\tau}\tau$ for $c_{b,\tau}\in\CC$. Then $\varphi(b)=\sum_{\tau\in E_\el(G,\zeta)}c_{b,\tau}\tau$. We have an isomorphism
\[
\begin{tikzcd}
    PW_{\el,f}(G,\zeta) \arrow[r,"\sim"] & PW_{\el,f}^\st(G,\zeta)\oplus PW_{\el,f}^\unst(G,\zeta)
\end{tikzcd}
\]
defined by 
\[
(\varphi(\tau))_{\tau\in E_\el(G,\zeta)}\mapsto (\varphi(b))_{b\in B_\el(G,\zeta)}=((\varphi(b))_{b\in B_\el^\st(G,\zeta)},(\varphi(b))_{b\in B_\el^\unst(G,\zeta)}),
\]
that is, by change-of-basis. We identify $PW_{\el,f}(G,\zeta)=PW_{\el,f}^\st(G,\zeta)\oplus PW_{\el,f}^\unst(G,\zeta)$ via the above isomorphism. Note that $PW_{\el,f}^\unst(G,\zeta)$ is the subspace of functions in $PW_{\el,f}(G,\zeta)$ that vanish on $D_\el^\st(G,\zeta)$ and thus does not depend on the choice of $B_\el^\unst(G,\zeta)$.

Let $\Ic_{f,\cusp}^\st(G,\zeta)$ be a closed subspace of $\Ic_{f,\cusp}(G,\zeta)$ consisting of all functions that are constant on regular semisimple stable classes. Let $\Sc_{f,\cusp}(G,\zeta)$ denote the image of $\Ic_{f,\cusp}(G,\zeta)$ in $\Sc_{f}(G,\zeta)$ and let $\Ic_{f,\cusp}^\unst(G,\zeta)$ denote the kernel of the quotient map
\[
\Ic_{f,\cusp}(G,\zeta)\longrightarrow\Sc_{f,\cusp}(G,\zeta).
\]
Note that $\Ic_{f,\cusp}^\st(G,\zeta)\cap\Ic_{f,\cusp}^\unst(G,\zeta)=0$.

\begin{proposition}\label{lem:pseudo}
    The invariant Fourier transform
    \[
    \begin{tikzcd}
        \Ic_{f,\cusp}(G,\zeta)\arrow[r,"\sim"] & PW_{\el,f}(G,\zeta)
    \end{tikzcd}
    \]
    restricts to isomorphisms
    \[
    \begin{tikzcd}
        \Ic_{f,\cusp}^{?}(G,\zeta)\arrow[r,"\sim"] & PW_{\el,f}^?(G,\zeta)
    \end{tikzcd}
    \]
    Consequently, we have
    \[
    \Ic_{f,\cusp}(G,\zeta)=\Ic_{f,\cusp}^{\st}(G,\zeta)\oplus\Ic_{f,\cusp}^{\unst}(G,\zeta).
    \]
\end{proposition}

\begin{proof}
    It suffices to prove that the subspace of $\Ic_{f,\cusp}(G,\zeta)$ corresponding to the summand $PW_{\el,f}^?(G,\zeta)$ lies in $\Ic_{f,\cusp}^?(G,\zeta)$. 
    
    Let $b\in B_\el^?(G,\zeta)$. The element $\delta_b=(\delta_{b,b'})_{b'\in B_\el(G,\zeta)}\in PW_{\el,f}^?(G,\zeta)$ has its inverse invariant Fourier transform equal to $\|b\|_\el^{-2}f[b]_G\in\Ic_{f,\cusp}(G,\zeta)$. Thus, it suffices to prove that the pseudocoefficient $f[b]_G$ lies in $\Ic_{f,\cusp}^?(G,\zeta)$
    
    If $b\in B_\el^\unst(G,\zeta)$, then $f[b]_G(\phi)=0$ for all $\phi\in\Phi_2(G,\zeta)=B_\el^\st(G,\zeta)$, and thus $f[b]_G\in\Ic_{f,\cusp}^\unst(G,\zeta)$ by stable spectral density for $C_{c,\cusp}^{\infty}(G,\zeta,K)$.
    
    Suppose that $b\in B_\el^\st(G,\zeta)=\Phi_2(G,\zeta)$ and let us write $b=\phi$ and $f[b]_G=f[\phi]_G$. Let $\gamma\in\Gamma_{\rs}(G)$. If $\gamma$ is non-elliptic, then $f[\phi]_G(\gamma)=0$. If $\gamma$ is elliptic, then
    \[
    f[\phi]_G(\gamma)=m(\gamma)^{-1}|D^G(\gamma)|^{1/2}\ol{\Theta_\phi(\gamma)}.
    \]
    We have that $D^G$ and $\Theta_\phi$ are constant on regular semisimple stable classes. Thus, it suffices to show that $m(\gamma)$ is constant on regular semisimple stable classes. By definition, we have $m(\gamma)=\vol(G_\gamma(F)/A_G(F))$. Suppose that $\gamma'$ is stably conjugate to $\gamma$. Then there exists $g\in G$ such that $\Int(g):G_\gamma\to G_{\gamma'}$ is defined over $F$. The Haar measures on $G_\gamma(F)$ and $G_{\gamma'}(F)$ are normalised so that they correspond under $\Int(g)$, and thus $m(\gamma)=m(\gamma')$.
\end{proof}

We obtain the following stable Paley--Wiener theorem as a corollary

\begin{corollary}
    The stable Fourier transform gives an isomorphism
    \[
    \begin{tikzcd}
        \Sc_{f,\cusp}(G,\zeta) \arrow[r,"\sim"] & PW_{\el,f}^\st(G,\zeta).
    \end{tikzcd}
    \]
\end{corollary}

\begin{proof}
    The composition of the invariant Fourier transform 
    \[
    \Fc:\Ic_{f,\cusp}(G,\zeta)\longrightarrow PW_{\el,f}(G,\zeta)
    \]
    with the natural projection $PW_{\el,f}(G,\zeta)\to PW_{\el,f}^\st(G,\zeta)$ is the stable Fourier transform. It is surjective and its kernel is $\Ic_{f,\cusp}^{\unst}(G,\zeta)$ by \Cref{lem:pseudo}, so it descends to an isomorphism $\Sc_{f,\cusp}(G,\zeta)\to PW_{\el,f}^\st(G,\zeta)$.
\end{proof}

\subsubsection{The cuspidal case.} \label{sec:cuspidalcase}

Let $\zeta$ be a unitary character of $A_G(F)$, and let $\mu$ be an infinitesimal character of $G$. We denote by $T_\el(G,\zeta,\mu)$ the set of $\tau\in T_\el(G,\zeta)$ with $\mu_\tau=\mu$. Define $D_\el(G,\zeta,\mu)=\CC T_\el(G,\zeta,\mu)$. 

\begin{lemma} \label{bounded}
    The quotient $T_\el(G,\zeta,\mu)/\CC^1$ is finite. If $F$ is archimedean, then its cardinality is bounded independently of $(\zeta,\mu)$.
\end{lemma}

\begin{proof}
    If $F$ is non-archimedean, there are finitely many elements of $T_\temp(G)/\CC^1$ with a given infinitesimal character. Suppose that $F$ is archimedean. We may assume that $F=\RR$. Since $L^G(M_0)/W_0^G$ is finite, it suffices to show that for each $M\in\Lc^G(M_0)$, the number of possibilities for $\sigma\in\Pi_2(M)$ such that $(M,\sigma,\wt{r})\in T_\el(G,\zeta,\mu)$ is bounded independently of $(\zeta,\mu)$. There are at most $|W(G,T)/W(M,T)|$ possibilities for $\mu_\sigma$ since it maps to $\mu$. Consider the restriction $\zeta_\sigma|_{A_M(\RR)^\circ}$ of the central character of $\sigma$ to $A_M(\RR)^\circ$. Since $H_M:A_M(\RR)^\circ\to\ak_M$ is an isomorphism, we have $\zeta_\sigma|_{A_M(\RR)^\circ}=e^{\langle\lambda,H_M(\cdot)\rangle}$ for some $\lambda\in i\ak_M^*$. Every element of $W_\reg^G(\sigma)$ fixes $\zeta_\sigma|_{A_M(\RR)^\circ}$. Since $W_\reg^G(\sigma)$ is the subset of $w\in W_\reg^G(\sigma)$ such that $\ak_M^w=\ak_G$, we have $\lambda\in i\ak_G^*$. At the same time, $\zeta|_{A_G(\RR)^\circ}=\zeta_\sigma|_{A_G(\RR)^\circ}=e^{\langle\lambda,H_G(\cdot)\rangle}$. It follows that $\lambda$ and thus $\zeta_\sigma|_{A_M(\RR)^\circ}$ is determined by $\zeta$. There are $|A_M(\RR)/A_M(\RR)^\circ|$ ways that $\zeta_\sigma|_{A_M(\RR)}$ can extend $\zeta_\sigma|_{A_M(\RR)^\circ}$, so there are $|A_M(\RR)/A_M(\RR)^\circ|$ possibilities for $\zeta_\sigma|_{A_M(\RR)}$. As remarked above, in \cite{AdamsVogan} Adams and Vogan prove that there are finitely many discrete series representations with a fixed infinitesimal character and split radical character and these representations form a discrete series $L$-packet. (The group $A_M(\RR)$ is the split radical of $M(\RR)$.) Since the cardinality of discrete series $L$-packets of $M$ is bounded, the lemma follows.
\end{proof}

Fix a set of representatives $\Xs^G$ for the set of $i\ak_G^*$-orbits of unitary characters of $A_G(F)$. We denote the $i\ak_G^*$-orbit of a pair $(\zeta,\mu)$ by $[\zeta,\mu]$. Fix a set of representatives $\Ps^G$ for the set of orbits $[\zeta,\mu]$, such that for each $(\zeta,\mu)\in\Ps^G$ we have $\zeta\in\Xs^G$. 

Let $T_\el(G,[\zeta,\mu])$ be the subspace of all $\tau\in T_\el(G)$ such that $(\zeta_\tau|_{A_G(F)},\mu_\tau)$ belongs to $[\zeta,\mu]$. Let $\ak_{G,\mu}^\vee$ denote the isotropy subgroup of $\mu$ in $\ak_{G,\CC}^*$. We have 
\[
\ak_{G,F}^\vee\subseteq\ak_{G,\mu}^\vee\subseteq\wt{\ak}_{G,F}^\vee.
\]
Each connected component of $T_\el(G,[\zeta,\mu])/\CC^1$ meets $T_\el(G,\zeta,\mu)/\CC^1$, and two elements of $T_\el(G,\zeta,\mu)/\CC^1$ lie in the same connected component of $T_\el(G,[\zeta,\mu])/\CC^1$ if and only if they are in the same orbit under the finite group $\ak_{G,\mu}^\vee/\ak_{G,F}^\vee$. Let $\ul{E}_\el(G,\zeta,\mu)\subseteq T_\el(G,\zeta,\mu)$ be a set of representatives for the connected components of $T_\el(G,[\zeta,\mu])/\CC^1$, or equivalently a set of representatives for the $\ak_{G,\mu}^\vee$-orbits in $T_\el(G,\zeta,\mu)/\CC^1$. Then the set of translates $\ak_{G,\mu}^\vee\cdot\ul{E}_\el(G,\zeta,\mu)\subseteq T_\el(G,\zeta,\mu)$ contains an orthogonal basis $E_\el(G,\zeta,\mu)$ of $D_\el(G,\zeta,\mu)$. 

We define $\Phi_2(G,\zeta,\mu)$ to be the set of $\phi\in\Phi_2(G,\zeta)$ with $\mu_\phi=\mu$, and we define $\Phi_2(G,[\zeta,\mu])$ to be the set of all $\phi\in\Phi_2(G)$ with $(\zeta_\phi|_{A_G(F)},\mu_\phi)\in[\zeta,\mu]$. Then $D_\el^\st(G,\zeta,\mu)=\CC\Phi_2(G,\zeta,\mu)$. Let $D_\el^\unst(G,\zeta,\mu)$ be the orthogonal complement of $D_\el^\st(G,\zeta,\mu)$ in $D_\el(G,\zeta,\mu)$ with respect to the elliptic inner product. (Recall that $D_\el(G,\zeta,\mu)$ is finite-dimensional.) Let $\ul{B}_\el^\st(G,\zeta,\mu)\subseteq\Phi_2(G,\zeta,\mu)$ be a set of representatives for the $i\ak_G^*$-orbits in $\Phi_2(G,[\zeta,\mu])$, or equivalently a set of representatives for the $\ak_{G,\mu}^\vee$-orbits in $\Phi_2(G,\zeta,\mu)$. Then $\ak_{G,\mu}^\vee\cdot\ul{B}_\el^\st(G,\zeta,\mu)=\Phi_2(G,\zeta,\mu)$. We write $B_\el^\st(G,\zeta,\mu)=\Phi_2(G,\zeta,\mu)$. The unitary representation of the finite group $\ak_{G,\mu}^\vee/\ak_{G,F}^\vee$ on $D_\el(G)$ stabilises $D_\el^\unst(G,\zeta,\mu)$. Let $B_\el^\unst(G,\zeta,\mu)=\ul{B}_\el^\unst(G,\zeta,\mu)$ be an orthonormal eigenbasis of $D_\el^\unst(G,\zeta,\mu)$ for this representation. Define 
\[
\ul{B}_\el(G,\zeta,\mu)=\ul{B}_\el^\st(G,\zeta,\mu)\cup\ul{B}_\el^\unst(G,\zeta,\mu)
\]
and
\[
B_\el(G,\zeta,\mu)=B_\el^\st(G,\zeta,\mu)\cup B_\el^\unst(G,\zeta,\mu).
\]
Define $\ul{E}_\el(G)=\coprod_{(\zeta,\mu)\in\Ps^G}\ul{E}_\el(G,\zeta,\mu)$, and define $E_\el(G)$, $\ul{B}_\el^?(G)$, and $B_\el^?(G)$ similarly. Let $\ul{B}_\el(G)=\ul{B}_\el^\st(G)\cup\ul{B}_\el^\unst(G)$ and $B_\el(G)=B_\el^\st(G)\cup B_\el^\unst(G)$. 

In order to be able to treat the case of test functions and Schwartz functions at the same time, we write $\Fs=\Cc_{(c)}$, $\Ic\Fs=\Ic_{(c)}$, and $\Sc\Fs=\Sc_{(c)}$. We also write $\wh{\Fs}=\Ss$ (resp. $\wh{\Fs}=PW$) if $\Fs=\Cc$ (resp. $\Fs=\Cc_c$).

By definition, we have identifications
\[
\wh{\Fs}_\el(G)=\wh{\Fs}\Bigg(\coprod_{\tau\in\ul{E}_\el(G)}i\ak_G^*\cdot\tau\Bigg)
\quad,\quad
\wh{\Fs}_\el^\st(G)=\wh{\Fs}\Bigg(\coprod_{b\in\ul{B}_\el^\st(G)}i\ak_G^*\cdot b\Bigg).
\]
Let $\wh{\Fs}_\el(G,[\zeta,\mu])$ (resp. $\wh{\Fs}_\el^\st(G,[\zeta,\mu])$) be the subspace of $\wh{\Fs}_\el(G)$ (resp. $\wh{\Fs}_\el^\st(G)$) consisting of functions supported on $T_\el(G,[\zeta,\mu])$ (resp. $\Phi_2(G,[\zeta,\mu])$). The above identifications restrict to identifications 
\begin{align*}
    \wh{\Fs}_\el(G,[\zeta,\mu])&=\wh{\Fs}\Bigg(\coprod_{\tau\in\ul{E}_\el(G,\zeta,\mu)}i\ak_G^*\cdot\tau\Bigg), \\
    \wh{\Fs}_\el^\st(G,[\zeta,\mu])&=\wh{\Fs}\Bigg(\coprod_{b\in\ul{B}_\el^\st(G,\zeta,\mu)}i\ak_G^*\cdot b\Bigg).
\end{align*}
We define 
\begin{align*}
    \wh{\Fs}_\el^\unst(G)&=\wh{\Fs}\Bigg(\coprod_{b\in\ul{B}_\el^\unst(G)}i\ak_G^*\cdot b\Bigg), \\
    \wh{\Fs}_\el^\unst(G,[\zeta,\mu])&=\wh{\Fs}\Bigg(\coprod_{b\in\ul{B}_\el^\unst(G,\zeta,\mu)}i\ak_G^*\cdot b\Bigg).
\end{align*}
We have two orthogonal bases $E_\el(G)$ and $B_\el(G)$ of $\bigoplus_{(\zeta,\mu)\in\Ps^G}D_\el(G,\zeta,\mu)$. We define change-of-basis matrices
\[
b=\sum_{\tau\in E_\el(G)}c_{b,\tau}\tau \quad,\quad \tau=\sum_{b\in B_\el(G)}c_{\tau,b}b.
\]
It follows from the decompositions 
\[
E_\el(G)=\coprod_{(\zeta,\mu)\in\Ps^G}E_\el(G,\zeta,\mu)
\quad,\quad
B_\el(G)=\coprod_{(\zeta,\mu)\in\Ps^G}B_\el(G,\zeta,\mu)
\]
into orthogonal bases of $D_\el(G,\zeta,\mu)$ that these change-of-basis matrices are block-diagonal. Let $\varphi=(\varphi_{\tau})_{\tau\in\ul{E}_\el(G)}\in\wh{\Fs}_\el(G)$. We identify $\varphi$ with its linear extension to $D_\el(G)$. For $\tau\in E_\el(G)$ and $b\in B_\el(G)$, we write $\varphi_\tau(\lambda)=\varphi(\tau_\lambda)$ and $\varphi_b(\lambda)=\varphi(b_\lambda)$. Each element of $\tau\in E_\el(G)$ (resp. $b\in B_\el(G)$) is a translate of an element of $\ul{\tau}\in\ul{E}_\el(G)$ (resp. $\ul{b}\in\ul{B}_\el(G)$) by an element of $\lambda_0\in\wt{\ak}_{G,F}^\vee$, and thus $\varphi_\tau(\lambda)=\varphi_{\ul{\tau}}(\lambda+\lambda_0)$ (resp. $\varphi_b(\lambda)=\varphi_{\ul{b}}(\lambda+\lambda_0)$). For $b\in\ul{B}_\el(G)$ and $\lambda\in i\ak_G^*$, we have $b_\lambda=\sum_{\tau\in E_\el(G)}c_{b,\tau}\tau_\lambda$, and therefore
\[
\varphi_{b}(\lambda)=\sum_{\tau\in E_\el(G)}c_{b,\tau}\varphi_\tau(\lambda)=\sum_{\ul{\tau},\lambda_0}c_{b,\ul{\tau}_{\lambda_0}}\varphi_{\ul{\tau}}(\lambda+\lambda_0).
\]
Consider the linear map $\varphi\mapsto(\varphi_{b})_{b\in\ul{B}_\el(G)}=(\varphi^\st,\varphi^\unst)$, where $\varphi^?=(\varphi_{b})_{b\in\ul{B}_\el^?(G)}$. For each $(\zeta,\mu)\in\Ps^G$, this linear map restricts to an isomorphism of topological vector spaces 
\[
\wh{\Fs}_\el(G,\zeta,\mu)\longrightarrow\wh{\Fs}_\el^\st(G,\zeta,\mu)\oplus\wh{\Fs}_\el^\unst(G,\zeta,\mu).
\]
with inverse
\[
\varphi_\tau(\lambda)=\sum_{b\in B_\el(G)}c_{\tau,b}\varphi_b(\lambda)=\sum_{\ul{b},\lambda_0}c_{\tau,\ul{b}_{\lambda_0}}\varphi_{\ul{b}}(\lambda+\lambda_0).
\]
This follows easily since each of the spaces $\wh{\Fs}_\el(G,\zeta,\mu)$, $\wh{\Fs}_\el^\st(G,\zeta,\mu)$, and $\wh{\Fs}_\el^\unst(G,\zeta,\mu)$ is a finite direct sums of classical $\wh{\Fs}$-spaces. 

\begin{lemma}
    We have an isomorphism of topological vector spaces 
    \begin{align*}
    \wh{\Fs}_{\el}(G)&\longrightarrow\wh{\Fs}_{\el}^\st(G)\oplus\wh{\Fs}_{\el}^\unst(G) \\
    \varphi&\longmapsto(\varphi^\st,\varphi^\unst)
    \end{align*}
\end{lemma}

The argument is similar to that in \cite[Ch. IV, \S1.5 and \S2.2]{MWI}.

\begin{proof}
    First consider the case where $F$ is non-archimedean. In this case, we have locally convex direct sum decompositions 
    \[
    \wh{\Fs}_{\el}(G)=\bigoplus_{(\zeta,\mu)\in\Ps^G}\wh{\Fs}_{\el}(G,\zeta,\mu)
    \quad,\quad
    \wh{\Fs}_{\el}^?(G)=\bigoplus_{(\zeta,\mu)\in\Ps^G}\wh{\Fs}_{\el}^?(G,\zeta,\mu).
    \]
    Thus, it suffices to show that the columns of the change-of-basis matrices have finitely many non-zero entries. This is true since they are block-diagonal as remarked above and the dimension of $D_\el(G,\zeta,\mu)$ is finite for each $(\zeta,\mu)\in\Ps^G$.

    Now consider the archimedean case. As in \cite[Ch. IV, \S1.5 and \S2.2]{MWI}, it suffices to show that: (1) the number of nonzero entries in a column of the change-of-basis matrices is bounded; and (2) the entries of the change-of-basis matrices are bounded. Property (1) follows since the dimension of $D_\el(G,\zeta,\mu)$ is bounded independently of $(\zeta,\mu)$ by \Cref{bounded}. For property (2), recall that we have seen that $\{\|\tau\|_\el:\tau\in T_\el(G)\}$ is bounded. Moreover, by construction $\|b\|_\el=1$ for $b\in B_\el^\unst(G)$. Suppose that $b\in B_\el^\st(G)=\Phi_2(G)$ and write $b=\phi$. We have $\phi=\sum_{\pi\in\Pi_\phi}\pi$. Therefore $\|\phi\|_\el^2=\sum_{\pi\in\Pi_\phi}\|\pi\|_\el^2$. As mentioned above, for $\pi\in\Pi_2(G)$ we have $\|\pi\|_\el=1$. (See the discussion after the orthogonality relations \Cref{OrthogonalityRelations}.) Therefore $\|\pi\|_\el=|\Pi_\phi|$, which is bounded.
\end{proof}

We will identify $\wh{\Fs}_{\el}(G)=\wh{\Fs}_{\el}^\st(G)\oplus\wh{\Fs}_{\el}^\unst(G)$ via the above isomorphism. Note that the subspace $\wh{\Fs}_{\el}^\unst(G)$ of $\wh{\Fs}_{\el}(G)$ consists of those functions $\varphi$ whose linear extension to $D_\el(G)$ vanishes on $D_\el^\st(G)$. Thus, it does not depend on any of the choices made above.

Define $\Ic\Fs_\cusp^\st(G)$ to be the closed subspace of $\Ic\Fs_\cusp(G)$ consisting of normalised invariant orbital integrals that are constant on regular semisimple stable classes. Let $\Sc\Fs_\cusp(G)$ denote the image of $\Ic\Fs_\cusp(G)$ in $\Sc\Fs(G)$ and let $\Ic\Fs_\cusp^\unst(G)$ denote the kernel of the quotient map $\Ic\Fs_\cusp(G)\to\Sc\Fs_\cusp(G)$. Note that $\Ic\Fs_\cusp^\st(G)\cap\Ic\Fs_\cusp^\unst(G)=0$.

\begin{proposition}\label{prop:cuspidal}
    The invariant Fourier transform
    \[
    \begin{tikzcd}
        \Ic\Fs_\cusp(G) \arrow[r,"\sim"] & \wh{\Fs}_\el(G)
    \end{tikzcd}
    \]
    restricts to isomorphisms
    \[
    \begin{tikzcd}
        \Ic\Fs_\cusp^?(G) \arrow[r,"\sim"] & \wh{\Fs}_\el^{?}(G).
    \end{tikzcd}
    \]
    Consequently, we have
    \[
    \Ic\Fs_\cusp(G)=\Ic\Fs_\cusp^\st(G)\oplus\Ic\Fs_\cusp^\unst(G).
    \]
\end{proposition}

\begin{corollary}
    The stable Fourier transform gives an isomorphism
    \[
    \begin{tikzcd}
        \Sc\Fs_\cusp(G) \arrow[r,"\sim"] & \wh{\Fs}_\el^\st(G).
    \end{tikzcd}
    \]
\end{corollary}

We start with the following lemma which we will use to construct functions in $\Fs(G)$.

\begin{lemma}\label{lem:fcn}
    Let $f\in C_{c}^\infty(G,\zeta)$ and let $\phi\in C_c^\infty(\ak_{G,F})$ (resp. $\phi\in\Ss(\ak_{G,F})$). The function $f^\phi\defeq f\cdot(\phi\circ H_G)$ lies in $C_c^\infty(G)$ (resp. $\Cc(G)$).
\end{lemma}

Note that if $f$ is $K$-finite, then so is $f^\phi$.

\begin{proof}
    Let $f^\phi=f\cdot(\phi\circ H_G)$. Evidently, $f^\phi$ is smooth. Let $C\subseteq G(F)$ be a compact subspace such that $\supp(f)=CA_G(F)$.
    
    First, we suppose that $\phi\in C_c^\infty(\ak_{G,F})$. We have 
    \[
    \supp(f^\phi)\subseteq \supp(\phi\circ H_G)\cap\supp (f)\subseteq H_G^{-1}(\supp(\phi))\cap CA_G(F).
    \]
    Thus, to show that $\supp(f^\phi)$ is compact, it suffices to show that $H_G^{-1}(\supp\phi)\cap CA_G(F)$ is compact. Let $x\in H_G^{-1}(\supp(\phi))\cap CA_G(F)$ and write $x=ca$, where $c\in C$ and $a\in A_G(F)$. Then $H_G(a)=H_G(x)-H_G(c)$, which lies in the compact subspace $\supp(\phi)-H_G(C)$ of $\ak_{G,F}$. Let $\wt{C}=\wt{\ak}_{G,F}\cap(\supp(\phi)-H_G(C))$, a compact subspace of $\wt{\ak}_{G,F}$. Then $a\in H_G|_{A_G(F)}^{-1}(\wt{C})$, and $x\in CH_G|_{A_G(F)}^{-1}(\wt{C})$ Thus, 
    \[
    H_G^{-1}(\supp(\phi))\cap CA_G(F)\subseteq CH_G|_{A_G(F)}^{-1}(\wt{C}).
    \]
    Therefore it suffices to show that $H_G|_{A_G(F)}^{-1}(\wt{C})$ is compact. This follows since $H_G|_{A_G(F)}:A_G(F)\to\wt{\ak}_{G,F}$ is proper, as it is a continuous surjective homomorphism with compact kernel $A_G(F)^1$.
    
    For the rest of the proof, we suppose that $\phi\in\Ss(\ak_{G,F})$. First we show that $f^\phi$ is rapidly decreasing. Let $r>0$. We must show that $|f^\phi(g)|\ll\Xi(g)(1+\sigma(g))^{-r}$. Recall the decomposition $G(F)=KM_0(F)K$, that $\Xi$ is bi-$K$-invariant, and that $1+\sigma(k_1m_0k_2)\asymp1+\sigma(m_0)$ and $1+\sigma(m_0)\asymp 1+\|H_0(m_0)\|$ for $k_1,k_2\in K$ and $m_0\in M_0(F)$. Therefore, it suffices to show $|f^\phi(k_1m_0k_2)|\ll\Xi(m_0)(1+\|H_0(m_0)\|)^{-r}$.
    
    Since $f\in\Cc(G,\zeta)$, we have $|f(g)|\ll\Xi(g)(1+\sigma^{A_G}(g))^{-r}$. We have $1+\sigma^{A_G}(k_1m_0k_2)\asymp 1+\sigma^{A_G}(m_0)$, and therefore $|f(k_1m_0k_2)|\ll\Xi(m_0)(1+\sigma^{A_G}(m_0))^{-r}$. Since $\phi\in\Ss(\ak_{G,F})$, we have $|\phi(X)|\ll(1+\|X\|)^{-r}$. Therefore 
    \[
    |f^\phi(k_1m_0k_2)|\ll\Xi(m_0)(1+\|H_G(m_0)\|)^{-r}(1+\sigma^{A_G}(m_0))^{-r}.
    \]
    Decompose $H_0(m_0)\in\ak_{M_0,F}$ as $H_0(m_0)=H_G(m_0)+H_0(m_0)^G$ according to the decomposition $\ak_{M_0,F}=\ak_{G,F}\oplus\ak_{M_0,F}^G$. Since $1+\sigma(m_0)\asymp 1+\|H_0(m_0)\|$, we have
    \[
    1+\sigma^{A_G}(m_0)\asymp 1+\inf_{a\in A_G(F)}\|H_0(a)+H_0(m_0)\|.
    \]
    Since $\ak_{M_0,F}/\wt{\ak}_{G,F}=\ak_{G,F}/\wt{\ak}_{G,F}\oplus\ak_{M_0,F}^G$ and $\ak_{G,F}/\wt{\ak}_{G,F}$ is finite, we obtain 
    \[
    1+\inf_{a\in A_G(F)}\|H_0(a)+H_0(m_0)\|\asymp1+\|H_0(m_0)^G\|.
    \]
    Therefore 
    \[
    1+\sigma^{A_G(F)}(m_0)\asymp 1+\|H_0(m_0)^G\|.
    \]
    Thus, we have
    \begin{align*}
    |f^\phi(k_1m_0k_2)|/\Xi(m_0)&\ll(1+\|H_G(m_0)\|)^{-r}(1+\|H_0(m_0)^G\|)^{-r} \\
    &\ll (1+\|H_0(m_0)\|)^{-r}
    \end{align*}
    as required.
    
    Suppose that $F$ is non-archimedean. In this case, we must show that there exists a compact open subgroup $K_0$ of $G(F)$ such that $f^\phi$ is bi-$K_0$-invariant. Since $F$ is non-archimedean, $\supp(f)$ is open and has compact open image in $G(F)/A_G(F)$. Thus, we may take $C$ to be compact open. Then there exists a compact open subgroup $K_0$ of $G(F)$ such that $K_0CK_0=C$ and $f$ is constant on the $K_0$ double cosets in $C$. Since $f$ is zero outside of $CA_G(F)$, it suffices to show that $f$ is bi-$K_0$-invariant on $CA_G(F)$. Write $C=\bigcup_{i=1}^MK_0g_iK_0$ and define $c_i=f(K_0g_iK_0)$ for $i=1,\dots,M$. Then $f(k_0g_ik_0'a)=\zeta(a)^{-1}c_i$ for all $k_0,k_0'\in K_0$, $i=0,\dots,M$, and $a\in A_G(F)$. For any $k_0'',k_0'''$ we have
    \[
    f(k_0'k_0g_ik_0'ak_0'')=f(k_0'k_0g_ik_0'k_0'''a)=\zeta(a)^{-1}c_i=f(k_0g_ik_0'a).
    \]
    Therefore $f$ is bi-$K_0$-invariant. Since $H_G$ is zero on all compact subgroups of $G(F)$, we have that $f^\phi$ is bi-$K_0$-invariant.
    
    Now assume that $F$ is archimedean. We must show that for all $X,Y\in\Uk(\gk_\CC)$, the function $L(X)R(Y)f^\phi$ is rapidly decreasing. For $X\in\gk$, we have
    \[
    (R(X)f^\phi)(g)=[D\phi (H_G(g))\cdot (R(X)H_G)(g)\big]f(g)+(R(X)f)^\phi(g),
    \]
    and therefore $R(X)f^\phi$ is a linear combination of functions of the form $f^\phi$. A similar formula hods for $L(X)$. It follows that for all $X,Y\in\Uk(\gk_\CC)$, the function $L(X)R(Y)f^\phi$ is a linear combination of functions of the form $f^\phi$, and is thus rapidly decreasing by what we have already shown.
\end{proof}

We now form the union
\[
\ul{B}_\el^?(G,\zeta)=\bigcup_{(\zeta,\mu)\in\Ps^G} \ul{B}_\el^?(G,\zeta,\mu),
\]
and define $\ul{B}_\el(G,\zeta)=\ul{B}_\el^\st(G,\zeta)\cup \ul{B}_\el^\unst(G,\zeta)$. Let $b\in \ul{B}_\el(G,\zeta)$. Since the elements of $\ul{B}_\el(G)$ are mutually orthogonal, the pseudocoefficient $f[b]_G$ satisfies
\[
\int_{G(F)/A_G(F)}\Theta_{b'}(x)f[b](x)\dd{x}=\|b\|_\el^2\delta_{b}(b'),
\]
for all $b'\in \ul{B}_\el(G,\zeta)$. The domain of integration $G(F)/A_G(F)$ fibres over the finite set $\ak_{G,F}/\wt{\ak}_{G,F}$ with open fibres. We will need the following formula for the integral of $\Theta_{b'}(x)f[b](x)$ on a single fibre in the case when $b\in \ul{B}_\el^\st(G,\zeta)$.

\begin{proposition}\label{prop:pseudoIntegral}
    Let $b\in\ul{B}_\el^\st(G,\zeta)$. For $b'\in \ul{B}_\el(G,\zeta)$ and $X\in\ak_{G,F}/\wt{\ak}_{G,F}$, we have
    \[
    \int_{H_G^{-1}(X+\wt{\ak}_{G,F})/A_G(F)}\Theta_{b'}(x)f[b](x)\dd{x}=\|b\|_\el^2|\ak_{G,b}/\wt{\ak}_{G,F}|^{-1}\delta_b(b')\1_{\ak_{G,b}/\wt{\ak}_{G,F}}(X).
    \]
\end{proposition}

\begin{proof}
    Let $b'\in B_\el(G,\zeta)$. For $X\in\ak_{G,F}/\wt{\ak}_{G,F}$, define
    \[
    f_{b,b'}(X)=\int_{H_G^{-1}(X+\wt{\ak}_{G,F})/A_G(F)}\Theta_{b'}(x)f[b](x)\dd{x}.
    \]
    We will calculate the Fourier transform of $f_{b,b'}$. The Pontryagin dual of $\ak_{G,F}/\wt{\ak}_{G,F}$ is $\wt{\ak}_{G,F}^\vee/\ak_{G,F}^\vee$. For $\lambda\in \wt{\ak}_{G,F}^\vee/\ak_{G,F}^\vee$ we have
    \begin{align*}
    \wh{f_{b,b'}}(\lambda)&=\sum_{X\in\ak_{G,F}/\wt{\ak}_{G,F}}e^{-\langle\lambda,X\rangle}f_{b,b'}(X) \\
    &=\sum_{X\in\ak_{G,F}/\wt{\ak}_{G,F}}e^{-\langle\lambda,X\rangle}\int_{H_G^{-1}(X+\wt{\ak}_{G,F})/A_G(F)}\Theta_{b'}(x)f[b](x)\dd{x} \\
    &=\sum_{X\in\ak_{G,F}/\wt{\ak}_{G,F}} \int_{H_G^{-1}(X+\wt{\ak}_{G,F})/A_G(F)} e^{-\langle\lambda,H_G(x)\rangle}\Theta_{b'}(x)f[b](x)\dd{x} \\
    &=\sum_{X\in\ak_{G,F}/\wt{\ak}_{G,F}} \int_{H_G^{-1}(X+\wt{\ak}_{G,F})/A_G(F)} \Theta_{b'}(x)f[b_\lambda](x)\dd{x} \\
    &=\int_{G(F)/A_G(F)}\Theta_{b_{-\lambda}'}(x)f[b](x)\dd{x} \\
    &=\langle b, b_{-\lambda}' \rangle_\el
    \end{align*}
    
    If $b'\in \ul{B}_\el^\unst(G,\zeta)$, then $b_{-\lambda}'\in D_\el^\unst(G,\zeta)$ and thus $\langle b, b_{-\lambda}' \rangle_\el=0$. If $b'\in \ul{B}_\el^\st(G,\zeta)\subseteq \Phi_2(G,\zeta)$, then $b_{-\lambda}'\in \Phi_2(G,\zeta)$ and $\langle b, b_{-\lambda}' \rangle_\el=\|b\|_\el^2\delta_{b}(b_{-\lambda}')$. Thus, we have $\wh{f_{b,b'}}(\lambda)=\|b\|_\el^2\delta_b(b_{-\lambda}')$. 
    
    Suppose $b=b_{-\lambda}'$. Then $\mu_b=(-\lambda)\cdot\mu_{b'}$. It follows that $[\zeta,\mu_b]=[\zeta,\mu_{b'}]$. Since $(\zeta,\mu_b),(\zeta,\mu_{b'})\in\Ps^G$, it follows that $\mu_b=\mu_{b'}$. Consequently, $\lambda$ lies in the stabiliser $\ak_{G,b}^\vee$ of $b$ in $\wt{\ak}_{G,F}^\vee$. Moreover, since $b,b'\in\ul{B}_\el^\st(G,\zeta,\mu_{b})$ and $b=b_{-\lambda}'$, we have $b=b'$. It follows that $\delta_b(b_{-\lambda}')=\delta_b(b')\1_{\ak_{G,b}^\vee/\ak_{G,F}^\vee}(\lambda)$.
    
    Thus, we have shown that
    \[
    \wh{f_{b,b'}}=\|b\|_\el^2\delta_{b}(b')\1_{\ak_{G,b}^\vee/\ak_{G,F}^\vee}.
    \]
    When $b'\neq b$, we obtain $f_{b,b'}=0$ as claimed. It remains to be shown that
    \[
    f_{b,b}=\|b\|_\el^2|\ak_{G,b}/\wt{\ak}_{G,F}|^{-1}\1_{\ak_{G,b}/\wt{\ak}_{G,F}}.
    \]
    We have 
    \[
    \wh{f_{b,b}}=\|b\|_\el^2\1_{\ak_{G,b}^\vee/\ak_{G,F}^\vee}.
    \]
    Thus, we must show that the Fourier transform of $\1_{\ak_{G,b}/\wt{\ak}_{G,F}}$ on $\ak_{G,F}/\wt{\ak}_{G,F}$ is $|\ak_{G,b}/\wt{\ak}_{G,F}|\cdot \1_{\ak_{G,b}^\vee/\ak_{G,F}^\vee}$. This follows from the general formula for the Fourier transform of a characteristic function of a subgroup of a finite abelian group. Let $A$ be a finite abelian group and let $B$ be a subgroup of $A$. Then for $\chi\in\wh{B}$, we have
    \begin{align*}
        \wh{\1_B}(\chi)&=\sum_{a\in A}\chi(a)^{-1}\1_B(a) \\
        &=\sum_{b\in B}\chi(b)^{-1} \\
        &=|B|\cdot\1_{A^\perp}(\chi),
    \end{align*}
    where $A^\perp = \{\chi\in\wh{A} : \chi|_B=1\}$. If we take $A=\ak_{G,F}/\wt{\ak}_{G,F}$ and $B=\ak_{G,b}/\wt{\ak}_{G,F}$, then $A^\perp=\ak_{G,b}^\vee/\ak_{G,F}^\vee$, and we obtain the desired result.
\end{proof}

For $b\in \ul{B}_\el(G)$, define $\wh{\Fs}_\el(G)_b$ to be the subspace of all $\varphi=(\varphi_{b'})_{b'\in \ul{B}_\el(G)}$ in $\wh{\Fs}_\el(G)$ with $\varphi_{b'}=0$ for $b'\neq b$.

\begin{proposition} \label{prop:cuspidalInverseFourier}
    Let $b\in \ul{B}_\el^\st(G)$ and let $\varphi=(\varphi_{b'})_{b'\in \ul{B}_\el(G)}\in\wh{\Fs}_\el(G)_b$. Regarding $\varphi_b\in\wh{\Fs}(i\ak_G^*/\ak_{G,b}^\vee)$ as an element of $\wh{\Fs}(i\ak_G^*/\ak_{G,F}^\vee)$, let $\phi_b\in\Fs(\ak_{G,F})$ be its Fourier transform. The inverse invariant Fourier transform of $\varphi$ is
    \[
    \vol(A_G(F)^1)^{-1}\|b\|_\el^{-2}\frac{|\ak_{G,b}/\wt{\ak}_{G,F}|}{|\ak_{G,b}^\vee/\ak_{G,F}^\vee|}\phi_b(H_G(\gamma))f[b]_G(\gamma)
    \]
    and this lies in $\Ic\Fs_\cusp^\st(G)$.
\end{proposition}

\begin{proof}
    Let $\zeta=\zeta_b$ be the $A_G(F)$-character of $b$. Choose a function $f[b]\in C_{c,\cusp}^\infty(G,\zeta,K)$ representing the pseudocoefficient $f[b]_G\in\Ic C_{c,\cusp}^{\infty}(G,\zeta,K)$ of $b$. Define $f:G(F)\to\CC$ by 
    \[
    f(x)=\phi_b(H_G(x))f[b](x).
    \]
    By \Cref{lem:fcn}, we have that $f\in\Fs(G)$. We have $f_G(\gamma)=\phi_b(H_G(\gamma))f[b]_G(\gamma)$, which does not depend on the choice of $f[b]$ and lies in $\Ic\Fs_{\cusp}(G)$. By \Cref{lem:pseudo}, we have $f[b]_G\in\Ic C_{c,\cusp}^{\infty,\st}(G,\zeta,K)$. That is, $f[b]_G(\gamma)$ is constant on regular semisimple stable classes. Since $H_G$ is constant on elements of $G(F)$ that lie in the same $G$-conjugacy class, we have that $f_G$ is constant on regular semisimple stable classes, i.e. $f_G\in\Ic\Fs_{\cusp}^\st(G)$.
    
    It remains for us to show that for all $b'\in \ul{B}_\el(G)$ and $\lambda\in i\ak_G^*$ we have
    \[
    f_G(b_\lambda')=\vol(A_G(F)^1)\|b\|_\el^2\frac{|\ak_{G,b}^\vee/\ak_{G,F}^\vee|}{|\ak_{G,b}/\wt{\ak}_{G,F}|}\varphi_b(\lambda)\delta_b(b').
    \]
    Let $\zeta'=\zeta_{b'}$ be the $A_G(F)$-character of $b'$. We have
    \begin{align*}
        f_G(b_\lambda')=&\int_{G(F)}e^{\langle\lambda,H_G(x)\rangle}\Theta_{b'}(x)\phi_b(H_G(x))f[b](x)\dd{x} \\
        =&\int_{G(F)/A_G(F)}\int_{A_G(F)}e^{\langle\lambda,H_G(xa)\rangle}\Theta_{b'}(xa)\phi_b(H_G(xa))f[b](xa)\dd{a}\dd{x} \\
        =&\int_{G(F)/A_G(F)}e^{\langle\lambda,H_G(x)\rangle}\Theta_{b'}(x)f[b](x) \\
        &\int_{A_G(F)}e^{\langle\lambda,H_G(a)\rangle}\zeta'(a)\zeta(a)^{-1}\phi_b(H_G(xa))\dd{a}\dd{x}.
    \end{align*}
    The inner integral in the third equality is equal to
    \begin{align*}
        &\int_{A_G(F)/A_G(F)^1}\int_{A_G(F)^1}e^{\langle\lambda,H_G(aa^1)\rangle}\zeta'(aa^1)\zeta(aa^1)^{-1}\phi_b(H_G(xaa^1))\dd{a^1}\dd{a} \\
        &=\int_{A_G(F)/A_G(F)^1}e^{\langle\lambda,H_G(a)\rangle}\zeta'(a)\zeta(a)^{-1}\phi_b(H_G(xa))\int_{A_G(F)^1}\zeta'(a^1)\zeta(a^1)^{-1}\dd{a^1}\dd{a}.
    \end{align*}
    Now, 
    \[
    \int_{A_G(F)^1}\zeta'(a^1)\zeta(a^1)^{-1}\dd{a^1}=\vol(A_G(F)^1)\delta_{\zeta|_{A_G(F)^1}}(\zeta'|_{A_G(F)^1}),
    \]
    which is equal to 1 if $\zeta,\zeta'$ lie in the same $i\ak_G^*$-orbit, and 0 otherwise. Since $\zeta,\zeta'$ lie in our set $\Xs^G$ of representatives of $i\ak_G^*$-orbits, we obtain 
    \[
    \int_{A_G(F)^1}\zeta'(a^1)\zeta(a^1)^{-1}\dd{a^1}=\vol(A_G(F)^1)\delta_{\zeta}(\zeta')
    \]
    Therefore we have $f_G(b'_\lambda)=0$ if $\zeta'\neq\zeta$. Thus, we may assume that $\zeta'=\zeta$. Then
    \begin{align*}
        &\vol(A_G(F)^1)^{-1}f_G(b_\lambda') \\
        =&\int_{G(F)/A_G(F)}\Theta_{b'}(x)f[b](x) \\
        &\int_{A_G(F)/A_G(F)^1}e^{\langle\lambda,H_G(x)+H_G(a)\rangle}\phi_b(H_G(x)+H_G(a))\dd{a}\dd{x} \\
        =&\int_{G(F)/A_G(F)}\Theta_{b'}(x)f[b](x)\int_{\wt{\ak}_{G,F}}e^{\langle\lambda,H_G(x)+\wt{X}\rangle}\phi_b(H_G(x)+\wt{X})\dd{\wt{X}}\dd{x} \\
        =&\sum_{X\in\ak_{G,F}/\wt{\ak}_{G,F}}\int_{H_G^{-1}(X+\wt{\ak}_{G,F})}\Theta_{b'}(x)f[b](x)\int_{\wt{\ak}_{G,F}}e^{\langle\lambda,H_G(x)+\wt{X}\rangle}\phi_b(H_G(x)+\wt{X})\dd{\wt{X}}\dd{x} \\
        =&\sum_{X\in\ak_{G,F}/\wt{\ak}_{G,F}}\int_{H_G^{-1}(X+\wt{\ak}_{G,F})}\Theta_{b'}(x)f[b](x)\int_{\wt{\ak}_{G,F}}e^{\langle\lambda,X+\wt{X}\rangle}\phi_b(X+\wt{X})\dd{\wt{X}}\dd{x}
    \end{align*}
    Consequently, by \Cref{prop:pseudoIntegral}, we have
    \begin{align*}
    f_G(b_\lambda')=&\vol(A_G(F)^1)\|b\|_\el^2|\ak_{G,b}/\wt{\ak}_{G,F}|^{-1}\delta_b(b') \\
    &\sum_{X\in\ak_{G,b}/\wt{\ak}_{G,F}}\int_{\wt{\ak}_{G,F}}e^{\langle\lambda,X+\wt{X}\rangle}\phi_b(X+\wt{X})\dd{\wt{X}} \\
    &=\vol(A_G(F)^1)\|b\|_\el^2|\ak_{G,b}/\wt{\ak}_{G,F}|^{-1}\delta_b(b')\int_{\ak_{G,b}}e^{\langle\lambda,X\rangle}\phi_b(X)\dd{X}.
    \end{align*}
    For $X\in\ak_{G,F}$, we have
    \begin{align*}
        \phi_b(X)&=\int_{i\ak_G^*/\ak_{G,F}^\vee}e^{-\langle\lambda,X\rangle}\varphi_b(\lambda)\dd{\lambda} \\
        &=\int_{i\ak_G^*/\ak_{G,b}^\vee}\sum_{\wt{\lambda}\in\ak_{G,b}^\vee/\ak_{G,F}^\vee}e^{-\langle\lambda+\wt{\lambda},X\rangle}\varphi_b(\lambda+\wt{\lambda})\dd{\lambda} \\
        &=\int_{i\ak_G^*/\ak_{G,b}^\vee}e^{-\langle\lambda,X\rangle}\varphi_b(\lambda)\sum_{\wt{\lambda}\in\ak_{G,b}^\vee/\ak_{G,F}^\vee}e^{-\langle\wt{\lambda},X\rangle}\dd{\lambda} \\
        &=\int_{i\ak_G^*/\ak_{G,b}^\vee}e^{-\langle\lambda,X\rangle}\varphi_b(\lambda)|\ak_{G,b}^\vee/\ak_{G,F}^\vee|\1_{\ak_{G,b}}(X)\dd{\lambda} \\
        &=|\ak_{G,b}^\vee/\ak_{G,F}^\vee|\1_{\ak_{G,b}}(X)\wh{\varphi_b}(X).
    \end{align*}
    Therefore
    \[
    \int_{\ak_{G,b}}e^{\langle\lambda,X\rangle}\phi_b(X)\dd{X}=|\ak_{G,b}^\vee/\ak_{G,F}^\vee|\int_{\ak_{G,b}}e^{\langle\lambda,X\rangle}\wh{\varphi_b}(X)\dd{X}=|\ak_{G,b}^\vee/\ak_{G,F}^\vee|\varphi_b(\lambda).
    \]
    Thus, we have
    \[
    f_G(b_\lambda')=\vol(A_G(F)^1)\|b\|_\el^2\frac{|\ak_{G,b}^\vee/\ak_{G,F}^\vee|}{|\ak_{G,b}/\wt{\ak}_{G,F}|}\varphi_b(\lambda)\delta_{b}(b')
    \]
    as required.
\end{proof}

Recall that we denote the invariant Fourier transform by $\Fc$. Stable spectral density for $\Fs_\cusp(G)$ is equivalent to the assertion that $\Fc^{-1}(\wh{\Fs}_\el^\unst(G))\subseteq\Ic\Fs_\cusp^\unst(G)$.

\begin{lemma}
    We have $\Fc^{-1}(\Ss_\el^\unst(G))\subseteq\Ic_\cusp^\unst(G)$. That is, stable spectral density holds for $\Cc_\cusp(G)$.
\end{lemma}

\begin{proof}
    By stable spectral density for $\Cc_c^\infty(G)$, we have $\Fc^{-1}(PW_\el^\unst(G))\subseteq\Ic_{c,\cusp}^\unst(G)$. Since $PW_\el^\unst(G)$ is dense in $\Ss_\el^\unst(G)$ and $\Ic_\cusp^\unst(G)$ is closed in $\Ic_\cusp(G)$, it follows that we have $\Fc^{-1}(\Ss_\el^\unst(G))\subseteq\Ic_\cusp^\unst(G)$.
\end{proof}

We now prove \Cref{prop:cuspidal}.

\begin{proof}
    Let $\varphi\in\wh{\Fs}_\el^\st(G)$. Write $\ul{B}_\el^\st(G)$ as a countable increasing union of finite sets $\ul{B}_\el^\st(G)=\bigcup_{i=1}^\infty \ul{B}_{\el,i}^\st(G)$. For each $i$, define $\varphi_i=(\varphi_{i,b})_{b\in \ul{B}_{\el,i}(G)}\in\wh{\Fs}_\el^\st(G)$ by $\varphi_{i,b}=\varphi_b$ if $b\in \ul{B}_{\el,i}(G)$ and $\varphi_{i,b}=0$ otherwise. Then $\lim_{i}\varphi_i=\varphi$. By \Cref{prop:cuspidalInverseFourier}, we have $\Fc^{-1}(\varphi_i)\in\Ic\Fs_\cusp^\st(G)$. Since $\Ic\Fs_\cusp^\st(G)$ is closed in $\Ic\Fs_\cusp(G)$, it follows that $\Fc(\varphi)\in\Ic\Fs_\cusp^\st(G)$. Therefore we have $\Fc^{-1}(\wh{\Fs}_\el^\st(G))\subseteq\Ic\Fs_\cusp^\st(G)$. 
    
    Since $\wh{\Fs}_\el(G)=\wh{\Fs}_\el^\st(G)\oplus \wh{\Fs}_\el^\unst(G)$, we have 
    \begin{align*}
    \Ic\Fs_\cusp(G)&=\Fc^{-1}(\wh{\Fs}_\el^\st(G))\oplus \Fc^{-1}(\wh{\Fs}_\el^\unst(G)) \\
    &\subseteq\Ic\Fs_\cusp^\st(G)\oplus\Ic\Fs_\cusp^\unst(G) \\
    &\subseteq\Ic\Fs_\cusp(G).
    \end{align*}
    Therefore $\Fc^{-1}(\wh{\Fs}_\el^?(G))=\Ic\Fs_\cusp^?(G)$ as required.
\end{proof}

\subsubsection{$K$-finite functions.}

If we trace through the above, replacing $\Ic\Fs$ with $\Ic_f$, $\Sc\Fs$ with $\Sc_f$, and $\wh{\Fs}$ with $PW_f$, we obtain a proof of the following.

\begin{proposition}
    The invariant Fourier transform
    \[
    \begin{tikzcd}
        \Ic_{f,\cusp}(G) \arrow[r,"\sim"] & PW_{\el,f}(G)
    \end{tikzcd}
    \]
    restricts to isomorphisms
    \[
    \begin{tikzcd}
        \Ic_{f,\cusp}^?(G) \arrow[r,"\sim"] & PW_{\el,f}^{?}(G).
    \end{tikzcd}
    \]
    Consequently, we have
    \[
    \Ic_{f,\cusp}(G)=\Ic_{f,\cusp}^\st(G)\oplus\Ic_{f,\cusp}^\unst(G).
    \]
\end{proposition}

\begin{corollary}
    The stable Fourier transform gives an isomorphism
    \[
    \begin{tikzcd}
        \Sc_{f,\cusp}(G) \arrow[r,"\sim"] & PW_{\el,f}^\st(G).
    \end{tikzcd}
    \]
\end{corollary}

\subsubsection{Proof of the main stable Paley--Wiener theorems.}
We now apply the constructions and results of \Cref{sec:cuspidalcase} to each semistandard Levi $L\in\Lc^G(M_0)$. We define
\[
\wh{\Fs}^\unst(G)=\Bigg(\bigoplus_{L\in\Lc^G(M_0)}\wh{\Fs}_\el^\unst(L)\Bigg)^{W_0^G}=\bigoplus_{L\in\Lc^G(M_0)/W_0^G}\wh{\Fs}_\el^\unst(L)^{W^G(L)}.
\]
We have decompositions
\[
\wh{\Fs}(G)=\Bigg(\bigoplus_{L\in\Lc^G(M_0)}\wh{\Fs}_\el(G)\Bigg)^{W_0^G}=\bigoplus_{L\in\Lc^G(M_0)/W_0^G}\wh{\Fs}_\el(L)^{W^G(L)}
\]
and $\wh{\Fs}_\el(L)=\wh{\Fs}_\el^\st(L)\oplus\wh{\Fs}_\el^\unst(L)$ for each $L\in\Lc^G(M_0)$. Thus, we obtain a decomposition $\wh{\Fs}(G)=\wh{\Fs}^\st(G)\oplus\wh{\Fs}^\unst(G)$. By \Cref{PWDenseSchwartz} the inclusion $PW^\unst(G)\to\Ss^\unst(G)$ is continuous with dense image.

We can now deduce the full stable spectral density theorem for $\Cc(G)$ from the stable spectral density theorem for $C_c^\infty(G)$. 

\begin{theorem}[Stable spectral density] \label{StableSpectralDensity}
    Let $f\in\Cc(G)$. If $f^G(\phi)=0$ for all $\phi\in\Phi_\temp(G)$, then $f^G(\delta)=0$ for all $\delta\in\Delta_\rs(G)$. Equivalently, $\Fc^{-1}(\Ss^\unst(G))\subseteq\Ic\Cc^\unst(G)$.
\end{theorem}

\begin{proof}
    By stable spectral density for $C_c^\infty(G)$, we have $\Fc^{-1}(PW^\unst(G))\subseteq\Ic C_c^{\infty,\unst}(G)$. Since $PW^\unst(G)$ is dense in $\Ss^\unst(G)$ and $\Ic\Cc^\unst(G)$ is closed in $\Ic\Cc(G)$, it follows that we have $\Fc^{-1}(\Ss^\unst(G))\subseteq\Ic\Cc^\unst(G)$.
\end{proof}

Note that $\Fc(\Ic\Fs^\unst(G))\subseteq\wh{\Fs}^\unst(G)$ follows immediately from the definitions. Thus, we have $\Fc^{-1}(\wh{\Fs}^\unst(G))=\Ic\Fs^\unst(G)$. 

We now prove the main stable Paley--Wiener theorems.

\begin{proof}
    The stable Fourier transform $\Fc^\st:\Ic\Fs(G)\to\wh{\Fs}^\st(G)$ is the composition of the invariant Fourier transform $\Fc$ and the projection onto $\wh{\Fs}^\st(G)$. The kernel is $\Ic\Fs^\unst(G)$, so it descends to a continuous bijection $\Fc^\st:\Sc\Fs(G)\to\wh{\Fs}^\st(G)$. Since its inverse is the composition of continuous maps
    \[
    \begin{tikzcd}
        \wh{\Fs}^\st(G) \arrow[r,hook] & \wh{\Fs}(G) \arrow[r,"\Fc^{-1}"] & \Ic\Fs(G) \arrow[r,two heads] & \Sc\Fs(G)
    \end{tikzcd}
    \]
    the stable Fourier transform $\Fc^\st:\Sc\Fs(G)\to\wh{\Fs}^\st(G)$ is a topological isomorphism. 
    
    In the same way, one can prove that the stable Fourier transform restricts to an isomorphism of topological vector spaces $\Fc^\st:\Sc_f(G)\to PW_f^\st(G)$.
\end{proof}

Note that the stable Paley--Wiener theorems for $\Sc_{f,\cusp}(G,\zeta)$ and $\Sc\Fs_{\cusp}(G)$ can be proved in a similar way. The proof given above gives more information, namely that the subspace $\Ic_{f,\cusp}^\st(G,\zeta)$ (resp. $\Ic\Fs_{\cusp}^\st(G)$) corresponds to $\wh{\Fs}_{\el,f}^\st(G,\zeta)$ (resp. $\wh{\Fs}_\el^\st(G)$) under the invariant Fourier transform.

Using the stable Paley--Wiener Theorem for $\Sc(G)$, a Fourier inversion formula for elements of $\Sc(G)$ can be obtained from Arthur's Fourier inversion formula for elements of $\Ic(G)$. The argument is the same as the one used by Arthur to prove \cite[6.3]{ArthurRelations}.

\section{Stable Transfer} \label{sec:transfer}
In this section we prove \Cref{mainthm}. We also describe some examples of stable transfer operators explicitly. We begin by recalling the setup from the introduction. Let $H$ and $G$ be connected reductive groups over our local field $F$, and let $\xi:\Lgp{H}\to\Lgp{G}$ be an $L$-homomorphism (taken up to equivalence). Assume that $G$ is quasisplit so that $G$-relevance is automatic. Then we have a map
\[
\xi_*:\Phi(H)\longrightarrow\Phi(G)
\]
defined by $\xi_*(\phi)=\xi\circ\phi$. This restricts to a map
\[
\xi_*:\Phi_\temp(H)\longrightarrow\Phi_\temp(G)
\]
if and only if $\xi$ is tempered in the sense that $\xi|_{W_F}$ is a tempered $L$-parameter of $G$. Suppose that $\xi$ is injective and tempered. In \cite[Questions A \& B]{LanST}, Langlands asked whether for each $f^G\in\Sc_c(G)$ there exists $f^H\in\Sc_c(H)$ such that
\[
f^H(\phi)=f^G(\xi_*(\phi))
\]
for all $\phi\in\Phi_\temp(H)$. Note that if $f^H$ exists, it is uniquely determined by stable spectral density for test functions. More generally, one can ask whether for each $f^G\in\Sc(G)$ there exists $f^H\in\Sc(H)$ with the above property, which would also be uniquely determined by spectral density. The function $f^H$ is called the stable transfer of $f^G$ along $\xi$, and we denote it by $\Tc_\xi f^G$. It has also been called the functorial transfer of $f^G$ and, to avoid confusion with endoscopic transfer, the stable-stable transfer of $f^G$. We will prove the following.

\begin{theorem}\label{thm:pullback}
    Let $\xi:\Lgp{H}\to\Lgp{G}$ be an equivalence class of injective tempered $L$-homomorphisms, and assume that $G$ is quasisplit. Then pullback along $\xi_*$ gives a well-defined continuous linear map 
    \[
    \xi^*:\Ss^\st(G)\longrightarrow \Ss^\st(H).
    \]
    Furthermore, it restricts to continuous linear maps
    \[
    \xi^*:PW^\st(G)\longrightarrow PW^\st(H)
    \]
    and
    \[
    \xi^*:PW_f^\st(G)\longrightarrow PW_f^\st(H).
    \]
\end{theorem}

It follows that stable transfer along $\xi$ can be constructed by taking the stable Fourier transform, pulling back along $\xi_*$, and then taking the inverse stable Fourier transform:
\[
\begin{tikzcd}
    \Sc(G) \arrow[r,"\Tc_\xi"]\arrow[d,"\Fc^\st"] & \Sc(H) \arrow[d,"\Fc^\st"] \\
    \Ss(G) \arrow[r, "\xi^*"] & \Ss(H)
\end{tikzcd}
\]
We obtain \Cref{mainthm} as a corollary, which we restate here for convenience. In particular, we obtain an affirmative answer to \cite[Questions A \& B]{LanST}.

\begin{corollary} \label{existence}
    Let $\xi:\Lgp{H}\to\Lgp{G}$ be an equivalence class of injective tempered $L$-homomorphisms, and assume that $G$ is quasisplit. If $F$ is non-archimedean, we assume \Cref{hypothesis} for $H$ and $G$. There exists a continuous linear operator
    \[
    \Tc_\xi:\Sc(G)\longrightarrow\Sc(H)
    \]
    whose transpose
    \[
    \Tc_\xi':\Sc(H)'\longrightarrow\Sc(G)'
    \]
    satisfies
    \[
    \Tc_\xi'\Theta_\phi=\Theta_{\xi_*\phi}
    \]
    for all $\phi\in\Phi_\temp(H)$. Moreover, $\Tc_\xi$ restricts to continuous linear operators $\Sc_c(G)\to\Sc_c(H)$ and $\Sc_f(G)\to\Sc_f(H)$.
\end{corollary}

\subsection{Reduction to $\pi_0$-finiteness of functorial transfer}
Let $\xi:\Lgp{H}\to\Lgp{G}$ be an equivalence class of injective tempered $L$-homomorphisms. To prove the first claim in \Cref{thm:pullback} and thus that stable transfer along $\xi$ exists, we must show that pullback along $\xi_*:\Phi_\temp(H)\to\Phi_\temp(G)$ gives a well-defined continuous linear map
\[
\xi^*:\Ss^\st(G)\longrightarrow\Ss^\st(H).
\]
Recall that we have decompositions
\[
\Phi_\temp(G)=\coprod_{M_G}\Phi_2(M_G)/W^G(M_G) \quad,\quad \Phi_\temp(H)=\coprod_{M_H}\Phi_2(M_H)/W^G(M_H) 
\]
and corresponding decompositions
\[
\Ss^\st(G)=\bigoplus_{M_G}\Ss_\el^\st(M_G)^{W^G(M_G)} \quad,\quad \Ss^\st(H)=\bigoplus_{M_H}\Ss_\el^\st(M_H)^{W^H(M_H)}.
\]
Thus, it suffices to show that for each $M_H$ and $M_G$, pullback along the partially defined map 
\[
\xi_*^{M_H,M_G}:\Phi_2(M_H)/W^H(M_H)\to\Phi_2(M_G)/W^G(M_G)
\]
and extension by zero gives a well-defined continuous linear map
\[
\xi_{M_H,M_G}^*:\Ss_\el^\st(M_G)^{W^G(M_G)}\longrightarrow\Ss_\el^\st(M_H)^{W^H(M_H)}.
\]
In turn, it suffices to show that pullback along the partially defined map $\xi_*^{M_H,M_G}:\Phi_2(M_H)\to\Phi_2(M_G)$ and extension by zero gives 
a well-defined continuous linear map
\[
\xi_{M_H,M_G}^*:\Ss_\el^\st(M_G)\longrightarrow\Ss_\el^\st(M_H).
\]
We will establish this by showing that we can apply \Cref{pullback}. Moreover, applying \Cref{pullback} also gives the second claim in \Cref{thm:pullback}, thus establishing \Cref{thm:pullback} in full. We will now turn to an examination of the partially defined map $\xi_*^{M_H,M_G}:\Phi_2(M_H)\to\Phi_2(M_G)$ in more detail.

Let $\Gc$ be a $\lambda$-group. The outer action $\Gamma_F\to\Out(\Gc^0)$ determines an action $\Gamma_F\to\Aut(Z(\Gc^0))$. We will write $A_{\Gc}=Z(\Gc^0)^{\Gamma_F,\circ}$ for brevity.

\begin{lemma}
    Let $\xi:\Lgp{H}\to\Lgp{G}$ be an equivalence class of elliptic $L$-homomorphisms. Then $\xi$ restricts to a homomorphism
    \[
    \xi:A_{\Lgp{H}}\longrightarrow A_{\Lgp{G}}.
    \]
\end{lemma}

\begin{proof}
    Let $\Mcal=C_{\Lgp{G}}(\xi(A_{\Lgp{H}}))=C_{\Lgp{G}}(\xi(Z(\dual{H})^{\Gamma_F,\circ}))$. Since $Z(\dual{H})^{\Gamma_F}=C_{\dual{H}}(\Lgp{H})$, we have $\xi(\Lgp{H})\subseteq\Mcal$. Therefore $\Mcal$ maps onto $W_F$. Since $\xi(A_{\Lgp{H}})$ is a torus in $\dual{G}$, we have that $\Mcal$ is Levi subgroup of $\Lgp{G}$. Since $\xi$ is elliptic, we have 
    $\Mcal=\Lgp{G}$, or in other words
    $\xi(A_{\Lgp{H}})\subseteq(C_{\dual{G}}(\Lgp{G}))=Z(\dual{G})^{\Gamma_F}$. The claim follows.
\end{proof}

To simplify notation, for a Levi subgroup $M$ of $G$ we write the canonical equivalence class of $L$-embeddings $\iota_M^G:\Lgp{M}\to\Lgp{G}$ as an inclusion $\Lgp{M}\hookrightarrow\Lgp{G}$. Fix $\phi\in\Phi_2(M_H)$ and suppose that $\xi_*(\phi)\in\Phi_2(M_G)$. That is, $\phi$ is in the domain of the partially defined map $\xi_*^{M_H,M_G}:\Phi_2(M_H)\to\Phi_2(M_G)$.

Let $M_{H,\xi}$ be the Levi subgroup of $G$ determined up to $G(
F)$-conjugacy by the property that $\Lgp{M_{H,\xi}}$ belongs to the $\dual{G}$-conjugacy class Levi subgroups of $\Lgp{G}$ that contains $\xi(\Lgp{M_H})$ minimally. Then $\xi$ factors through $\Lgp{M_{H,\xi}}\hookrightarrow\Lgp{G}$ as an equivalence class of injective elliptic $L$-homomorphisms $\xi:\Lgp{M_H}\to\Lgp{M_{H,\xi}}$. The above lemma gives us an injection $\xi:A_{\Lgp{M_H}}\to A_{\Lgp{M_{H,\xi}}}$.

Recall that we have an identification $\Lie(A_{\Lgp{M_H}})=X_*(A_{\Lgp{M_H}})\otimes_\ZZ\CC$. Under this identification, we have that the injection $\Lie(\xi):\Lie(A_{\Lgp{M_H}})\to\Lie(A_{\Lgp{M_{H,\xi}}})$ is the map $\xi_*\otimes\id$, where $\xi_*:X_*(A_{\Lgp{M_H}})\to X_*(A_{\Lgp{M_{H,\xi}}})$. We thus write $\xi_*=\Lie(\xi)$. We have the commutative diagram
\[
\begin{tikzcd}
    \Lie(A_{\Lgp{M_H}}) \arrow[r,"\xi_*"] & \Lie(A_{\Lgp{M_{H,\xi}}}) \\
    X_*(A_{\Lgp{M_H}}) \arrow[u,hook] \arrow[r,"\xi_*"] & X_*(A_{\Lgp{M_{H,\xi}}}) \arrow[u,hook]
\end{tikzcd}
\]
Recall that we have a canonical isomorphism $X^*(M_H)\cong X_*(A_{\Lgp{M_{H}}})$ and that we write its extension $\ak_{M_{H},\CC}^*\cong\Lie(A_{\Lgp{M_H}})$ as $\lambda\mapsto\lambda^\vee$. Thus, we may view the above commutative diagram as the following commutative diagram
\[
\begin{tikzcd}
    \ak_{M_H,\CC}^* \arrow[r,"\xi_*"] & \ak_{M_{H,\xi},\CC}^* \\
    X^*(M_H) \arrow[u,hook] \arrow[r,"\xi_*"] & X^*(M_{H,\xi}) \arrow[u,hook]
\end{tikzcd}
\]
Note that the injection $\xi_*:\ak_{M_H,\CC}^*\to\ak_{M_{H,\xi},\CC}^*$ restricts to injections $\xi_*:\ak_{M_H}^*\to\ak_{M_{H,\xi}}^*$ and $\xi_*:i\ak_{M_H}^*\to i\ak_{M_{H,\xi}}^*$. 

Since $\xi_*(\phi)\in\Phi_2(M_G)$ it follows that $\Lgp{M_G}\hookrightarrow\Lgp{G}$ factors through $\Lgp{M_{H,\xi}}\hookrightarrow\Lgp{G}$. Thus, we may assume that $M_G\subseteq M_{H,\xi}$, and therefore that $M_G$ is a Levi subgroup of $M_{H,\xi}$. Then $\ak_{M_{H,\xi}}^*\subseteq\ak_{M_G}^*$ and we have $\xi_*:i\ak_{M_H}^*\to i\ak_{M_G}^*$. Now, for all $\lambda\in\ak_{M_H,\CC}^*$, we have 
\[
\xi_*(\phi_\lambda)=\xi_*(\phi)_{\xi_*(\lambda)}.
\]
This explains how $\xi_*$ behaves on the component of $\phi$. 

Suppose that $F$ is archimedean. We may assume that $F=\RR$ by restriction of scalars. We must show that $\|\xi_*\phi\|_{M_G}\gg\|\phi\|_{M_H}$ for all $\phi$ in the domain of the partially defined map $\xi_*^{M_H,M_G}:\Phi_2(M_H)\to\Phi_2(M_G)$. Since $\|\cdot\|_{M_G}\asymp\|\cdot\|_G$ and $\|\cdot\|_{M_H}\asymp\|\cdot\|_H$, it suffices to show that $\|\xi_*\phi\|_{G}\gg\|\phi\|_{H}$ for $\phi\in\Phi_\temp(H)$. 

Let $\Tc_H^0$ and $\Tc_G^0$ be $\Gamma_\RR$-stable maximal tori of $\dual{H}$ and $\dual{G}$, respectively. Choose a representative $L$-homomorphism $\xi$ such that $\xi_0(\Tc_H^0)\subseteq\Tc_G^0$. Note that $a_\xi(\CC^\times)\subseteq C_{\dual{G}}(\xi_0(\Tc_H^0))$. Since $\xi_0(\Tc_H^0)\subseteq\Tc_G^0$, we have $\Tc_G^0\subseteq C_{\dual{G}}(\xi_0(\Tc_H^0))$. Therefore $\Tc_G^0$ is a maximal torus of $C_{\dual{G}}(\xi_0(\Tc_H^0))$. By replacing $\xi$ by a $C_{\dual{G}}(\xi_0(\Tc_H^0))$-conjugate, we may assume that both $\xi(\Tc_H^0)$ and $a_\xi(\CC^\times)$ are contained in $\Tc_G^0$. Recall how the infinitesimal character attached to an $L$-parameter is defined. We may identify the infinitesimal character $\mu_\phi$ of $\phi$ with an element of $\Lie(\Tc_H^0)/W(\dual{H},\Tc_H)$. Let $\mu_\xi=\mu_{\xi|_{W_\RR}}$ denote the infinitesimal character attached to the $L$-parameter $\xi|_{W_\RR}$. We may identify $\mu_{\xi_*\phi}$ and $\mu_\xi$ with elements of $\Lie(\Tc_G^0)/W(\dual{G},\Tc_G)$. We have $\mu_{\xi_*\phi}=\Lie(\xi)\mu_\phi+\mu_\xi$. Fix a $W(\dual{H},\Tc_H)$-invariant inner product on $\Lie(\Tc_H^0)$ and a $W(\dual{G},\Tc_G)$-invariant inner product on $\Lie(\Tc_G^0)$, and denote the associated norms by $\|\cdot\|_{\dual{H}}$ and $\|\cdot\|_{\dual{G}}$, respectively. We have $\|\phi\|_H\asymp\|\mu_\phi\|_{\dual{H}}$ and $\|\xi_*\phi\|_G\asymp\|\mu_{\xi_*\phi}\|_{\dual{G}}$, so it suffices to prove that $\|\mu_{\xi_*\phi}\|_{\dual{G}}\gg\|\mu_\phi\|_{\dual{H}}$. We have
\[
\|\mu_{\xi_*\phi}\|_{\dual{G}}=\|\Lie(\xi)\mu_\phi+\mu_\xi\|_{\dual{G}}\asymp\|\Lie(\xi)\mu_\phi\|_{\dual{G}}.
\]
Since $\Lie(\xi):\Lie(\Tc_H^0)\to\Lie(\Tc_G^0)$ is injective, we have that $\|\Lie(\xi)(\cdot)\|_{\dual{G}}$ is a norm on $\Lie(\Tc_H^0)$, and is thus equivalent to $\|\cdot\|_{\dual{H}}$. Thus,
\[
\|\mu_{\xi_*\phi}\|_{\dual{G}}\asymp\|\mu_\phi\|_{\dual{H}}
\]
as required. 

In order to show that we can apply \Cref{pullback}, and thus conclude the proof of \Cref{thm:pullback}, all that remains is for us to show that the partially defined map $\xi_*^{M_H,M_G}:\Phi_2(M_H)\to\Phi_2(M_G)$ maps at most finitely many connected components to a given connected component. The next subsection is devoted to this.

\subsection{$\pi_0$-finiteness of functorial transfer}

In this subsection, we prove that functorial transfer is $\pi_0$-finite.

\begin{theorem}\label{thm:finiteness}
   Let $H$ and $G$ be connected reductive groups over $F$ with $G$ quasisplit. If $\xi:\Lgp{H}\to\Lgp{G}$ is an injective $L$-homomorphism, then the preimage of a connected component of $\Phi(G)$ under the map $\xi_*:\Phi(H)\to\Phi(G)$ is a finite union of connected components of $\Phi(H)$, that is, $\pi_0(\xi_*):\pi_0(\Phi(H))\to\pi_0(\Phi(G))$ has finite fibres. 
\end{theorem}

Since the restriction of $\xi_*$ to a connected component of $\Phi(H)$ is injective in the archimedean case and a quotient by a finite group action in the non-archimedean case, we obtain the following corollary.

\begin{corollary}
    The map $\xi_*:\Phi(H)\to\Phi(G)$ has finite fibres.
\end{corollary}

We also obtain the following corollary in the tempered case, which gives what we need to conclude the proof of \Cref{thm:pullback}.

\begin{corollary}
    If $\xi:\Lgp{H}\to\Lgp{G}$ is a tempered injective $L$-homomorphism, then the preimage of a connected component of $\Phi_\temp(G)$ under the map $\xi_*:\Phi_\temp(H)\to\Phi_\temp(G)$ is a finite union of connected components of $\Phi_\temp(H)$, that is, $\pi_0(\xi_*):\pi_0(\Phi_\temp(H))\to\pi_0(\Phi_\temp(G))$ has finite fibres. 
\end{corollary}

We now begin the proof of \Cref{thm:finiteness}. We write write $L_F^1$ for the preimage of $W_F^1$ in $L_F$. If $F$ is non-archimedean, we have $L_F^1=\SL_2\times W_F^1=\SL_2\times I_F$. If $F$ is archimedean, we have $L_F=W_F$, and thus $L_F^1=W_F^1=S^1\cup S^1j\subseteq\HH^\times$. We also write $\Lgp{G}^1=\dual{G}\rtimes W_F^1$.

For $\phi\in\Phi(G)$, we refer to $\phi|_{L_F^1}$ as its inertial $L$-parameter. (See \cite{Latham} for a discussion of inertial $L$-parameters for $p$-adic groups.) For $\lambda\in\ak_{G,\CC}^*$, its associated cohomology class $a_\lambda\in H_c^1(W_F,Z(\dual{G}))$ is trivial on $W_F^1$. Consequently, for $\phi\in\Phi(G)$ and $\lambda\in\ak_{G,\CC}^*$ we have $\phi_\lambda|_{L_F^1}=\phi|_{L_F^1}$, that is, $\phi_\lambda$ and $\phi$ have the same inertial $L$-parameter. Thus, the restriction map
\[
\Phi(G)\longrightarrow\Hom_{c,W_F^1}(L_F^1,\Lgp{G}^1)/\dual{G}
\]
is constant on connected components. Here, we have written $\Hom_{c,W_F^1}(L_F^1,\Lgp{G}^1)$ to denote the set of continuous homomorphisms $L_F^1\to\Lgp{G}^1$ over $W_F^1$. We will use similar notation below.

\begin{proposition}
    Let $G$ be a connected reductive group over $F$. The fibres of the restriction map
    \[
    \Phi(G)\longrightarrow\Hom_{c,W_F^1}(L_F^1,\Lgp{G}^1)/\dual{G}
    \]
    are finite unions of connected components. When $F$ is archimedean, the fibres are the connected components of $\Phi(G)$.
\end{proposition}

When $F$ is non-archimedean, the fibres can indeed be unions of more than one connected component. Suppose that $G$ is anisotropic modulo its centre. Then $G(F)^1$ is the unique maximal compact subgroup of $G(F)$ and $G(F)_1$ is the unique Iwahori subgroup of $G(F)$ \cite[\S3.3.1]{Haines}. Suppose further that $G$ is a torus. Note 1 at the end of \cite{Rapoport} gives a description of $G(F)_1$ and $G(F)^1$. The discussion there initially concerns a complete discretely valued field $L$ with algebraically closed residue field, however at the end of Note 1, Rapoport explains the necessary changes when $L$ is replaced by a discretely valued field $F$ with perfect residue field. The group $G$ has an lft N\'eron model $\Gc$ over $\Spec\Oc_F$. Let $\Gc^1$ be the maximal subgroup scheme of finite type over $\Spec\Oc_F$. Let $\Gc^{1,\circ}$ be the identity component of $\Gc^1$. Then $\Gc^1(\Oc_F)=G(F)^1$, $\Gc^{1,\circ}(\Oc_F)=G(F)_1$, and $G(F)_1$ has finite index in $G(F)^1$. Identify $\Phi(G)$ with $\Pi(G)$ using the local Langlands correspondence for tori. The fibre in the proposition containing the trivial character $1:G(F)\to\CC^\times$ contains the group of weakly unramified characters $X^\wnr(G)=\Hom_c(G(F)/G(F)_1,\CC^\times)$, but the connected component of $1$ is the group of unramified characters $X^\nr(G)=\Hom_c(G(F)/G(F)^1,\CC^\times)$. If $G(F)_1\subsetneq G(F)^1$, then $X^\wnr(G)\supseteq X^\nr(G)$.

The following proof was inspired by the proof of \cite[Lemma 7.2.3]{varma}.

\begin{proof}
    The fibres are unions of connected components. We must show that the fibres contain only finitely many connected components. Fix $\phi\in\Phi(G)$ and let $\phi_0:L_F^1\to\Lgp{G}$ denote the restriction of $\phi$. Let $\phi'\in\Phi(G)$ be in the fibre containing $\phi$. By replacing $\phi'$ by a $\dual{G}$-conjugate if necessary, we may assume that $\phi'$ restricts to $\phi_0$. 
    
    Assume that $F$ is non-archimedean. Then $L_F^1=\SL_2\times I_F$. Let 
    \[
    C_{\phi_0}=C_{\dual{G}}(\phi_0(\SL_2\times I_F))=C_{\dual{G}}(\phi_0(I_F))\cap C_{\dual{G}}(\phi_0(\SL_2)).
    \]
    Both $C_{\dual{G}}(\phi_0(I_F))$ and $C_{\dual{G}}(\phi_0(\SL_2))$ are reductive groups by Lemma 10.2.2 and the proof of Lemma 10.1.1 in \cite{KotCuspidal}. Therefore $C_{\phi_0}$ is reductive.

    Both $\phi(\Fr)$ and $\phi'(\Fr)$ normalise $\phi_0(\SL_2\times I_F)$, and therefore also normalise $C_{\phi_0}$. Let $\theta,\theta'\in\Aut(C_{\phi_0})$ denote the automorphisms defined by $\Int(\phi(\Fr))$ and $\Int(\phi'(\Fr))$. Since $\phi(\Fr)$ and $\phi'(\Fr)$ are semisimple elements of $\Lgp{G}$, the automorphisms $\theta,\theta'$ are semisimple. Therefore they each preserve a Borel pair of $C_{\phi_0}$ by \cite[Theorem 7.5]{SteinbergEndom}. By replacing $\phi,\phi'$ by $C_{\phi_0}$-conjugates if necessary, we may assume that $\theta,\theta'$ preserve the same Borel pair $(\Bc,\Tc)$ of $C_{\phi_0}^\circ$.

    Let $x=\phi'(\Fr)^{-1}\phi(\Fr)\in\dual{G}$. Since $\phi(\Fr)$ and $\phi'(\Fr)$ act on $\phi_0(\SL_2\times I_F)$ in the same way, we have that $x$ acts trivially on $\phi_0(\SL_2\times I_F)$, and thus $x\in C_{\phi_0}$. Moreover, $x\in N_{C_{\phi_0}}(\Bc,\Tc)$. We have $N_{C_{\phi_0}}(\Bc,\Tc)/\Tc\simeq C_{\phi_0}/C_{\phi_0}^\circ$ as explained in \cite[\S1]{DigneMichel}. Fix representatives $c_1,\dots,c_N$ of $N_{C_{\phi_0}}(\Bc,\Tc)/\Tc$. For each $i=1,\dots,N$, let $\phi_i:\SL_2\times W_F\times\to\Lgp{G}$ be the extension of $\phi_0$ that satisfies $\phi_i(\Fr)=c_i\phi(\Fr)$. We will show that $\phi'$ lies in the component of $\phi_i$ for some $i=1,\dots,N$.
    
    Let $\Mcal=C_{\Lgp{G}}(\Tc^{\theta',\circ})$. Then $\Mcal$ is a Levi subgroup of $\Lgp{G}$ containing the image of $\phi'$. Let $\Mcal'\subseteq\Mcal$ be a minimal Levi of $\Lgp{G}$ containing the image of $\phi'$. Then we have $\Tc^{\theta',\circ}\subseteq A_{\Mcal'}$. Therefore, it suffices to show that there exists $t_0\in\Tc^{\theta',\circ}$ and $t_1\in\Tc$ such that $t_0t_1\phi'(\Fr)t_1^{-1}=c_i\phi(\Fr)$ for some $i=1,\dots,N$. Indeed, this says that the unramified twist of $\phi'$ by $t_1\in A_{\Mcal'}=H_c^1(W_F/I_F,A_{\Mcal'})$ is equivalent to $\phi_i$, so that $\phi'$ lies in the connected component of $\phi_i$.
    
    We may write $x=tc_i$ for some $t\in\Tc$ and some $i=1,\dots,N$. Then $t^{-1}\phi'(\Fr)=c_i\phi(\Fr)$. Let $(1-\theta')\Tc=\{t\theta'(t)^{-1} : t\in\Tc\}$. It follows from the Smith normal form that the homomorphism
    \[
    \Tc^{\theta',\circ}\times(1-\theta')\Tc\longrightarrow\Tc
    \]
    given by multiplication in $\Tc$ is surjective. Therefore, we may write $t^{-1}=t_0t_1\theta'(t_1)^{-1}$ for some $t_0\in\Tc^{\theta',\circ}$ and $t_1\in\Tc$. Then 
    \begin{align*}
        c_i\phi(\Fr)&=t_0t_1\theta'(t_1)^{-1}\phi'(\Fr) \\
        &=t_0t_1\phi'(\Fr)t_1^{-1}\phi'(\Fr)^{-1}\phi'(\Fr) \\
        &=t_0t_1\phi'(\Fr)t_1^{-1}.
    \end{align*}
    This concludes the proof in the non-archimedean case.

    Now assume that $F$ is archimedean. By restriction of scalars, we may assume that $F=\RR$. Recall that $L_\RR=W_\RR=W_\RR^1\times\RR_{>0}$. Since $\RR_>0$ lies in the centre of $W_\RR$ and acts trivially on $\dual{G}$, we obtain that both $a_\phi(\RR_{>0})$ and $a_{\phi'}(\RR_{>0})$ centralise $\phi_0(W_\RR^1)$. Therefore $a_\phi(\RR_{>0})$ and $a_{\phi'}(\RR_{>0})$ are contained in $C_{\phi_0}\defeq C_{\dual{G}}(\phi_0(W_\RR^1))$. Since $a_\phi(\RR_{>0})$ and $a_{\phi'}(\RR_{>0})$ are connected abelian subgroups of $C_{\phi_0}$ consisting of semisimple elements, they are contained in maximal tori of $C_{\phi_0}$. Let $\Tc$ be a maximal torus of $C_{\phi_0}$. By replacing $\phi,\phi'$ by $C_{\phi_0}$-conjugates if necessary, we may assume that $a_\phi(\RR_{>0})$ and $a_{\phi'}(\RR_{>0})$ are both contained in $\Tc$. For each $w\in W_\RR$, define $a(w)=\phi'(w)\phi(w)^{-1}$. For $w\in W_\RR^1$ we have $a(w)=1$, and for $r\in\RR_{>0}$ we have $a(r)=a_{\phi'}(r)a_\phi(r)^{-1}$. Therefore $a\in\Hom_c(W_\RR/W_\RR^1,\Tc)$. 
    
    Since $a_\phi(\RR_{>0})\subseteq\Tc$, it follows that that $\phi(\RR_{>0})$ centralises $\Tc$. Since $\phi(W_\RR^1)=\phi_0(W_\RR^1)$ also centralises $\Tc$, we have that $\phi(W_\RR)$ is contained in $\Mcal\defeq C_{\Lgp{G}}(\Tc)$. Now, $\Mcal$ is a Levi subgroup of $\Lgp{G}$ containing the image of $\phi$. Let $\Mcal_\phi\subseteq\Mcal$ be a minimal Levi of $\Lgp{G}$ containing the image of $\phi$. Then $\Tc\subseteq A_{\Mcal_\phi}$. Therefore $a\in\Hom_c(W_\RR/W_\RR^1,A_{\Mcal_\phi})$ and  $\phi'=a\cdot\phi$, an unramified twist of $\phi$. Thus, $\phi'$ lies in the same connected component as $\phi$.
\end{proof}

We have the following commutative diagram
\[
\begin{tikzcd}
    \Phi(H) \arrow[r,"\xi_*"] \arrow[d] & \Phi(G) \arrow[d] \\
    \Hom_{c,W_F^1}(L_F^1,\Lgp{H})/\dual{H} \arrow[r,"\xi_*"] & \Hom_{c,W_F^1}(L_F^1,\Lgp{G})/\dual{G}
\end{tikzcd}
\]
The right vertical arrow maps each connected component to a point. The fibres of the left vertical arrow are finite unions of connected components. Therefore to prove \Cref{thm:finiteness}, it suffices to show that the bottom map has finite fibres. To do so, we will make use of the following theorem of Vinberg \cite[Theorem 1]{vinberg}.

\begin{theorem} \label{thm:Vinberg}
    Let $k$ be an algebraically closed field of characteristic zero. Let $\xi:\Hc\to\Gc$ be an embedding of (not necessarily connected) reductive groups over $k$. For each $n\in\ZZ_{\geq0}$, the natural morphism
    \[
    \xi_*:\Hc^n//\Hc\longrightarrow\Gc^n//\Gc
    \]
    is finite. (The quotients are GIT quotients with respect to the conjugation actions.)
\end{theorem}

\begin{proposition}\label{prop:finiteness}
    Let $k$ be an algebraically closed field of characteristic zero. Let $\xi:\Hc\to\Gc$ be an embedding of (not necessarily connected) reductive groups over $k$. Equip $\Hc$ and $\Gc$ with a topology that is at least as fine as the Zariski topology. The pushforward map of conjugacy classes of continuous homomorphisms
    \[
    \xi_*:\Hom_c(\Gamma,\Hc)/\Hc\longrightarrow\Hom_c(\Gamma,\Gc)/\Gc
    \]
    has finite fibres.
\end{proposition}

\begin{proof}
    Let $\Hc_\Zar$ and $\Gc_\Zar$ denote $\Hc$ and $\Gc$ equipped with the Zariski topology. We have a commutative diagram
    \[
    \begin{tikzcd}
        \Hom_c(\Gamma,\Hc)/\Hc \arrow[r,"\xi_*"] \arrow[d,hook] & \Hom_c(\Gamma,\Gc)/\Gc \arrow[d,hook] \\
        \Hom_c(\Gamma,\Hc_\Zar)/\Hc \arrow[r,"\xi_*"] & \Hom_c(\Gamma,\Gc)_\Zar)/\Gc
    \end{tikzcd}
    \]
    where the vertical arrows are inclusions. It follows that it suffices to show that the bottom arrow has finite fibres. We may thus assume that $\Hc$ and $\Gc$ are equipped with the Zariski topology. 
    
    Let $\gamma_1,\dots,\gamma_n$ be a set of topological generators for $\Gamma$. This determines a commutative diagram
    \[
    \begin{tikzcd}
        \Hom_c(\Gamma,\Hc)/\Hc \arrow[r,"\xi_*"] \arrow[d,hook] & \Hom_c(\Gamma,\Gc)/\Gc \arrow[d,hook] \\
    \Hc^n//\Hc \arrow[r,"\xi_*"] & \Gc^n//\Gc
    \end{tikzcd}
    \]
    The bottom arrow has finite fibres by \Cref{thm:Vinberg}, and therefore so does the top arrow, as claimed.
\end{proof}

\begin{proposition}
    Let $H$ and $G$ be connected reductive groups over $F$. Let $\xi:\Lgp{H}\to\Lgp{G}$ be a continuous injective homomorphism over $W_F$ such that the restriction $\xi_0:\dual{H}\to\dual{G}$ is algebraic. The map
    \[
    \xi_*:\Hom_{c,W_F^1}(L_F^1,\Lgp{H}^1)/\dual{H}\longrightarrow \Hom_{c,W_F}(L_F^1,\Lgp{G}^1)/\dual{G}
    \]
    has finite fibres.
\end{proposition}

Recall that we write $\xi(h,w)=(\xi_0(h)a_\xi(w),w)$ for all $w\in W_F$. Note that we can identify the map $\xi_*$ with the map
\[
\xi_*:H_c^1(L_F^1,\dual{H})\longrightarrow H_c^1(L_F^1,\dual{G})
\]
defined by $\xi_*a=\xi_{0,*}a\cdot a_\xi$, that is, $\xi_*a(l)=\xi_0*(a(l))a_\xi(w(l))$. We will identify $a_\xi$ with its inflation to $L_F$ and thus simply write $a_\xi(l)=a_\xi(w(l))$. It will be convenient to view $\xi_*$ in this way in order to save space and use language from group cohomology.

\begin{proof}
    Let $a\in H_c^1(L_F^1,\dual{H})$. We will show that $\xi_*^{-1}(\xi_*a)$ is finite.
    
    Assume that $F$ is non-archimedean. Then $W_F^1=I_F$ and $L_F^1=\SL_2\times I_F$. There exists a compact open normal subgroup $J$ of $I_F$ such that $J$ acts trivially on $\dual{H}$ and $\dual{G}$, and $a(J)=a_\xi(J)=1$. Since $a_\xi(J)=1$, the homomorphism $\xi:\dual{H}\rtimes W_F^1\to\dual{G}\rtimes W_F^1$ descends to a homomorphism
    \[
    \xi:\dual{H}\rtimes W_F^1/J\to\dual{G}\rtimes W_F^1/J.
    \]
    We have a commutative diagram
    \[
    \begin{tikzcd}
        H_c^1(L_F^1/J,\dual{H}) \arrow[d,hook] \arrow[r,"\xi_*"] & H_c^1(L_F^1/J,\dual{H}) \arrow[d,hook] \\
        H_c^1(L_F^1,\dual{H}) \arrow[r,"\xi_*"] & H_c^1(L_F^1,\dual{H})
    \end{tikzcd}
    \]
    where the vertical arrows are the injective inflation maps.
    
    Since $\xi_*a=\xi_{0,*}a\cdot a_\xi$, we have $\xi_*a(J)=1$. Let $a'\in\xi_*^{-1}(\xi_*a)$. Then there exists $g\in\dual{G}$ such that 
    \[
    \xi_0(a'(l))a_\xi(l)=\xi_*a'(l)=g\xi_*a(l)\pre{l\!}{g^{-1}},
    \]
    for all $l\in L_F^1$. Since $\xi_*(J)=a_\xi(J)=1$, we have $\xi(a'(J))=1$. Since $\xi$ is injective, we obtain $a'(J)=1$. Thus, we have that $\xi_*^{-1}(\xi_*a)$ lies in $H_c^1(L_F^1/J,\dual{H})\hookrightarrow H_c^1(L_F^1,\dual{H})$. Consequently, we obtain the commutative diagram
    \[
    \begin{tikzcd}
        \xi_*^{-1}(\xi_*a) \arrow[r,"\xi_*"] \arrow[d,hook] & \{\xi_*a\} \arrow[d] \arrow[d,hook] \\
        H_c^1(L_F^1/J,\dual{H}) \arrow[r,"\xi_*"] & H_c^1(L_F^1/J,\dual{G})
    \end{tikzcd}
    \]
    and it suffices to show that the bottom map has finite fibres. This map can be viewed as the map
    \[
    \begin{tikzcd}
        \Hom_{c,I_F/J}(L_F^1/J,\dual{H}\rtimes I_F/J)/\dual{H} \arrow[r,"\xi_*"] & \Hom_{c,I_F/J}(L_F^1/J,\dual{G}\rtimes I_F/J)/\dual{G}
    \end{tikzcd}
    \]
    where we are writing $\Hom_{c,I_F/J}$ to indicate continuous homomorphisms over $I_F/J$. We have the commutative diagram
    \[
    \begin{tikzcd}[column sep=small]
        \Hom_{c,I_F/J}(L_F^1/J,\dual{H}\rtimes I_F/J)/\dual{H} \arrow[r,"\xi_*"] \arrow[d] & \Hom_{c,I_F/J}(L_F^1/J,\dual{G}\rtimes I_F/J)/\dual{G} \arrow[d] \\
        \Hom_{c,I_F/J}(L_F^1/J,\dual{H}\rtimes I_F/J)/\conj \arrow[r,"\xi_*"] \arrow[d,hook] & \Hom_{c,I_F/J}(L_F^1/J,\dual{G}\rtimes I_F/J)/\conj \arrow[d,hook] \\
        \Hom_c(\SL_2\times I_F/J,\dual{H}\rtimes I_F/J)/\conj \arrow[r,"\xi_*"] & \Hom_c(\SL_2\times I_F/J,\dual{G}\rtimes I_F/J)/\conj
    \end{tikzcd}
    \]
    where we have written ``$/\conj$'' to indicate quotients by conjugacy of the codomain. The fibres of the upper vertical arrow on the left are finite as they are the sets of $\dual{H}$-conjugacy classes that comprise a conjugacy class under $\dual{H}\rtimes I_F/J$ and $I_F/J$ is finite. Since the bottom arrow has finite fibres by \Cref{prop:finiteness}, it follows that the top arrow does and therefore $\xi_*^{-1}(\xi_*a)$ is finite.

    Now, assume that $F$ is archimedean. By restriction of scalars, we may assume that $F=\RR$. Recall that $L_\RR=W_\RR=\CC^\times\cup\CC^\times j$ and $L_\RR^1=W_\RR^1=S^1\cup S^1j$. We have $W_\RR=W_\RR^1\times\RR_{>0}$. We have an extension of topological groups
    \[
    \begin{tikzcd}
        1 \arrow[r] & \{\pm1\} \arrow[r] & S^1\rtimes\langle j\rangle \arrow[r] & W_\RR^1 \arrow[r] & 1
    \end{tikzcd}
    \]
    where the homomorphism $\{\pm1\}\to S^1\rtimes\langle j\rangle$ is the diagonal embedding and the homomorphism $S^1\rtimes\langle j\rangle\to W_\RR^1$ is given by multiplication in $W_\RR^1$. We have a commutative diagram
    \[
    \begin{tikzcd}
        H_c^1(W_\RR^1,\dual{H}) \arrow[r,"\xi_*"] \arrow[d,hook] &  H_c^1(W_\RR^1,\dual{G}) \arrow[d,hook] \\
        H_c^1(S^1\rtimes\langle j\rangle,\dual{H}) \arrow[r,"\xi_*"] & H_c^1(S^1\rtimes\langle j\rangle,\dual{G})
    \end{tikzcd}
    \]
    where the vertical arrows are the inflation maps. Thus, it suffices to show that the bottom map has finite fibres. This map can be viewed as the pushforward map
    \[
    \xi_*:\Hom_{c,\langle j\rangle}(S^1\rtimes\langle j\rangle,\dual{H}\rtimes\langle j\rangle)/\dual{H}\longrightarrow\Hom_{c,\langle j\rangle}(S^1,\rtimes\langle j\rangle,\dual{G}\rtimes\langle j\rangle)/\dual{G}
    \]
    where $\Hom_{c,\langle j\rangle}$ indicates continuous homomorphisms over $\langle j\rangle$. We have the commutative diagram
    \[
    \begin{tikzcd}
        \Hom_{c,\langle j\rangle}(S^1\rtimes\langle j\rangle,\dual{H}\rtimes\langle j\rangle)/\dual{H} \arrow[r,"\xi_*"] \arrow[d] & \Hom_{c,\langle j\rangle}(S^1,\rtimes\langle j\rangle,\dual{G}\rtimes\langle j\rangle)/\dual{G} \arrow[d] \\
        \Hom_{c,\langle j\rangle}(S^1\rtimes\langle j\rangle,\dual{H}\rtimes\langle j\rangle)/\conj \arrow[r,"\xi_*"] \arrow[d,hook] & \Hom_{c,\langle j\rangle}(S^1\rtimes\langle j\rangle,\dual{G}\rtimes\langle j\rangle)/\conj \arrow[d,hook] \\
        \Hom_{c}(S^1\rtimes\langle j\rangle,\dual{H}\rtimes\langle j\rangle)/\conj \arrow[r,"\xi_*"] & \Hom_{c}(S^1\rtimes\langle j\rangle,\dual{G}\rtimes\langle j\rangle)/\conj
    \end{tikzcd}
    \]
    As above, the left vertical map has finite fibres. By \Cref{prop:finiteness}, the bottom map has finite fibres. Consequently, the top map has finite fibres.
\end{proof}

This completes the proof of \Cref{thm:finiteness}, and thus of \Cref{thm:pullback} and its corollary \Cref{existence}. In particular, we have an affirmative answer to \cite[Questions A \& B]{LanST}.

\subsection{Examples}
Let $H$ and $G$ be connected reductive groups over a local field $F$ of characteristic zero with $G$ quasisplit, and let $\xi:\Lgp{H}\to\Lgp{G}$ be an equivalence class of injective tempered $L$-parameters. The stable transfer operator $\Tc_\xi$ is given by the distribution kernel $\Kc_\xi(\delta_H,\delta_G)$ defined by the formal integral
\[
\Kc_\xi(\delta_H,\delta_G)=\int_{\phi\in\Phi_\temp(H)}S^H(\delta_H,\phi)S^G(\xi_*(\phi),\delta_G)\dd{\phi},
\]
where $S^G(\xi_*(\phi),\delta_G)=|D^G(\delta)|^{1/2}\Theta_{\xi_*(\phi)}(\delta_G)$ is the normalised stable tempered character of $\xi_*(\phi)$ and $S^H(\delta_H,\phi)$ is the kernel of stable Fourier inversion in \Cref{thm:stableinversion}. One would like to describe $\Kc_\xi(\delta_H,\delta_G)$ as explicitly as possible. We conclude by describing $\Tc_\xi$ explicitly in some simple cases. First, we note some cases that go back to the work of Harish-Chandra. 

Suppose that $H=1$, in which case $\Lgp{H}=W_F$ and $\xi$ is just a tempered $L$-parameter of $G$. Let $f^G\in\Sc(G)$. Then the function $\Tc_\xi f^G$ on $H(F)=1$ is just the constant $f^G(\xi)=\Theta_\xi(f^G)$. That is, $\Tc_\xi$ is given by
\[
\Tc_\xi f^G=\int_{\Delta_\sr(G)}|D^G(\delta)|^{1/2}\Theta_\xi(\delta)f^G(\delta)\dd{\delta}.
\]
To obtain a formula for $\Theta_\xi(\delta)$, it suffices to treat the case when $\xi$ is discrete by the formula for parabolically induced characters \cite[Lemma 4.7.6]{Silberger}. Harish-Chandra obtained formulas for stable discrete series characters of real groups (see \cite{GKM}). For $p$-adic groups, formulas for stable discrete series characters are not known in general. See \cite{KalICM} for the current state of the art.

Suppose that $H$ is a Levi subgroup $M$ of $G$ and $\xi:\Lgp{M}\to\Lgp{G}$ is the canonical equivalence class of $L$-embeddings. Then $\Tc_\xi:\Sc(G)\to\Sc(M)$ is parabolic descent. Thus, stable transfer coincides with suitably normalised endoscopic transfer in this case \cite[\S1]{AdamsMBVogan}. If $f\in\Sc(G)$, then for $G$-regular semisimple elements $m\in M(F)$, we have $(\Tc_\xi f^G)(m)=f^G(m)$. 

Suppose $G=H^*$ is the quasisplit inner form of $H$ and $\xi:\Lgp{H}\to\Lgp{H^*}$ is the identity. Endoscopic transfer $\Tc_{\End}:\Ic(H)\to\Sc(H^*)$ descends to a continuous linear map $\Tc_H^{H^*}:\Sc(H)\to\Sc(H^*)$. Furthermore, it can be normalised so that $(\Tc_H^{H^*} f^H)(\phi)=f^H(\phi)$ for all $\phi\in\Phi(H)$. This is proved for real groups in \cite{She79} and is part of the refined local Langlands conjecture for $p$-adic groups \cite{KalethaRefined}. For $f^H\in\Sc(H)$ and $\phi\in\Phi(H)$ we have $(\Tc_\xi\Tc_H^{H^*} f^H)(\phi)=(\Tc_H^{H^*} f^H)(\phi)=f^H(\phi)$, and thus $\Tc_\xi\Tc_H^{H^*} f^H=f^H$ by spectral density. That is, endoscopic transfer gives a section for stable transfer along $\xi$.

\subsubsection{Tori.}
Let $S$ and $T$ be tori over $F$ and let $\xi:\Lgp{S}\to\Lgp{T}$ be an injective tempered $L$-embedding. Since $\dual{T}$ is abelian, the restriction $\xi_0:\dual{S}\to\dual{T}$ of $\xi$ is $W_F$-equivariant, or equivalently $\Gamma_F$-equivariant. Let $\xi_0^*:T\to S$ be the homomorphism with $(\xi_0^*)^\vee=\xi_0$. Then $\xi_{0,*}:H_c^1(W_F,\dual{S}^{(1)})\to H_c^1(W_F,\dual{T}^{(1)})$ and $\xi_0^*:T(F)\to S(F)$ are adjoint with respect to the Langlands pairings for $S$ and $T$. Let $a\in H_c^1(W_F,(\dual{S})^{1})=\Phi_\temp(S)$. We have $\xi_*(a)=a_\xi\cdot\xi_{0,*}(a)$, where $a_\xi\in Z_c^1(W_F,(\dual{T})^1)$ is the 1-cocycle determined by $\xi$. For $t\in T(F)$, we have
\[
\chi_{\xi_*(a)}(t)=\langle\xi_*(a),t \rangle=\langle a_\xi,t \rangle\langle (\xi_{0,*}a),t\rangle=\langle a_\xi,t \rangle\langle a,\xi_0^*(t) \rangle=\chi_{a_\xi}\chi_{a}(\xi_0^*(t)).
\]
We will write $\chi_\xi=\chi_{a_\xi}$ for brevity. 

Since $\xi_0$ is injective, $\xi_0^*:T\to S$ is a quotient homomorphism. The Haar measure $\dd{t}$ on $T(F)$ disintegrates into measures $\mu_s$ on the fibres $(\xi_0^*)^{-1}(s)$ such that 
\[
\int_{T(F)}\dd{t}=\int_{S(F)}\int_{(\xi_0^*)^{-1}(s)}\dd{\mu_s(t)}\dd{s}.
\]
We can describe the measures $\mu_s$ concretely as follows. Let $D=\ker(\xi_0^*:T(F)\to S(F))$. If $s=\xi_0^*(t)$ we have $(\xi_0^*)^{-1}(s)=tD$ and the measure $\mu_s$ on $(\xi_0^*)^{-1}(s)$ is obtained by transporting a suitably normalised Haar measure on $D$ along the map $D\to tD, d\mapsto td$. 

Let $f\in\Sc(T)=\Cc(T)$. For all $a\in H_c^1(W_F,(\dual{S})^1)=\Phi_\temp(S)$, we have 
\begin{align*}
    \int_{S(F)}\chi_a(s) (\Tc_\xi f)(s) \dd{s}&=(\Tc_\xi f)(a) \\
    &=f(\xi_*(a)) \\
    &=\int_{T(F)}\chi_{\xi_*(a)}(t)f(t)\dd{t} \\
    &=\int_{T(F)}\chi_\xi(t)\chi_a(\xi_0^*(t))(t)f(t)\dd{t} \\
    &=\int_{S(F)}\chi_a(s)\int_{(\xi_0^*)^{-1}(s)}\chi_\xi(t)f(t)\dd{\mu_s(t)}\dd{s}.
\end{align*}
It follows that for all $s\in S(F)$ we have
\begin{align*}
(\Tc_\xi f)(s)&=\int_{(\xi_0^*)^{-1}(s)}\chi_\xi(t)f(t)\dd{\mu_s(t)} \\
&=\int_{T(F)}\chi_\xi(t)f(t)\dd{\mu_s(t)}
\end{align*}
where we have used $\mu_s$ to also denote the pushforward of $\mu_s$ to $T(F)$.

\subsubsection{Complex groups.}
Let $H,G$ be connected reductive groups over $\CC$ and let $\xi:\Lgp{H}\to\Lgp{G}$ be an equivalence class of injective tempered $L$-homomorphism. Let $S$ be a minimal Levi subgroup (maximal torus) of $H$ and let $T$ be a minimal Levi subgroup (maximal torus) of $G$. We have a commutative diagram
\[
\begin{tikzcd}
    \Lgp{H} \arrow[r, "\xi"] & \Lgp{G} \\
    \Lgp{S} \arrow[u] \arrow[r, "\xi"] & \Lgp{T} \arrow[u] \\
\end{tikzcd}
\]
of equivalence classes of $L$-homomorphisms. We may apply the considerations of the preceding subsection to $\xi:\Lgp{S}\to\Lgp{T}$. 

Define $s\in S$ to be $\xi$-regular if
\[
\dim(T_{G\dash\sing}\cap(\xi_0^*)^{-1}(s))<\dim((\xi_0^*)^{-1}(s))
\]
and $\xi$-singular otherwise. Let $S_{\xi\dash\reg}$ (resp. $S_{\xi\dash\sing}$) denote the set of $\xi$-regular (resp. $\xi$-singular) elements of $S$. If $s\in S_{\xi\dash\reg}$, then $\mu_s(T_{G\dash\sing}\cap(\xi_0^*)^{-1}(s))=0$. Let $D=\ker(\xi_0^*)$. We have a direct product decomposition $D=AD^\circ$ for a finite subgroup $A\subseteq D$.

\begin{lemma}
    We have $S_{\xi\dash\sing}=\bigcup_{a\in A}\bigcup_{\alpha\in\Phi(G,T), \alpha(D^\circ)=1}\xi_0^*(a\ker\alpha)$. Furthermore, $S_{\xi\dash\sing}$ is a closed subvariety of $S$ of positive codimension.
\end{lemma}

\begin{proof}
    Suppose that $s\in S_{\xi\dash\sing}$, that is $s\in S$ with
    \[
    \dim(T_{G\dash\sing}\cap(\xi_0^*)^{-1}(s))=\dim((\xi_0^*)^{-1}(s)).
    \]
    Then $T_{G\dash\sing}$ contains an irreducible component of $(\xi_0^*)^{-1}(s)$ of maximum dimension.  Write $(\xi_0^*)^{-1}(s)=tD$. Then there exists $a\in A$ such that $taD^\circ\subseteq T_{G\dash\sing}$. Let $\alpha\in\Phi(G,T)$ be a root such that $\alpha(ta D^\circ)=1$. Then $\alpha(ta)=1$ and $\alpha(D^\circ)=1$. It follows that 
    \[
    s\in\bigcup_{a\in A}\bigcup_{\alpha\in\Phi(G,T), \alpha(D^\circ)=1}\xi_0^*(a\ker\alpha).
    \]

    Conversely, suppose that $s\in\bigcup_{a\in A}\bigcup_{\alpha\in\Phi(G,T), \alpha(D^\circ)=1}\xi_0^*(a\ker\alpha)$. There exist $a\in A$, $\alpha\in\Phi(G,T)$ with $\alpha(D^\circ)=1$, and $c\in\ker\alpha$ such that $s=\xi_0^*(ak)$. Then $(\xi_0^*)^{-1}(s)=akD$. Now, $kD^\circ$ is an irreducible component of $(\xi_0^*)^{-1}(s)$ of maximum dimension that is contained in $\ker\alpha$, and thus $T_{G\dash\sing}$. Therefore $s\in S_{\xi\dash\sing}$.

    We have shown that $S_{\xi\dash\sing}=\bigcup_{a\in A}\bigcup_{\alpha\in\Phi(G,T), \alpha(D^\circ)=1}\xi_0^*(a\ker\alpha)$. Now, we have $\xi_0^*(a\ker\alpha)=\xi_0^*(a)\xi_0^*(\ker\alpha)$, and $\xi_0^*(\ker\alpha)$ is a closed subgroup of $S$. Therefore $S_{\xi\dash\sing}$ is a closed subvariety of $S$. Moreover, we have
    \begin{align*}
    \dim \xi_0^*(a\ker\alpha)&=\dim\ker\alpha - \dim\ker\xi_0^* \\
    &\leq \dim T-1-\dim\ker\xi_0^* \\
    &=\dim S -1.
    \end{align*}
    Thus, $S_{\xi\dash\sing}$ has positive codimension in $S$.
\end{proof}

Let $f^G\in\Sc(G)$ and write $f^H$, $f^T$, $f^S$ for its associated stable transfers to $H$, $T$, and $S$, respectively. Let $s\in S_{H\dash\reg}$. We have $f^H(s)=f^S(s)$ by parabolic descent. Using the notation of the preceding subsection, we obtain
\[
f^H(s)=\int_{T}\chi_\xi(t)f^T(t)\dd{\mu_s(t)}.
\]
Suppose further that $s\in S_{\xi\dash\reg}$. Then $\mu_s(T_{G\dash\sing}\cap(\xi_0^*)^{-1}(s))=0$, and  
\[
f^H(s)=\int_{T_{G\dash\reg}}\chi_\xi(t)f^T(t)\dd{\mu_s(t)}.
\]
Parabolic descent gives us
\[    f^H(s)=\int_{T_{G\dash\reg}}\chi_\xi(t)f^G(t)\dd{\mu_s(t)}
\]
Let $c_G:T_{G\dash\reg}\to T_{G\dash\reg}/W(G,T)=\Delta_\rs(G)$ be the natural map. Then
\[
f^H(s)=\int_{\Delta_\rs(G)}f^G(\delta)\dd{(c_G)_*(\chi_\xi\mu_s)(\delta)}
\]

Since $S_{H\dash\sing}\cup S_{\xi\dash\sing}$ is a closed subvariety of $S$ of positive codimension, the set of $H$-regular and $\xi$-regular elements of $S$ is an open dense subset. Its image $\Delta_{\rs,\xi\dash\reg}(H)$ in $\Delta_\rs(H)$ is an open dense subset. Thus, the above formula determines $f^H$.

\bibliographystyle{amsalpha}
\bibliography{refs.bib}

\newcommand{\etalchar}[1]{$^{#1}$}
\providecommand{\bysame}{\leavevmode\hbox to3em{\hrulefill}\thinspace}
\providecommand{\MR}{\relax\ifhmode\unskip\space\fi MR }
\providecommand{\MRhref}[2]{%
  \href{http://www.ams.org/mathscinet-getitem?mr=#1}{#2}
}
\providecommand{\href}[2]{#2}
\begin{thebibliography}{CFM{\etalchar{+}}22}

\bibitem[ABV92]{ABV}
Jeffrey Adams, Dan Barbasch, and David~A. Vogan, Jr., \emph{The {L}anglands
  classification and irreducible characters for real reductive groups},
  Progress in Mathematics, vol. 104, Birkh\"{a}user Boston, Inc., Boston, MA,
  1992. \MR{1162533}

\bibitem[AK]{AdamsKaletha}
Jeffrey Adams and Tasho Kaletha, \emph{Discrete series ${L}$-packets for real
  reductive groups}, Preprint.

\bibitem[AMBV]{AdamsMBVogan}
Jeffrey Adams, Lucas Mason-Brown, and David Vogan, \emph{A generalization of
  endoscopic lifting and the unitarity of arthur packets}, Preprint.

\bibitem[Art81]{ArtTFInvForm}
James Arthur, \emph{The trace formula in invariant form}, Ann. of Math. (2)
  \textbf{114} (1981), no.~1, 1--74. \MR{625344}

\bibitem[Art89]{ArtIntResI}
\bysame, \emph{Intertwining operators and residues. {I}. {W}eighted
  characters}, J. Funct. Anal. \textbf{84} (1989), no.~1, 19--84. \MR{999488}

\bibitem[Art91]{ArthurLocalTF}
\bysame, \emph{A local trace formula}, Inst. Hautes \'{E}tudes Sci. Publ. Math.
  (1991), no.~73, 5--96. \MR{1114210}

\bibitem[Art93]{ArthurElliptic}
\bysame, \emph{On elliptic tempered characters}, Acta Math. \textbf{171}
  (1993), no.~1, 73--138. \MR{1237898}

\bibitem[Art94a]{ArthurFourier}
\bysame, \emph{On the {F}ourier transforms of weighted orbital integrals}, J.
  Reine Angew. Math. \textbf{452} (1994), 163--217. \MR{1282200}

\bibitem[Art94b]{ArthurPW}
\bysame, \emph{The trace {P}aley {W}iener theorem for {S}chwartz functions},
  Representation theory and analysis on homogeneous spaces ({N}ew {B}runswick,
  {NJ}, 1993), Contemp. Math., vol. 177, Amer. Math. Soc., Providence, RI,
  1994, pp.~171--180. \MR{1303605}

\bibitem[Art96]{ArthurRelations}
\bysame, \emph{On local character relations}, Selecta Math. (N.S.) \textbf{2}
  (1996), no.~4, 501--579. \MR{1443184}

\bibitem[Art08]{ArthurProblemsReal}
\bysame, \emph{Problems for real groups}, Representation theory of real
  reductive {L}ie groups, Contemp. Math., vol. 472, Amer. Math. Soc.,
  Providence, RI, 2008, pp.~39--62. \MR{2478455}

\bibitem[Art13]{ArthurEndo}
\bysame, \emph{The endoscopic classification of representations}, American
  Mathematical Society Colloquium Publications, vol.~61, American Mathematical
  Society, Providence, RI, 2013, Orthogonal and symplectic groups. \MR{3135650}

\bibitem[Art17]{ArthurProblemsBE}
\bysame, \emph{Problems beyond endoscopy}, Representation theory, number
  theory, and invariant theory, Progr. Math., vol. 323, Birkh\"auser /
  Springer, Cham, 2017, pp.~23--45. \MR{3753907}

\bibitem[AV16]{AdamsVogan}
Jeffrey Adams and David~A. Vogan, Jr., \emph{Contragredient representations and
  characterizing the local {L}anglands correspondence}, Amer. J. Math.
  \textbf{138} (2016), no.~3, 657--682. \MR{3506381}

\bibitem[Bor79]{Borel}
A.~Borel, \emph{Automorphic {$L$}-functions}, Automorphic forms,
  representations and {$L$}-functions ({P}roc. {S}ympos. {P}ure {M}ath.,
  {O}regon {S}tate {U}niv., {C}orvallis, {O}re., 1977), {P}art 2, Proc. Sympos.
  Pure Math., vol. XXXIII, Amer. Math. Soc., Providence, RI, 1979, pp.~27--61.
  \MR{546608}

\bibitem[Bou94]{Bouaziz}
Abderrazak Bouaziz, \emph{Int\'{e}grales orbitales sur les groupes de {L}ie
  r\'{e}ductifs}, Ann. Sci. \'{E}cole Norm. Sup. (4) \textbf{27} (1994), no.~5,
  573--609. \MR{1296557}

\bibitem[CFM{\etalchar{+}}22]{CunninghamEtAl}
Clifton L.~R. Cunningham, Andrew Fiori, Ahmed Moussaoui, James Mracek, and Bin
  Xu, \emph{Arthur packets for {$p$}-adic groups by way of microlocal vanishing
  cycles of perverse sheaves, with examples}, Mem. Amer. Math. Soc.
  \textbf{276} (2022), no.~1353, ix+216. \MR{4391878}

\bibitem[DGG22]{DalalGG}
Rahul Dalal and Mathilde Gerbelli-Gauthier, \emph{Statistics of cohomological
  automorphic representations on unitary groups via the endoscopic
  classification}, 2022, Preprint.

\bibitem[DM94]{DigneMichel}
Fran\c{c}ois Digne and Jean Michel, \emph{Groupes r\'{e}ductifs non connexes},
  Ann. Sci. \'{E}cole Norm. Sup. (4) \textbf{27} (1994), no.~3, 345--406.
  \MR{1272294}

\bibitem[GG63]{GelfandGraev}
I.~M. Gel'fand and M.~I. Graev, \emph{Representations of the group of
  second-order matrices with elements in a locally compact field and special
  functions on locally compact fields}, Uspehi Mat. Nauk \textbf{18} (1963),
  no.~4(112), 29--99. \MR{155931}

\bibitem[GKM97]{GKM}
M.~Goresky, R.~Kottwitz, and R.~MacPherson, \emph{Discrete series characters
  and the {L}efschetz formula for {H}ecke operators}, Duke Math. J. \textbf{89}
  (1997), no.~3, 477--554. \MR{1470341}

\bibitem[Hai14]{Haines}
Thomas~J. Haines, \emph{The stable {B}ernstein center and test functions for
  {S}himura varieties}, Automorphic forms and {G}alois representations. {V}ol.
  2, London Math. Soc. Lecture Note Ser., vol. 415, Cambridge Univ. Press,
  Cambridge, 2014, pp.~118--186. \MR{3444233}

\bibitem[HC76]{HCHAIII}
Harish-Chandra, \emph{Harmonic analysis on real reductive groups. {III}. {T}he
  {M}aass-{S}elberg relations and the {P}lancherel formula}, Ann. of Math. (2)
  \textbf{104} (1976), no.~1, 117--201. \MR{439994}

\bibitem[HC99]{HCAdmissible}
\bysame, \emph{Admissible invariant distributions on reductive {$p$}-adic
  groups}, University Lecture Series, vol.~16, American Mathematical Society,
  Providence, RI, 1999, With a preface and notes by Stephen DeBacker and Paul
  J. Sally, Jr. \MR{1702257}

\bibitem[Hen00]{Henniart}
Guy Henniart, \emph{Une preuve simple des conjectures de {L}anglands pour
  {${\rm GL}(n)$} sur un corps {$p$}-adique}, Invent. Math. \textbf{139}
  (2000), no.~2, 439--455. \MR{1738446}

\bibitem[HT01]{HarTay}
Michael Harris and Richard Taylor, \emph{The geometry and cohomology of some
  simple {S}himura varieties}, Annals of Mathematics Studies, vol. 151,
  Princeton University Press, Princeton, NJ, 2001, With an appendix by Vladimir
  G. Berkovich. \MR{1876802}

\bibitem[JL20]{JohnstoneLuo}
Daniel Johnstone and Zhilin Luo, \emph{On the stable transfer for
  $\operatorname{Sym}^n$ lifting of $\operatorname{GL}_2$}, 2020, Preprint.

\bibitem[Joh17]{JohnstoneThesis}
Daniel~Lentine Johnstone, \emph{A {G}elfand-{G}raev {F}ormula and {S}table
  {T}ransfer {F}actors for ${SL}_2({F})$}, ProQuest LLC, Ann Arbor, MI, 2017,
  Thesis (Ph.D.)--The University of Chicago. \MR{3705993}

\bibitem[Joh22]{JohnstoneArticle}
Daniel Johnstone, \emph{A gelfand--graev formula and stable transfer for
  $\operatorname{SL}_\ell({F})$ and $\operatorname{GL}_\ell({F})$ in the
  unramified case, $\ell$ an odd prime}, 2022, Preprint.

\bibitem[Kal15]{KalethaEpipelagic}
Tasho Kaletha, \emph{Epipelagic {$L$}-packets and rectifying characters},
  Invent. Math. \textbf{202} (2015), no.~1, 1--89. \MR{3402796}

\bibitem[Kal16]{KalethaRefined}
\bysame, \emph{The local {L}anglands conjectures for non-quasi-split groups},
  Families of automorphic forms and the trace formula, Simons Symp., Springer,
  [Cham], 2016, pp.~217--257. \MR{3675168}

\bibitem[Kal21]{KalDouble}
\bysame, \emph{On $l$-embeddings and double covers of tori over local fields},
  2021, Preprint.

\bibitem[Kal23]{KalICM}
\bysame, \emph{Representations of reductive groups over local fields},
  I{CM}---{I}nternational {C}ongress of {M}athematicians. {V}ol. 4. {S}ections
  5--8, EMS Press, Berlin, [2023] \copyright2023, pp.~2948--2975. \MR{4680348}

\bibitem[Kaz86]{KazCusp}
David Kazhdan, \emph{Cuspidal geometry of {$p$}-adic groups}, J. Analyse Math.
  \textbf{47} (1986), 1--36. \MR{874042}

\bibitem[Kna76]{KnaComm}
A.~W. Knapp, \emph{{Commutativity of intertwining operators. II}}, Bulletin of
  the American Mathematical Society \textbf{82} (1976), no.~2, 271 -- 273.

\bibitem[Kot82]{KotConj}
Robert~E. Kottwitz, \emph{Rational conjugacy classes in reductive groups}, Duke
  Math. J. \textbf{49} (1982), no.~4, 785--806. \MR{683003}

\bibitem[Kot84]{KotCuspidal}
\bysame, \emph{Stable trace formula: cuspidal tempered terms}, Duke Math. J.
  \textbf{51} (1984), no.~3, 611--650. \MR{757954}

\bibitem[Kot97]{KotIsoII}
\bysame, \emph{Isocrystals with additional structure. {II}}, Compositio Math.
  \textbf{109} (1997), no.~3, 255--339. \MR{1485921}

\bibitem[Kot05]{KottwitzHarmonic}
\bysame, \emph{Harmonic analysis on reductive {$p$}-adic groups and {L}ie
  algebras}, Harmonic analysis, the trace formula, and {S}himura varieties,
  Clay Math. Proc., vol.~4, Amer. Math. Soc., Providence, RI, 2005,
  pp.~393--522. \MR{2192014}

\bibitem[KP23]{KalPra}
Tasho Kaletha and Gopal Prasad, \emph{Bruhat-{T}its theory---a new approach},
  New Mathematical Monographs, vol.~44, Cambridge University Press, Cambridge,
  2023. \MR{4520154}

\bibitem[KS75]{KnappSteinSingularIV}
A.~W. Knapp and E.~M. Stein, \emph{Singular integrals and the principal series.
  {IV}}, Proc. Nat. Acad. Sci. U.S.A. \textbf{72} (1975), 2459--2461.
  \MR{376964}

\bibitem[KS80]{KnappSteinInterII}
\bysame, \emph{Intertwining operators for semisimple groups. {II}}, Invent.
  Math. \textbf{60} (1980), no.~1, 9--84. \MR{582703}

\bibitem[KS99]{KS}
Robert~E. Kottwitz and Diana Shelstad, \emph{Foundations of twisted endoscopy},
  Ast\'{e}risque (1999), no.~255, vi+190. \MR{1687096}

\bibitem[KT]{KalTai}
Tasho Kaletha and Olivier Ta\"ibi, \emph{The local langlands conjecture},
  Preprint.

\bibitem[Lab85]{LabesseTori}
J.-P. Labesse, \emph{Cohomologie, {$L$}-groupes et fonctorialit\'{e}},
  Compositio Math. \textbf{55} (1985), no.~2, 163--184. \MR{795713}

\bibitem[Lan89]{LanReal}
R.~P. Langlands, \emph{On the classification of irreducible representations of
  real algebraic groups}, Representation theory and harmonic analysis on
  semisimple {L}ie groups, Math. Surveys Monogr., vol.~31, Amer. Math. Soc.,
  Providence, RI, 1989, pp.~101--170. \MR{1011897}

\bibitem[Lan97]{LanAbelian}
\bysame, \emph{Representations of abelian algebraic groups}, 1997, Olga
  Taussky-Todd: in memoriam, pp.~231--250. \MR{1610871}

\bibitem[Lan00]{LanEB}
Robert~P. Langlands, \emph{Endoscopy and beyond}, 2000, Notes for a lecture
  given at the Institute for Advanced Study on March 30, 2000.

\bibitem[Lan04]{LanBE}
\bysame, \emph{Beyond endoscopy}, Contributions to automorphic forms, geometry,
  and number theory, Johns Hopkins Univ. Press, Baltimore, MD, 2004,
  pp.~611--697. \MR{2058622}

\bibitem[Lan07]{LanNouveau}
\bysame, \emph{Un nouveau point de rep\`ere dans la th\'eorie des formes
  automorphes}, Canad. Math. Bull. \textbf{50} (2007), no.~2, 243--267.
  \MR{2317447}

\bibitem[Lan13]{LanST}
\bysame, \emph{Singularit\'{e}s et transfert}, Ann. Math. Qu\'{e}. \textbf{37}
  (2013), no.~2, 173--253. \MR{3117742}

\bibitem[Lat21]{Latham}
Peter Latham, \emph{Typical representations, parabolic induction and the
  inertial local langlands correspondence}, 2021, Preprint.

\bibitem[LM15]{LapidMao}
Erez Lapid and Zhengyu Mao, \emph{A conjecture on {W}hittaker-{F}ourier
  coefficients of cusp forms}, J. Number Theory \textbf{146} (2015), 448--505.
  \MR{3267120}

\bibitem[Mg14]{MoegEndo}
Colette M\oe~glin, \emph{Paquets stables des s\'eries discr\`etes accessibles
  par endoscopie tordue; leur param\`etre de {L}anglands}, Automorphic forms
  and related geometry: assessing the legacy of {I}. {I}.
  {P}iatetski-{S}hapiro, Contemp. Math., vol. 614, Amer. Math. Soc.,
  Providence, RI, 2014, pp.~295--336. \MR{3220932}

\bibitem[Mok15]{MokEndo}
Chung~Pang Mok, \emph{Endoscopic classification of representations of
  quasi-split unitary groups}, Mem. Amer. Math. Soc. \textbf{235} (2015),
  no.~1108, vi+248. \MR{3338302}

\bibitem[MW16a]{MWI}
Colette M{\oe}glin and Jean-Loup Waldspurger, \emph{Stabilisation de la formule
  des traces tordue. {V}ol. 1}, Progress in Mathematics, vol. 316,
  Birkh\"{a}user / Springer, Cham, 2016. \MR{3823813}

\bibitem[MW16b]{MWII}
Colette Moeglin and Jean-Loup Waldspurger, \emph{Stabilisation de la formule
  des traces tordue. {V}ol. 2}, Progress in Mathematics, vol. 317,
  Birkh\"{a}user / Springer, Cham, 2016. \MR{3823814}

\bibitem[MW18]{MWLocal}
Colette M{\oe}glin and J.-L. Waldspurger, \emph{La formule des traces locale
  tordue}, Mem. Amer. Math. Soc. \textbf{251} (2018), no.~1198, v+183.
  \MR{3743601}

\bibitem[Ng{\^o}20]{Ngo}
B{\'a}o~Ch{\^a}u Ng{\^o}, \emph{Hankel transform, {L}anglands functoriality and
  functional equation of automorphic {$L$}-functions}, Jpn. J. Math.
  \textbf{15} (2020), no.~1, 121--167. \MR{4068833}

\bibitem[Rap05]{Rapoport}
Michael Rapoport, \emph{A guide to the reduction modulo {$p$} of {S}himura
  varieties}, no. 298, 2005, Automorphic forms. I, pp.~271--318. \MR{2141705}

\bibitem[Sak13]{SakBEI}
Yiannis Sakellaridis, \emph{Beyond endoscopy for the relative trace formula
  {I}: {L}ocal theory}, Automorphic representations and {$L$}-functions, Tata
  Inst. Fundam. Res. Stud. Math., vol.~22, Tata Inst. Fund. Res., Mumbai, 2013,
  pp.~521--590. \MR{3156863}

\bibitem[Sak19a]{SakBEII}
\bysame, \emph{Beyond endoscopy for the relative trace formula {II}: {G}lobal
  theory}, J. Inst. Math. Jussieu \textbf{18} (2019), no.~2, 347--447.
  \MR{3915291}

\bibitem[Sak19b]{SakFuncViaTrace}
\bysame, \emph{Relative functoriality and functional equations via trace
  formulas}, Acta Math. Vietnam. \textbf{44} (2019), no.~2, 351--389.
  \MR{3947974}

\bibitem[Sak22a]{SakTransferHankel1}
\bysame, \emph{Transfer operators and {H}ankel transforms between relative
  trace formulas, {I}: {C}haracter theory}, Adv. Math. \textbf{394} (2022),
  Paper No. 108010, 75. \MR{4355722}

\bibitem[Sak22b]{SakTransferHankel2}
\bysame, \emph{Transfer operators and {H}ankel transforms between relative
  trace formulas, {II}: {R}ankin-{S}elberg theory}, Adv. Math. \textbf{394}
  (2022), Paper No. 108039, 104. \MR{4355733}

\bibitem[Sak23]{SakIHES}
\bysame, \emph{Local and global questions ``beyond endoscopy''}, 2023,
  Preprint.

\bibitem[Sch13]{Scholze}
Peter Scholze, \emph{The local {L}anglands correspondence for {$\GL_n$} over
  {$p$}-adic fields}, Invent. Math. \textbf{192} (2013), no.~3, 663--715.
  \MR{3049932}

\bibitem[She79]{She79}
D.~Shelstad, \emph{Characters and inner forms of a quasi-split group over
  {${\bf R}$}}, Compositio Math. \textbf{39} (1979), no.~1, 11--45. \MR{539000}

\bibitem[She08]{SheTERGI}
\bysame, \emph{Tempered endoscopy for real groups. {I}. {G}eometric transfer
  with canonical factors}, Representation theory of real reductive {L}ie
  groups, Contemp. Math., vol. 472, Amer. Math. Soc., Providence, RI, 2008,
  pp.~215--246. \MR{2454336}

\bibitem[Sil78]{SilbergerDimension}
Allan~J. Silberger, \emph{The {K}napp-{S}tein dimension theorem for {$p$}-adic
  groups}, Proc. Amer. Math. Soc. \textbf{68} (1978), no.~2, 243--246.
  \MR{492091}

\bibitem[Sil79]{Silberger}
\bysame, \emph{Introduction to harmonic analysis on reductive {$p$}-adic
  groups}, Mathematical Notes, vol.~23, Princeton University Press, Princeton,
  NJ; University of Tokyo Press, Tokyo, 1979, Based on lectures by
  Harish-Chandra at the Institute for Advanced Study, 1971--1973. \MR{544991}

\bibitem[Ste68]{SteinbergEndom}
Robert Steinberg, \emph{Endomorphisms of linear algebraic groups}, Memoirs of
  the American Mathematical Society, vol. No. 80, American Mathematical
  Society, Providence, RI, 1968. \MR{230728}

\bibitem[SV80]{SpeVog}
Birgit Speh and David~A. Vogan, Jr., \emph{Reducibility of generalized
  principal series representations}, Acta Math. \textbf{145} (1980), no.~3-4,
  227--299. \MR{590291}

\bibitem[SZ18]{SilZin}
Allan~J. Silberger and Ernst-Wilhelm Zink, \emph{Langlands classification for
  {$L$}-parameters}, J. Algebra \textbf{511} (2018), 299--357. \MR{3834776}

\bibitem[Tho20]{thomas2020towards}
John~Owen Thomas, \emph{Towards stable-stable transfer involving symplectic
  groups}, Ph.D. thesis, Rutgers University-Graduate School-Newark, 2020.

\bibitem[Var24]{varma}
Sandeep Varma, \emph{Some comments on the stable bernstein center}, 2024,
  Preprint.

\bibitem[Vig94]{vigneras}
Marie-France Vign\'eras, \emph{Homologie cyclique, principe de {S}elberg et
  pseudo-coefficient}, Invent. Math. \textbf{116} (1994), no.~1-3, 651--676.
  \MR{1253209}

\bibitem[Vin96]{vinberg}
E.~B. Vinberg, \emph{On invariants of a set of matrices}, J. Lie Theory
  \textbf{6} (1996), no.~2, 249--269. \MR{1424635}

\bibitem[Vog87]{VoganUnitary}
David~A. Vogan, Jr., \emph{Unitary representations of reductive {L}ie groups},
  Annals of Mathematics Studies, vol. 118, Princeton University Press,
  Princeton, NJ, 1987. \MR{908078}

\bibitem[Vog93]{VoganLLC}
\bysame, \emph{The local {L}anglands conjecture}, Representation theory of
  groups and algebras, Contemp. Math., vol. 145, Amer. Math. Soc., Providence,
  RI, 1993, pp.~305--379. \MR{1216197}

\end{thebibliography}

\end{document}